\documentclass[12pt,a4paper]{article}
\usepackage{fullpage}
\usepackage{amsmath,amsfonts,amsthm,amssymb,graphicx,xcolor}  
\newtheorem{theorem}{Theorem}
\newtheorem*{theorem*}{Theorem}
\newtheorem*{proposition*}{Proposition}
\newtheorem{lemma}[theorem]{Lemma}

\newtheorem{definition}[theorem]{Definition}
\newtheorem{proposition}[theorem]{Proposition}
\newtheorem{remark}[theorem]{Remark}
\usepackage[utf8]{inputenc}
\usepackage{color}
\usepackage[ocgcolorlinks, linkcolor=black]{hyperref}
\usepackage[mathscr]{euscript}
\usepackage{enumerate}

\usepackage{mathtools}

\usepackage{cancel}

\renewcommand{\S}{\mathbb{S}}
\newcommand{\R}{\mathbb{R}}

\newcommand{\eps}{\varepsilon}

\newcommand{\N}{\mathbb{N}}
\newcommand{\s}{\hspace{0.5pt}}
\newcommand{\supp}{\mathop{\rm supp}}
\newcommand{\p}{\partial}

\newcommand{\norm}[1]{\left\Vert #1 \right\Vert}

\newcommand{\abs}[1]{\lvert #1 \rvert}

\newcommand{\ap}[1]{a^{(#1)}}
\newcommand{\app}[2]{#1^{(#2)}}

\renewcommand{\Im}{\operatorname{Im}}
\newcommand{\C}{\mathbb{C}}

\numberwithin{equation}{section}
\numberwithin{theorem}{section}

\usepackage{authblk}
\begin{document}
\title{Gaussian beam interactions and inverse source problems for nonlinear wave equations}

\author[1]{Matti Lassas}

\author[1,2]{Tony Liimatainen}

\author[3]{Valter Pohjola}

\author[3]{Teemu Tyni}

\affil[1]{Department of Mathematics and Statistics, University of Helsinki, Finland}
\affil[2]{Department of Mathematics and Statistics, University of Jyväskylä, Finland}
\affil[3]{Research Unit of Applied and Computational Mathematics, University of Oulu, Finland}

\date{} 

\maketitle

\begin{abstract}
We study the inverse source problem for the semilinear wave equation
\[
(\Box_g + q_1)u + q_2 u^2 = F,
\]
on a globally hyperbolic Lorentzian manifold. We demonstrate that the coefficients $q_1$ and $q_2$, as well as the source term $F$, can be recovered up to a natural gauge symmetry inherent in the problem from local measurements. Furthermore, if $q_1$ is known, we establish the \emph{unique} recovery of the source $F$, which is in a striking contrast to inverse source problems for linear equations where unique recovery is not possible. Our results also generalize previous works by eliminating the assumption that $u= 0$ is a solution, and by accommodating quadratic nonlinearities.

A key contribution is the development of a calculus for nonlinear interactions of Gaussian beams. This framework provides an explicit representation for waves that correspond to sources involving products of two or more Gaussian beams. We anticipate this calculus will serve as a versatile tool in related problems, offering a concrete alternative to Fourier integral operator methods.
	 
		\medskip
		
		
        
	\end{abstract}

\tableofcontents

\section{Introduction} \label{sec_intro}
Let $(N,g)$ be a globally hyperbolic Lorentzian manifold of dimension $n+1$. By the global hyperbolicity, $N$ is isometric to $\mathbb{R} \times M$, where $M$ is a $n$-dimensional manifold, and its metric $g$ takes the form
\[
g(x,t)= -\beta(x,t)dt^2 + h_t(x),
\]
where $\beta>0$ is a smooth function on $N$ and $h_t$ is a smooth one-parameter family of Riemannian metrics on $M$, see~\cite{BS05time-splitting}. We consider an inverse problem for the nonlinear wave equation
\begin{equation}\label{eq_wave_equation}
\begin{cases}
\square_g u + q_1 u + q_2 u^2 =  F  +  f & \text{ in } [0,T]\times M,\\
u(0,x') = \phi_1(x'), \quad \p_t  u(0,x')= \phi_2(x')  &\text{ for } x'\in M,
\end{cases}
\end{equation}
where 
$\phi_1$ and $\phi_2$ are smooth functions on $M$, 
and $q_1$, $q_2$ and $F$ are smooth functions on $N$. The function $f$ is a locally supported source function, which will be controlled in the inverse problem. 
In local  coordinates, we write $x=(x_0,x')=(x^0,x^1,\ldots,x^n)$, $x^0=t$, and on $\R\times M$ the Lorentzian wave operator $\square_g$ of $g$ can be written as
\begin{equation}\label{eq:metric_local}
\square_g u = -\sum_{a,b=0}^n\frac{1}{\sqrt{|\det(g)|}}\frac{\p}{\p x^a}\left(
\sqrt{|\det(g)|}g^{ab} \frac{\partial u}{\partial x^b}
\right)=-g^{ab}\p_{ab}u +\Gamma^a\p_au.
\end{equation}
Here we denote $(g^{-1})_{ab}= (g^{ab})$ as usual.

Let $0<T_1<T_2<T$ and let us consider $U = [T_1,T_2] \times U' \subset [0,T]\times M$, where $U' \subset M$ is a smooth open relatively compact subset.
We assume that 
$$
D(U)\subset [0,T]\times\Omega,
$$
where $D(U)$ is the causal diamond of the set $U$, i.e. the intersection of the causal past and future of $U$ (see \eqref{eq_def_D}). We postpone the definitions of Lorentzian geometry to Section~\ref{sec: prelim}.

	Assume for now that the problem \eqref{eq_wave_equation} is well-posed in the energy space $E^{s+1}$ (see Appendix~\ref{sec_wellposed} for definitions) for sufficiently small sources $f$ and fixed Cauchy data $(\phi_1,\phi_2)$. Under this assumption, we define the \emph{source-to-solution map} $S$ as the mapping
\begin{equation}\label{eq:DNmap}
S : f \mapsto u_f|_U, \quad S: \mathcal{C} \subset E_c^s(U) \to E^{s+1}(U),
\end{equation}
where $u_f$ denotes the unique solution to \eqref{eq_wave_equation} corresponding to the source $f$. The well-posedness is guaranteed, when there exists a solution to the homogeneous problem ($f=0$), which is true, in particular, for sufficiently small $F$ and Cauchy data $(\phi_1,\phi_2)$ (see Appendix \ref{sec_wellposed} for details).

	Consider now two sets of coefficients $(q_1,q_2,F)$ and $(\tilde{q}_1,\tilde{q}_2,\tilde{F})$ with fixed Cauchy data $(\phi_1,\phi_2)$. Let $S$ and $\widetilde{S}$ denote the corresponding source-to-solution maps defined on neighborhoods $\mathcal{C}$ and $\widetilde{\mathcal{C}}$ of zero in $E^s(U)$, respectively. We say that $S = \widetilde{S}$ if the equality $S(f) = \widetilde{S}(f)$ holds for all $f$ in the non-empty open set $\mathcal{C} \cap \widetilde{\mathcal{C}}$. 
	The inverse source problem we address is the following:

	\begin{itemize}
		\item   \textbf{Inverse source problem:} Does $S= \widetilde S$ imply that $(q_1,q_2,F)=(\tilde q_1,\tilde q_2,\tilde F)$? 
	\end{itemize}
	
A natural limitation for the answer is due to the finite speed of propagation of waves, but the problem also has an inherent gauge symmetry. To illustrate this symmetry, consider first the linear case $q_2=0$ and denote $q_1=q$. Let $u$ be a solution to the equation \eqref{eq_wave_equation} in this case, and $\varphi$ an arbitrary $C^2$-function satisfying 
		\begin{equation}\label{eq:varphi_conditions}
		\varphi=0 \text{ in } U \quad  \text{ and } \quad \varphi(0,x) = \p_t  \varphi(0,x)=0 \text{ for } x\in \Omega. 
		 \end{equation}
		Then a short computation shows that $\tilde u:=u+\varphi$ solves  
		\begin{align}\label{eq_r2}
			\begin{split}
				\square_g \tilde u+ q \tilde u =\, &F+\square_g \varphi +q\varphi+ f.
			\end{split}
		\end{align}
		This shows that while $u$ solves the equation $\square_g u + q u = F+f$,  $\tilde{u}$ solves the same  equation with source term $F + \square_g\varphi + q\varphi+f $. Since $\varphi$ vanishes on $U$ and has zero Cauchy data, it follows that both sources $F$ and $F + \square_g\varphi + q\varphi$ yield identical source-to-solution maps. This demonstrates an obstruction to unique source determination in the linear case.
		
		In the nonlinear case, the source-to-solution map exhibits a more involved gauge invariance. Two sets of coefficients $(q_1,q_2,F)$ and $(\tilde{q}_1,\tilde{q}_2,\tilde{F})$ will produce identical source-to-solution maps if they are related through:
		\begin{align}\label{gauge_intro}
		\begin{cases}
			F+\square_g \varphi+ \tilde q_1 \varphi+\tilde q_2 \varphi^2= \tilde F & \text{ in } D(U), \\
			q_1= \tilde q_1  + 2 \tilde q_2 \varphi & \text{ in } D(U),\\
			q_2=\tilde q_2& \text{ in } D(U),
		\end{cases}
	\end{align}
	where $\varphi$ is any $C^2$ function vanishing on the measurement set $U$ and having zero Cauchy data. An elementary proof of this fact is included in the Appendix \ref{sec:gauge_symmetry_appendix}.

	The nonlinearity of the equation is crucial for our methods. For this reason we assume from now on that
\begin{equation} \label{eq_sign_cond}
\begin{aligned}
D(U) \subset \supp(q_2).
\end{aligned}
\end{equation}
	
	Our main result shows that the gauge conditions \eqref{gauge_intro} are the only obstructions to uniqueness in the inverse problem:
\begin{theorem} \label{thm_thm1}
Let  $q_1,q_2 , \tilde q_1, \tilde q_2 \in C^\infty (N)$, with $q_2$ satisfying \eqref{eq_sign_cond} and $F,\tilde F \in E^s(N)$.
Assume that the corresponding source-to-solution maps $S$ and $\tilde S$ are well-defined for $f\in E^s(U)$ small enough.
Then 
\[
 S=\widetilde S
\]
implies that \eqref{gauge_intro} hold for $\varphi := \tilde u_0 - u_0$,  where $u_0$ and $\tilde u_0$ 
solve \eqref{eq_wave_equation} with $f=0$.
\end{theorem}

When the linear term $q_1$ is known a priori, we obtain a stronger result on the unique recovery of the coefficients:
\begin{theorem}(Unique recovery, when the linear term is known)\label{thm:unique_recovery} Assume as in Theorem \ref{thm_thm1} and adopt its notation. Assume also in addition that $q_1=\tilde q_1$ on a set $B\subset D(U)$. Then $S=\tilde S$ implies%
\[
 \quad q_2 = \tilde q_2, \quad F = \tilde F \quad \text{ in } B.
\]
\end{theorem}
The above follows from Theorem \ref{thm_thm1} by using $q_1=\tilde q_1$ in \eqref{gauge_intro} and the assumption \eqref{eq_sign_cond}, which together yield $\varphi\equiv 0$ in $B$.  
This contrasts sharply with the linear case ($q_2=0$), where the source-to-solution map remains invariant under the transformation $F \mapsto F + \square_g \varphi + q\varphi$ for any $\varphi \in C^2(N)$ satisfying \eqref{eq:varphi_conditions}, even with known $q$. Thus, nonlinearity provides a mechanism to break this gauge symmetry in the inverse problem.

When the source term $F$ is known a priori instead of $q_1$, we obtain another unique recovery result.

\begin{theorem}(Unique recovery when source is known)\label{thm:unique_recovery_known_source} Assume as in Theorem \ref{thm_thm1} and adopt its notation. Assume also in addition that $F=\tilde F$ in $J^-(U)$ and $q_1=\widetilde q_1$, $q_2=\widetilde q_2$ in $J^-(U)\setminus D(U)$. Then $S=\tilde S$ implies unique recovery:
		\begin{align}\label{unique of coeff}
			q_1=\tilde q_1, \quad q_2 = \tilde q_2 \quad \text{ in } D(U)
		\end{align}
	\end{theorem}
\begin{proof}
Note first that
\begin{equation}\label{eq: gauge vanishes}
(\square_g  + q_1 + q_2(u_0+\widetilde u_0))(u_0-\widetilde u_0) =
\square_g u_0 + q_1 u_0 + q_2 u_0^2 - (\square_g \widetilde u_0 + \widetilde q_1 \widetilde u_0 + \widetilde q_2 \widetilde u_0^2) = 0
\end{equation}
in $J^-(U)\setminus D(U)$.
Now the unique solution to the linear wave equation
\begin{equation}\label{eq: wave eq in corollary}
\begin{cases}
\square_g v + q_1v + q_2(u_0+\widetilde u_0)v=0,&x\in M_T,\\
v(0)=\p_t v(0)=0
\end{cases}
\end{equation}
is $v=0$. By \eqref{eq: gauge vanishes}, in the set $J^-(U)\setminus D(U)$
the function $\varphi=u_0-\widetilde u_0$ also satisfies \eqref{eq: wave eq in corollary}. By the domain of dependence of linear wave equations (see, for example, \cite[Lemma~12.8]{Ringstrom2009}, the solution $v$ to the linear equation \eqref{eq: wave eq in corollary} at $p\in M_T$ depends only on the values of $v$ in the past $J^-(p)$. Hence, by the uniqueness of the linear wave equation, $v=\varphi=0$ in $J^-(U)\setminus D(U)$. 
Since $F=\widetilde F$ in $J^-(U)$, by \eqref{gauge_intro} $\varphi$ satisfies
\[
(\square_g + \widetilde q_1) \varphi + \widetilde q_2\varphi^2=0,\quad\text{in } D(U)
\]
and by the above computations this equation is satisfied also in the larger set $J^-(U)\cap\{t>0\}$.
Moreover $\varphi(0)=\p_t\varphi(0)=0$.
By uniqueness of solutions to the nonlinear wave equation $\varphi=0$ in $J^-(U)\cap\{t>0\}$. By Theorem~\ref{thm_thm1} the coefficients are thus uniquely determined in $J^-(U)$.
\end{proof}

	The quadratic nonlinearity in our problem presents greater challenges compared to the cubic or higher nonlinearities often assumed in the literature. In fact, handling quadratic nonlinearity leads us to develop an \emph{interaction calculus for Gaussian beams} described in the next section.

\subsection{Gaussian beam interactions and sketch of the proof}\label{sec:WKB_sketch_intro}
Consider two Gaussian beam (see Section \ref{sec_GBs}) solutions $v$ %
to $(\square_g + Q)v = 0$ of the form
\begin{equation}\label{eq:guassian_beam_form_intro}
 v=e^{\Phi/h}(a_0+ha_1+\cdots+h^Ka_K)+r,
\end{equation}
where $\Phi$ is the phase function and $a=a_0+ha_1+\cdots+h^Ka_K$ is the amplitude constructed with respect to a lightlike geodesic. The function $r=r_h$ is the correction term.

The main new technique we introduce in this paper is an interaction calculus for Gaussian beams, or interaction calculus in short. In this method we consider solutions $w$ to the linear wave equation
\begin{equation}\label{eq:first_level_interactions}
 (\square_g + Q) w = q v^{(1)}\app{v}{2},
\end{equation}
where are $v^{(1)}$ and $\app{v}{2}$ are Gaussian beams of the form \eqref{eq:guassian_beam_form_intro} of the form and corresponding to intersecting lightlike geodesics, and $q$ is a smooth function. The solution $w$ can then be considered to correspond to a quadratic interaction of two Gaussian beams. We show that in this case the solution $w$ has the form
\begin{equation}\label{eq:form_of_w_intro}
w =   e^{i(\Phi_1+\Phi_2)/h}  \sum_{j = 2}^{2K+2} h^jc_j +\hat r,
\end{equation}
where $\Phi_1$ and $\Phi_2$ are the phase functions of $v^{(1)}$ and
$\app{v}{2}$, and the coefficients $c_j$ are determined recursively from their
respective amplitudes. 
Here also $\hat r=\hat r_h$ is a correction term that becomes small in the parameter $h>0$,
provided that the geodesics of the Gaussian beams have a single intersection point (the general case
of several intersection points necessitates some modifications and is discussed briefly below).
We also consider more general interactions such as 
\begin{equation}\label{eq:second_level_interactions}
 (\square_g + Q) \omega = q vw,
\end{equation}
where $v$ is a Gaussian beam solution and $w$ as in \eqref{eq:first_level_interactions}. These represent iterated interactions of waves, since the source term itself involves the solution $w$ corresponding to interaction of $v^{(1)}$ and $\app{v}{2}$. The solution $\omega$ has a form similar to \eqref{eq:form_of_w_intro}.

We use the interaction calculus to prove Theorem \ref{thm_thm1}. We linearize the nonlinear equation \eqref{eq_wave_equation} up to four parameters in the data. First and second linerization yield problems that are, in general, not known to be solvable (an inverse problem for a linear wave equation and a density problem for products of three solutions). For this reason we proceed to the third linearization, which produces the equation
\begin{equation} \label{eq_wijk}
\begin{aligned}
\square_g \omega + (q_1 + 2 q_2 u_0)  \omega  = -  6q_2 vw,
\end{aligned}
\end{equation}
where $v$, as well as $\app{v}{0}$ below, are Gaussian beams, and $w$ is as above with $Q=q_1 + 2 q_2 u_0$. 
Together with the knowledge of the source-to-solution map, it follows that a sum of %
integrals of the form 
\begin{equation}\label{eq:known_integral}
 \int q_2\app{v}{0}v w.
\end{equation}
is known. The interaction calculus especially shows that the phase of $w$ is equal to the sum of $\Phi_1$ and $\Phi_2$ by \eqref{eq:form_of_w_intro}.  Thus, in \eqref{eq:known_integral} an exponential function of the sum of the phases of four Gaussian beams $\app{v}{0}$, $\app{v}{1}$, $\app{v}{2}$ and $v$ appears. The exponential function is thus an approximate a delta function and can be analyzed using stationary phase. This allows us to recover $q_2$, or more precisely, its square.

After recovering the square of $q_2$, we proceed to recover information about the linear term $q_1$. We observe that the integral \eqref{eq:known_integral} contains not only explicit information about $q_2$, but also implicit information about $q_1 + 2 q_2 u_0$ encoded in the subleading amplitude coefficients of the involved Gaussian beams. The interaction calculus further enables us to perform an asymptotic analysis in higher powers of the small parameter $h$ for integrals of the form \eqref{eq:known_integral}. By an appropriate choice of Gaussian beams, we may recover the subleading coefficient $c_3$ appearing in \eqref{eq:form_of_w_intro} through this calculus, which in turn can be used to solve for $q_1 + 2 q_2 u_0$. This is possible provided $q_2^2 \neq 0$,  as this quantity appears as a coefficient in the asymptotic analysis (see the proof of Lemma \ref{lem_a1}). We note that the approach for recovering the linear term through asymptotic
analysis of subleading amplitude coefficients in Gaussian beams parallels the
methodology in \cite{FO20} for stronger, cubic,  nonlinearities. 

To fully determine $q_2$
(rather than just its square), we employ fourth order linearization in Section 
\ref{sec_q2_only}, which involves analyzing higher order iterations of these
interactions. 
The case where the geodesics of Gaussian beams have multiple intersection points requires further analysis. To address this, we consider  decompositions such as 
$$
w = \hat w +  w_{N} + \hat r
$$
where $\hat w$ and $\hat r$ correspond to the terms in \eqref{eq:form_of_w_intro}. The additional term $w_N$ arises due to the presence of these extra intersection points and is not small in the parameter $h$. Nevertheless, we can show that the support of $w_N$ lies outside the set of interest, provided the geodesics are chosen appropriately. Section \ref{sec_geom_setup} is devoted to this analysis.

\subsubsection{Comparison to FIO calculus}

Fourier integral operator (FIO) calculus is highly effective for analyzing wave interactions and the propagation of their singularities. The WKB interaction calculus (or  interaction calculus in short) presented in this paper offers a powerful alternative in many applications. Within the FIO framework, interacting solutions are treated as conormal distributions and analyzed using intersecting Lagrangian manifolds.
In contrast, the interaction calculus developed here describes the nonlinear interactions of Gaussian beam solutions, which are smooth or have limited regularity, depending on the coefficients of the corresponding wave equation. A key advantage of the  interaction calculus is that it yields an explicit asymptotic expansion, see \eqref{eq:first_level_interactions}, for the interacting solution to any order in the asymptotic parameter. This is analogous to computing higher order symbol coefficients in FIO calculus, as demonstrated in applications to inverse problems in~\cite{chen2025stable}. Moreover, the interaction calculus is elementary to apply and naturally supports iterated interactions, as shown in Section \ref{sec_high_WKB_interaction}.

The use of Gaussian beams, rather than the plane wave type distributional solutions common in FIOs, may enable novel applications in inverse problems. Although not explicitly tracked here, the finite regularity requirements of the interaction calculus naturally extend its applicability beyond the $C^\infty$ category typically assumed in FIO techniques. Furthermore, Gaussian beams arise in distinct physical contexts compared to plane waves. For instance, they are frequently used to model laser beams. Consequently, the interaction calculus can model the propagation of interacting lasers in nonlinear media, potentially enabling studies of  new types of inverse problems.


\subsection{Motivation and related literature}
The motivation for this work is twofold. First, we demonstrate how nonlinearity can be leveraged to solve hyperbolic inverse source problems. Second, we introduce the interaction calculus for Gaussian beams detailed above, which is a novel tool for handling quadratic nonlinearities and analysing interactions of waves in general. 

Inverse problems for nonlinear hyperbolic equations have attracted considerable interest due to the remarkable property that nonlinear wave interactions can serve as a powerful tool for solving such problems. This was first observed in a striking result by Kurylev, Lassas and Uhlmann~\cite{KLU18}, where they proved that local measurements for the scalar wave equation with a quadratic nonlinearity determine the conformal class of a globally hyperbolic four-dimensional Lorentzian manifold. Their approach, now known as the \emph{higher-order linearization method}, originally appeared in the context of parabolic equations~\cite{isakov1993uniqueness} and has since become a fundamental technique in inverse problems.

Following~\cite{KLU18}, research on inverse problems for nonlinear hyperbolic equations has expanded rapidly. Subsequent works, including~\cite{chen2021detection,kurylev2022inverse,lassas2018inverse,uhlmann2020determination}, have explored inverse problems for connections, Einstein equations, and more general nonlinearities. The higher-order linearization method was later adapted to elliptic equations in~\cite{feizmohammadi2020inverse,LLLS21} for full data and in~\cite{krupchyk2020remark,lassas2020partial} for partial data. For boundary value problems in the hyperbolic setting, the method was further developed in~\cite{lassas2022uniqueness, hintz2022inverse}. 
We also mention the recent works regarding inverse problems for nonlinear hyperbolic equations \cite{wang2019inverse, de2019nonlinear, de2020nonlinear, hintz2022dirichlet, lassas2025stability, balehowsky2022inverse, lai2021reconstruction, lassas2024inverse, kian2021determination, sa2022recovery, tzou2023determining}.

Among these, the most closely related to our work are the following. The study~\cite{FO20} considered the recovery of the zeroth-order coefficient $V$ in the equation $(\square_g + V)u + u^3 = 0$ by analysing subprincipal symbols of Gaussian beams. 
Our Theorem~\ref{thm:unique_recovery} generalizes their result to quadratic nonlinearities, unknown coefficients, and the presence of a source term. 
Another closely related and very recent work is~\cite{QXT25}, where a similar source determination problem is studied
in Euclidean space using  Dirichlet-to-Neumann maps. The work \cite{liimatainen25_mean_curvature} studied the inverse source problem for the mean curvature equation. 
The work~\cite{lin2024determining} addressed the recovery of Cauchy data and the nonlinearity $a(x,u)$ in the wave equation $\square u + a(x,u) = 0$ from the Dirichlet-to-Neumann map in Minkowski space. We also mention the related work for parabolic equations \cite{lin2022simultaneous}. The recent stability study \cite{chen2025stable} in Minkowski space  revisited~\cite{FO20} by employing Fourier integral operator (FIO) calculus for intersecting Lagrangians. The authors used FIO calculus to compute subprincipal symbols for qubic interactions of distorted plane waves. In contrast, we use our interaction calculus  to expand solutions to interactions of Gaussian beams \eqref{eq:first_level_interactions}, and also to iterated interactions such as \eqref{eq:second_level_interactions},  to subleading order. 

The calculus is motivated by the calculus introduced in the elliptic setting for complex geometric optics solutions \cite{FLL23}, where the recovery of $V$ of the equation $\Delta_gu+Vu^2=0$ was considered.  The calculus there is applicable on cylindrical (CTA) geometries, which is a natural restriction. 
The interaction calculus for Gaussian beams on globally hyperbolic Lorentzian manifolds, however, is  applicable in general. Unlike the current work, the work \cite{FLL23} assumes that the corresponding geodesics only intersect once.

As established in Theorem~\ref{thm:unique_recovery}, the presence of nonlinearity can render inverse source problems uniquely solvable. This phenomenon was first observed for elliptic equations in~\cite{LL24} and later extended to parabolic settings in~\cite{kian2024determining}. (We mention that, unlike the current work, both of these earlier results relied on the solvability of the first linearized problem.) The phenomenon suggests that artificially introducing nonlinearity into linear systems could enhance parameter reconstruction, particularly for source recovery. The concept of induced nonlinearity has earlier been explored for example in medical imaging using microbubbles~\cite{lindner2004microbubbles} and nonlinear optics~\cite{deka2017nonlinear}, as well as in inverse problems employing modulation instead of nonlinearity~\cite{li2021acousto}. Much like contrast agents in X-ray imaging, induced nonlinearity can possibly offer a powerful enhancement mechanism, enabling novel applications.

\section{Preliminaries}\label{sec: prelim}

We will now introduce some definitions and results needed in the subsequent sections.
We begin by discussing Lorentzian geometry and then discuss how to reduce the problem
to zero inital data. In the next subsection
\ref{sec_GBs} review the construction of Gaussian beams.

\medskip
\noindent
The manifold $N$ we consider is globally hyperbolic, and  we use the signature $(-,+,+,+)$ for the metric $g$. 
We will moreover denote the inner product given by the metric  by $ \langle \cdot , \cdot \rangle_g$.
A vector $\xi \in T_pN$ is called timelike if $ \langle \xi , \xi \rangle_g <0$.
It is called light like if $ \langle \xi , \xi \rangle_g = 0$, and spacelike 
if $ \langle \xi, \xi \rangle_g > 0$. 
We say that a curve $\tau$ in $N$ is  causal if its
tangent vector is timelike or lightlike for all points on the curve. 
Since  $N$ is globally hyperbolic and thus time orientable, we have a globally defined smooth timelike 
vector field $V$ on $N$.
A future pointing causal curve is a curve for which $ \langle \dot \tau , V  \rangle_g < 0$.
The existence of a  future pointing causal curve from $p$ to $q$, is denoted by $p \leq q$.
The future and past sets of a set $U \subset N$ are given by
$$
J^+(U) := \cup_{p \in U}\{ q \in N \,:\, p \leq q \}, \qquad 
J^-(U) := \cup_{p \in U}\{ q \in N \,:\, q \leq p \},
$$
respectively. The causal diamond of $D(U)$ of a set $U \subset N$, is given by
\begin{equation} \label{eq_def_D}
\begin{aligned}
D(U) := J^+(U)\cap J^-(U), 
\end{aligned}
\end{equation}
we will abbreviate $D(U)$ by $D$.
The wave operator $\Box_g$ in \eqref{eq_wave_equation}, is given  in a coordinate chart by 
$$
\Box_g u = - |g|^{-1/2} \p_i\big( |g|^{1/2} g^{ij} \p_j u  \big),
$$
where we use the Einstein summation over repeated indices $i,j=0,\ldots, n$. Here also $\abs{g}$ is the determinant of the metric.  

\medskip
\noindent
Let us recall what the standard  Fermi coordinates are. These are constructed
relative to a lightlike geodesic $\gamma:I\to \R$, where $I$ is a bounded interval in
$\R$, by inverting the map 
\begin{equation}\label{Fermi_coords_def}
 (s,y)\mapsto \text{exp}_{\gamma(s)}\big(\sum_{k=1}^{n}y^k\s e_k(s)\big) \in N.
\end{equation}
Here $e_k(s)$ are the parallel transports along $\gamma$ of the last $n$ vectors of a frame $\{e_0,e_1,\ldots, e_n\}$ of $T_{\gamma(0)}N$ with 
\[
 e_0=\dot\gamma(0).
\]
The other vectors of the frame are chosen so that for $j,k=2,\ldots n$, it holds
\begin{equation}\label{light_cone_frame}
 \langle e_0 , e_0 \rangle_g = 0, 
\quad \langle e_1 ,e_1 \rangle_g = 0, 
\quad \langle e_0 , e_1  \rangle_g = -2, 
\quad \langle e_j , e_k \rangle_g = \delta_{jk}. 
\end{equation}
Fermi coordinates always exist on a neighborhood of the graph of $\gamma$.
We record the properties  of Fermi coordinates in the following  (see e.g. \cite{FO20}).

\begin{definition}[Fermi coordinates]
Given a lightlike geodesic $\gamma : (a-\delta,b+\delta) \to N$, $a<b$, $\delta >0$, a Fermi coordinate chart $(\varphi, U)$
and the corresponding coordinates $(z^0,z^1,\dots,z^n)=(s,z')$, have the following properties:
\begin{enumerate}[(i)]
\item $\varphi (U) = I \times B(0 ,\delta')$, where $I := (a+\delta', b-\delta' )$ for some small $\delta' > 0$.

\item $\varphi(\gamma(s)) = (s,0,..,0)$.

\item The metric tensor satisfies 
\begin{equation} \label{eq_fermi_g}
\begin{aligned}
g |_\gamma = 2 ds \otimes dz^1 + \sum_{i=2}^n dz^i \otimes dz^i, \qquad \p_{i} g_{jk}|_\gamma = 0,\quad i,j,k=0,\dots,n.
\end{aligned}
\end{equation}

\end{enumerate}
\end{definition}

\noindent 
We will also later employ the asymptotic notation $f(z) \sim g(z)$, as $z \to 0$, which is defined as 
$$ 
f(z) \sim g(z) \qquad \Leftrightarrow \qquad \lim_{z \to a}\frac{f(z)}{g(z)} =  1.
$$
Note that if $f(z) \sim z^a$, as $z \to 0$, then $1/f(z) \sim z^{-a}$, as $z \to 0$.

\medskip
\noindent
We now make some simplifications that will be used in the consequent sections.
The next result shows that we can reduce proving  Theorem \ref{thm_thm1} 
to the case where  $\phi_1=\phi_2=0$. We shall therefore  
assume that $\phi_1=\phi_2=0$ in the subsequent sections.
For convenience we introduce the notation
$$
M_T :=  [0,T] \times M.
$$

\begin{lemma} \label{lem_reduction}
Assume that Theorem \ref{thm_thm1} holds in the case $\phi_1 = \phi_2 = 0$. 
Then Theorem \ref{thm_thm1} holds also for all $\phi_1, \phi_2 \in C^\infty(M)$. 
\end{lemma}

\begin{proof} 
Let $u_0$ and $\tilde u_0$ be solutions of \eqref{eq_wave_equation}, with coefficients $q_1,q_2,F$ and $\tilde q_1, \tilde q_2, \tilde F$,
and with the Cauchy data given by $\phi_1, \phi_2 \in C^\infty(M)$ and the source $f=0$. 
Suppose $v$  furthermore solves
\begin{equation}\label{eq_wave_equation_2}
\begin{cases}
\square_g v =   0,& x\in M_T\\
v(0,x') = \phi_1(x'),\quad \p_t  v(0,x') = \phi_2(x') , &\text{ on } \quad M,
\end{cases}
\end{equation}
Setting $w := u_0 - v$, gives us a solution to
\begin{equation}\label{eq_wave_equation_3}
\begin{cases}
\square_g w + q_1 w + q_2 w^2 =  - 2q_2u_0 v - q_1 v + q_2 v^2 + F +  \hat f,& x\in M_T\\
w(0,x') = 0,\quad \p_t  w(0,x')=0, &\text{ on } \quad M.
\end{cases}
\end{equation}
We define $\tilde w:= \tilde u_0 - v$ correspondingly.  
We know the source-to-solutions maps $S : \hat f \mapsto w$ and $\tilde S: \hat f \mapsto \tilde w$ from \eqref{eq:DNmap}, since we 
know $w$ and $\tilde w$ in $U$. 
We moreover know that $S = \tilde S$.
Because we assume that Theorem \ref{thm_thm1} holds for $\phi_1 = \phi_2 = 0$, we have that
\begin{align} \label{eq_potentials_equiv}
\begin{cases}
q_2 = \tilde q_2, \\
q_1 =  \tilde q_1 + 2 \tilde q_2 \varphi, \\
F - 2q_2 u_0 v - q_1 v + q_2 v^2  = \tilde F - 2q_2\tilde u_0 v - \tilde q_1 v + q_2 v^2 -
W(\varphi), 
\end{cases}
\end{align}
where $W(\varphi) := \Box \varphi- \tilde q_1 \varphi - \tilde q_2 \varphi^2$ 
and $\varphi :=  \tilde w_0 - w_0 =  \tilde u_0 - u_0$.
Multiplying the middle equation of \eqref{eq_potentials_equiv} by $v$ and using $\varphi=u_0-\widetilde u_0$, we obtain
\begin{align*}
2q_2 u_0 v - 2q_2\tilde u_0 v + (q_1  - \tilde q_1) v  = 0.
\end{align*}
Therefore
$
F  = \tilde F + \Box \varphi- \tilde q_1 \varphi - \tilde q_2 \varphi^2. 
$
This together with \eqref{eq_potentials_equiv} shows that Theorem \ref{thm_thm1} holds.
\end{proof}

\noindent
One can  show that $q_1$, $q_2$ and $F$ are known in the measurement 
set $U$, provided that we know $S$.

\begin{lemma}\label{lem_coeff_determined_in_U} 
The map $S$ determines $q_1$, $q_2$ and $F$ in the measurement set $U$.
\end{lemma}

\begin{proof} %
Let $S_{(q_1,q_2,F)}: E_c^s(U) \to E^{s+1}(U)$ denote the source-to-solution map
\[
  S_{(q_1,q_2,F)}(\epsilon f) = u|_U,
\]
where $u$ is the unique solution of 
\begin{equation}\label{eq: determination in U}
\begin{cases}
\square_g u + q_1 u + q_2u^2 = F+\epsilon f\\
u(0)=\phi_0,\, \p_t u(0)=\phi_1.
\end{cases}
\end{equation}
For smooth coefficients $q_1,q_2,F$, this map is $C^\infty$ in the Fr\'echet sense with respect to $\epsilon f$ near $\epsilon f=0$.  

Assume that for two triples $(q_1,q_2,F)$ and $(\tilde q_1,\tilde q_2,\tilde F)$ we have 
\[
  S_{(q_1,q_2,F)} = S_{(\tilde q_1,\tilde q_2,\tilde F)}.
\]
Then the corresponding solutions $u$ and $\tilde u$ coincide on $U$, in particular
\[
  u_0|_U = S_{(q_1,q_2,F)}(0) = S_{(\tilde q_1,\tilde q_2,\tilde F)}(0) = \tilde u_0|_U.
\]

Let 
\[
  h := \partial_\epsilon u|_{\epsilon=0}, \qquad \tilde h := \partial_\epsilon \tilde u|_{\epsilon=0}.
\]
Then $h$ and $\tilde h$ satisfy the linearized equations
\begin{align*}
\Box_g h + (q_1+2q_2 u_0) h &= f, \\
\Box_g \tilde h + (\tilde q_1+2\tilde q_2 u_0) \tilde h &= f,
\end{align*}
with zero Cauchy data. Since $S_{(q_1,q_2,F)}= S_{(\tilde q_1,\tilde q_2,\tilde F)}$, we have $h=\tilde h$ on $U$. Subtracting gives
\[
\big(q_1+2q_2 u_0 - \tilde q_1 - 2\tilde q_2 u_0\big) h = 0 \quad\text{in } U.
\]
Fix $x_0\in U$. Choose any $h\in C_0^\infty(U)$ with $h(x_0)\neq 0$, and set 
\[
  f = \Box_g h + (q_1+2q_2 u_0)h \in E_c^s(U).
\]
By uniqueness of the linear problem, $h$ is the corresponding solution, so such an $h$ indeed arises from some source $f$. Hence at $x_0$ we obtain
\[
q_1(x_0) + 2 q_2(x_0) u_0(x_0) = \tilde q_1(x_0) + 2 \tilde q_2(x_0) u_0(x_0).
\]

Let next
\[
  w := \partial_\epsilon^2 u|_{\epsilon=0}, \qquad \tilde w := \partial_\epsilon^2 \tilde u|_{\epsilon=0}.
\]
Then $w$ and $\tilde w$ satisfy
\begin{align*}
\Box_g w + (q_1+2q_2 u_0) w + 2 q_2 h^2 &= 0, \\
\Box_g \tilde w + (\tilde q_1+2\tilde q_2 u_0) \tilde w + 2 \tilde q_2 h^2 &= 0,
\end{align*}
with zero Cauchy data. From $S_{(q_1,q_2,F)} = S_{(\tilde q_1,\tilde q_2,\tilde F)}$ we have $w=\tilde w$ on $U$, so subtracting yields
\[
\big(q_1+2q_2 u_0 - \tilde q_1 - 2\tilde q_2 u_0\big) w + 2(q_2-\tilde q_2) h^2 = 0\quad\text{on } U.
\]
At $x_0$, the first bracket vanishes by the earlier computation, so
\[
  2(q_2(x_0)-\tilde q_2(x_0))\,h(x_0)^2 = 0.
\]
Since $h(x_0)\neq0$, this implies $q_2(x_0)=\tilde q_2(x_0)$, and hence also $q_1(x_0)=\tilde q_1(x_0)$. Since $x_0$ was arbitrary, we have $q_1=\tilde q_1$ and $q_2=\tilde q_2$ on $U$.

Finally, subtracting the original nonlinear equations in \eqref{eq: determination in U} and using $q_1=\tilde q_1$, $q_2=\tilde q_2$, and $u=\tilde u$ on $U$ gives $F=\tilde F$ on $U$.
\end{proof}

\noindent
Finally, in the following sections, we show that $q_2(x) = \tilde q_2(x)$, for $x \in D(U)$.
Note that it is enough to consider $x \in D(U)$, such that $q_2(x) \neq 0$
and $\tilde q_2(x) \neq 0$. Indeed, since $A:=\{ x \in D(U) \,:\,  q_2(x) \neq 0 \text{ and } \tilde q_2(x) \neq 0  \}$ we have $\overline A = D(U)$. Since $q_2=\widetilde q_2$ in $A$, then by continuity and \eqref{eq_sign_cond} the equality holds also in $D(A)$.

\subsection{Gaussian beams} \label{sec_GBs}

In this section we review the construction and properties of Gaussian beams,  
which are parameter dependent solutions to a wave equation $(\Box_g+Q)v=0$ that concentrates in the neighbourhood of a given lightlike geodesic.
We follow mainly the presentation in \cite{FO20}, \cite{KKL01} and \cite{lassas2025stability}. For more 
on Gaussian beams see \cite{KKL01} and \cite{Ralston}.

Let $Q \in C^\infty(N)$ 
and $W \subset N$ is some domain. A Gaussian beam $v$ is  a solution to 
\begin{equation} \label{eq_GB_eq}
\begin{aligned}
(\Box_g + Q)v =  0 \quad  \text{ on } \quad W, %
\end{aligned}
\end{equation}
of the form
\begin{equation} \label{eq_GB_form}
\begin{aligned}
v = e^{i \s \Phi/h} a + r,
\end{aligned}
\end{equation}
where $h \in (0,1)$ is a parameter and  
$r$ is a remainder term that becomes small in a Sobolev norm as $h$ becomes small,
and where the amplitude $a$ and phase $\Phi$ will be specified in more detail below.
The main feature of the Gaussian beams are that they
concentrate near a given lightlike geodesic $\gamma$, when $h$ becomes small. We call the part 
\begin{equation} \label{eq_formal_GB}
\begin{aligned}
\hat v := e^{i \s \Phi/h} a,
\end{aligned}
\end{equation}
without the  correction term $r$ a formal (or approximate) Gaussian beam. The formal Gaussian beam $\hat v$
is an approximate solution of \eqref{eq_GB_eq}. In the following definitions and results
we make the above more precise.

\begin{definition}[Formal Gaussian beam]\label{def_FGB}
A formal Gaussian beam $\hat v$ of degree $K$ along a lightlike geodesic $\gamma$
is a function of the form	\eqref{eq_formal_GB}, where $\Phi \in C^\infty(N)$
and $a \in C_0^\infty( N)$,
are such that the conditions (1)--(5) below hold.

\medskip
\noindent
We use the notation
\begin{align*}
(\Box_g + Q) \hat v  = e^{i  \Phi/h} \big(  (\Box_g + Q)  a  + h^{-2}  \mathbf{H}( \Phi)a  - i h^{-1} \mathbf{T} a  \big),
\end{align*}
where $\mathbf H \Phi := \langle \nabla \Phi ,\, \nabla \Phi  \rangle_g$ and 
$\mathbf T a:= 2\langle \nabla \Phi ,\, \nabla a  \rangle_g - (\Box_g \Phi) a$.
Moreover we let $(z^0,z^1,\dots,z^n) = (s,z') \in I \times B_{\delta'}$ denote Fermi coordinates. 

\begin{enumerate}
\item The phase function $\Phi$ satisfies the eikonal equation $\mathbf H \Phi=0$ to higher order on $\gamma$ in the sense that
\begin{equation} \label{eq_Phi_eqs}
\begin{aligned}
\p^\alpha_z \mathbf H \Phi (s,0) = 0,\quad s \in I, %
\end{aligned}
\end{equation}
for a multi-index $\alpha$,  $|\alpha| \leq K$.

\item The phase $\Phi$ satisfies  the following conditions
\begin{equation} \label{eq_Phi_cond}
\begin{aligned}
\operatorname{Im} \Phi(z)|_\gamma = 0,\qquad \operatorname{Im}\Phi(z) \geq C |z'|^2, \quad \forall z \in I \times B_{\delta'}. 
\end{aligned}
\end{equation}

\item The amplitude function $a$ can be expanded in powers of $h$ as 
\begin{equation} \label{eq_GB_amplitude}
\begin{aligned}
a = \chi\sum_{j=0}^K a_j h^j, \qquad a_j = \sum_{i=0}^K a_{j,i},
\end{aligned}
\end{equation}
where $a_{j,i}$ are complex-valued homogeneous polynomials in $z^1,..,z^n$, and $\chi(z) := \tilde \chi(|z|/\delta')$,
$ \tilde \chi \in C^\infty_0(\R)$ is non-negative and
$\tilde \chi(t)|_{ |t|\leq 1/4} = 1$ and  $\tilde \chi(t)|_{ |t|\geq 1/2} = 0$, and where $\delta' > 0$ is the radius of the  tubular neighborhood
$B_{\delta'} \times I$ of $\gamma$ where the Fermi coordinates are defined.

\item The amplitude functions $a_j$ satisfy the following transport type equations 
\begin{equation} \label{eq_GB_transport}
\begin{aligned}
\p^\alpha_z (\mathbf T a_0) (s,0) &= 0, \qquad   \forall s \in I, \\
\p^\alpha_z (-i\mathbf T a_k + (\Box_g + Q) a_{k-1} ) (s,0) &= 0, \qquad   \forall s \in I,
\end{aligned}
\end{equation}
for a multi-index $\alpha$,  $|\alpha| \leq K$.
\end{enumerate}
\end{definition}

\noindent
It is not difficult to see that the formal Gaussian beam $\hat v$ of definition \eqref{def_FGB}
is an approximate solution of \eqref{eq_GB_eq} as the following lemma states (see proof of Lemma 2 in  sec. 4.2
in \cite{FO20}).

\begin{lemma}\label{lem_Formal_GB_approx_sol}
Let $\hat v$ be a formal Gaussian beam of degree $K$ and $\gamma$ the corresponding geodesic, which intersects
$M_T$ and which has endpoints outside $M_T$. Then for all $h > 0$
 \begin{equation}\label{eq_formal_GB_est}
\norm{ (\square_g + Q) \hat v_h}_{H^k(M_T)} = \mathcal{O} \big( h^{\frac{K+1}{2} - \frac{5}{4} - k }\big), \qquad
\norm{\hat v_h}_{L^\infty ( M_T)} = \mathcal{O}(1), 
\end{equation}
as $h \to 0$.
\end{lemma}

\noindent
The remainder term $r$ in \eqref{eq_formal_GB} corrects the formal Gaussian beam
$\hat v$, so that  $v$ becomes to be a solution to the wave equation \eqref{eq_GB_eq}.
The remainder $r$ is moreover small for small $h>0$,
as the following Lemma shows. This is a consequence of Lemma \ref{lem_Formal_GB_approx_sol}.

\begin{lemma} \label{lem_GB_solves_wave_eq}
Suppose $\hat v$ is a formal Gaussian beam of degree $K$ on $M_T$. Then there exists an 
$r \in C^{\infty}(M_T)$ such that $v := \hat v + r$ solves \eqref{eq_GB_eq} in $M_T$, and we have the estimate
$$
\| r \|_{H^k(M_T)} = \mathcal{O} \big( h^{\frac{K+1}{2} - \frac{1}{4} - k}\big).
$$
\end{lemma}

\begin{proof} 
The remainder term $r$ needs to solve the equations
\begin{equation}
\begin{cases}\label{eq_for_v}
(\Box_g + Q )r = -(\Box_g + Q )\hat v , &\text{ in } M_T,\\
r(0,x)= 0,\,\, \p_t r(0,x) = 0, &\text{ on } M.
\end{cases} 
\end{equation}
The linear theory guarantees the existence of a unique solution $r$ to this problem, 
see~\cite[Theorem~13 and Corollary~17]{Bar15}. 
By the energy estimate for the linear problem in Proposition \ref{prop_energy_est_lin} 
and the estimate \eqref{eq_formal_GB_est}, we have that
\begin{align*}
\| r \|_{H^k(M_T)} 
\leq C \| (\Box_g + Q )\hat v \|_{H^{k-1}(M_T)} 
= \mathcal{O} \big( h^{\frac{K+1}{2} - \frac{1}{4} - k}\big). 
\end{align*}
\end{proof}

\noindent
In order to prove the existence of formal Gaussian beams one  needs to construct the 
phase function $\Phi$ and amplitude functions $a_j$. Next we will review the main steps 
of this construction with an emphasis on the parts that are of importance here.
The phase function $\Phi$ is constructed as a sum
\begin{equation} \label{eq_Phi}
\begin{aligned}
\Phi(s,z') = \sum_{j=0}^K \Phi_j(s,z'), 
\end{aligned}
\end{equation}
where $(s,z')$ are the Fermi coordinates and $\Phi_j$ are complex-valued homogeneous polynomials in $z^1,..,z^n$.
As in subsection 4.2.1 of \cite{FO20} 
on sees from equations \eqref{eq_Phi_eqs} and \eqref{eq_fermi_g} that we can  set 
\begin{equation} \label{eq_Phi01}
\begin{aligned}
\Phi_0(s,z') = 0, \qquad   \Phi_1(s,z') = z^1,   
\end{aligned}
\end{equation}
The $\Phi_2$ needs to be constructed so that \eqref{eq_Phi_cond} holds. Since $\Phi_2$ 
is a homogeneous polynomial we can write it as
\begin{equation} \label{eq_Phi2_as_H}
\begin{aligned}
\Phi_2(s,z') = \sum_{i,j=1}^K H_{ij}(s) z^iz^j.
\end{aligned}
\end{equation}
It is not difficult to see that  equation \eqref{eq_Phi_eqs}, with $|\alpha|=2$ implies that
$H(s)$ satisfies the matrix valued Riccati equation
\begin{equation} \label{eq_Riccati}
\begin{aligned}
\frac{d}{ds} H + HCH + D = 0, \qquad s \in I,
\end{aligned}
\end{equation}
see subsection 4.2.1 of \cite{FO20}, where
\begin{equation} \label{eq_C_D}
\begin{aligned}
C = \operatorname{diag}(0,2,2,\dots), \qquad D_{ij} := \p^2_{ij} g^{11}, 
\end{aligned}
\end{equation}
and where $\operatorname{diag}$ stands for a diagonal matrix. The 
solvability of this equation is investigated in \cite{KKL01}, see  section 2.4 Lemma 2.56. 
In accordance with \cite{FO20} we state  the following result.

\begin{lemma} \label{lem_Riccati}
Let $s_0 \in I$ and set the initial value $H(s_0) = H_0$, where the matrix $H_0$ is s.t.	
$\operatorname{Im} H_0 >0$. The Riccati equation \eqref{eq_Riccati} admits a unique solution $H$ on $I$.
We have moreover that $\operatorname{Im} H(s) >0$, for $s \in I$, and the factorization
$H(s)= Z(s)Y^{-1}(s)$, where
\begin{equation} \label{eq_Y_and_Z}
\begin{aligned}
\frac{ d}{ ds } Y =  CZ,\quad Y(s_0) = \operatorname{Id} \qquad
\frac{ d}{ ds  } Z = -DY, \quad Z(s_0) = H_0.
\end{aligned}
\end{equation}
The matrix $Y(s)$, $s \in I$ is invertible, and we have the following formula
$$
\det( \operatorname{Im} H(s))\, (\det (Y(s)) )^2 = \det(\operatorname{Im}(H(s_0))).
$$
\end{lemma}

\noindent
For the construction of the  higher order terms $\Phi_{j}$ with $j\geq 3$ we refer to \cite{KKL01} and \cite{FO20}.

\medskip
\noindent
Next we turn to the  construction of the amplitude function $a$ in \eqref{eq_GB_amplitude}.
The amplitude function $a$ can be determined by solving the transport type equations in \eqref{eq_GB_transport}
for $a_k$, $k=0,..,K$.
For a detailed construction of these we refer to \cite{KKL01} section 2.4.19. 
Here we only specify how  $a_{0}$ and $a_{1}$ are obtained, and focus 
in particular on $a_{0,0}$ and $a_{1,0}$, which will be needed later.
To determine $a_{0,j}$ we consider the zero order terms in $h$ in \eqref{eq_GB_transport}, when $\alpha=0$,
which gives
\begin{equation}\label{eq_transport_for_a0}
\begin{aligned}
&-2\s \langle  \nabla \Phi, \nabla a_0 \rangle_g + \square_g \Phi a_0=0.
\end{aligned}
\end{equation}
From this we obtain an ODE in the variable $s$ for  $a_{0,0}$ on the geodesic $\gamma$.
In the Fermi coordinates and using \eqref{eq_Phi01}, we see that 
\eqref{eq_transport_for_a0} gives
$$
-2 \sum_{i,j=0}^{n} g^{ij} \p_i a_0 \p_j \Phi -  \sum_{i=2}^{n}  \p^2_{ii} \Phi a_{0} = 0. %
$$
On the geodesic $\gamma$, where the metric can be expressed using \eqref{eq_fermi_g} and  where $z=0$, this reduces to
an ODE for $a_{0,0}$, which reads as
$$
2\p_s a_{0,0} +  \operatorname{Tr}(CH) a_{0,0} = 0,\qquad \operatorname{Tr}(CH) = \sum_{i=2}^{n}  \p^2_{ii} \Phi,
$$
where we used that on the geodesic we have that $\p^2_{ij} \Phi = 2H_{ij}$, due to \eqref{eq_Phi2_as_H}.
Using Lemma \ref{lem_Riccati} and the matrix formula $\operatorname{Tr} \log Y = \log \det Y$, for an
invertible $Y$, we can write
\begin{equation} \label{eq_trCH}
\begin{aligned}
\operatorname{Tr}(CH) = 
\operatorname{Tr}(CZY^{-1}) =
\operatorname{Tr}(\frac{d}{ds}Y Y^{-1}) = 
\frac{d}{ds} \operatorname{Tr}(\log Y) =
\frac{d}{ds} \log  \det Y.
\end{aligned}
\end{equation}
Solving the above ODE for $a_{0,0}$ gives thus and that 
$$
a_{0,0}(s) = e^{-\frac{1}{2}\int\operatorname{Tr}(CH) } = (\det Y(s))^{-1/2}.
$$
We set the initial value to be consistent with Lemma \ref{lem_Riccati} 
determined by $\det Y(s_0) = 1$.

For $a_{1}$ we obtain similarly from \eqref{eq_GB_transport} and from Fermi coordinate representation of $g$,
the equation
\begin{equation} \label{eq_sum_ak}
\begin{aligned}
-2 \sum_{i,j=0}^{n} g^{ij} \p_i a_1 \p_j \Phi -  \sum_{i=2}^{n}  \p^2_{ii} \Phi a_{1} =  (\square_g + Q) a_{0}.
\end{aligned}
\end{equation}
Using \eqref{eq_Phi01} we see that 
on the geodesic $\gamma$ we have an ODE in the $s$ variable, which looks like
$$
-2 \p_s a_{1,0} -  \operatorname{Tr}(CH) a_{1,0} = \square_g a_0 + Qa_{0,0}.
$$
Note that there is no
$a_{1,k}$ terms here, because with $z=0$, the $j$-index that survives in \eqref{eq_sum_ak} is $j=1$,
since $\Phi(s,z) = \sum_{|\alpha|}\Phi_\alpha(s) z^\alpha$,
with $\Phi_0 = 0, \Phi_1 = z^1$ 
and since $g|_\gamma = 2 ds \otimes dz^1 + \dots$ the $i$-index that goes with $j$ and survives is $i=0$. So we are left with $\p_s$.
We set the initial value $a_{1,0}(s_0) = 0$. 
This initial value problem  can solved explicitly on $I$, taking into account \eqref{eq_trCH}. 
In this way we obtain
\begin{equation} \label{eq_a_integrals}
\begin{aligned}
a_{1,0}(s) &= a^\sharp_{1,0} (s) + a^\flat_{1,0}(s) . \\
a^\sharp_{1,0} (s) &= \frac{ 1 }{ 2 } (\det Y(s))^{-1/2} \int_{s_0}^{s}	\Box_g a_0 (t,0) (\det Y(t))^{1/2}  \,dt\\
a^\flat_{1,0} (s) &=  \frac{ 1 }{ 2 } (\det Y(s))^{-1/2} \int_{s_0}^{s} Q(t,0)	\,dt.
\end{aligned}
\end{equation}
For the construction of the other terms in the amplitude expansion of $a$ we refer the reader to 
\cite{KKL01} and \cite{FO20}.

\medskip
\noindent
Finally we want to create Gaussian beams using  compactly supported source terms in $U$.
Next we will specify how to chose these sources. 

Let $q_0\in U$ and assume that $\gamma: (a,b) \to N$ is a lightlike geodesic, with $ \gamma(a) \in U$ and $\gamma(b) \notin M_T$, and $\gamma(s_0)=q_0$ for some $a<s_0<b$, and denote
$$
 \dot \gamma(s_0) = V.
$$
Let  $\hat v$ be a formal Gaussian beam associated to $\gamma$ that passes through the point $q_0$. 
In the Fermi coordinate tube $I \times B_{\delta'}$ we have that $q_0 = (s_0,0)$.
Assume furthermore that $\delta'$ is small enough, so that the cut-off $\chi$ in the amplitude $a$, see \eqref{eq_GB_amplitude}, is such that
$$
\supp_{z \in B(0 ,\delta')} \big( a(s, z ) \big) \Subset U  ,\qquad \forall s \in I',
$$
for some small interval $I' := (s_0-\eps, s_0+\eps)$.
We pick the cut-off functions $\eta_\pm \in C_0^\infty(I \times B_{\delta'})$, defined on the Fermi coordinate patch,
such that
for $(s,z) \in \supp(a)$ 
\begin{equation} \label{eq_cut_offs}
\begin{aligned}
&\eta_-(s,z)|_{s < s_0 -\eps }= 0, \qquad &\eta_-(s,z)|_{s > s_0}= 1, \\
&\eta_+(s,z)|_{s < s_0 } = 1, \qquad &\eta_+(s,z)|_{s > s_0 +\eps }= 0.
\end{aligned}
\end{equation}
We furthermore extend $\eta_\pm$ to the entire set $I \times M$ by zero, 
so that $\eta_\pm$ is defined also on $M_T$.
To obtain a Gaussian beam $v_f$ given by a source term $f$, with $\supp(f) \Subset U$, we take it as
the solution of the problem
\begin{equation} \label{eq_GB_with_source}
\begin{aligned}
\begin{cases}
(\square_g + Q) v_f =  \eta_{+}(\square_g + Q)(\eta_- \hat v) =: f, & \text{ in } M_T\\
v_f(0,x') = 0,\quad \p_t v_f(0,x')=0,\quad   &x' \in M.
\end{cases}
\end{aligned}
\end{equation}
For the existence of a unique solution to this problem see~\cite[Theorem~13 and Corollary~17]{Bar15}. 
Next we show that the solution $v_f$ is approximates the formal  Gaussian beam $\hat v$ for times
greater than $t_0 + \eps$, with an error bounded by a power of $h$.
We will in fact show prove the estimate
\begin{equation} \label{eq_v_f_rem}
\begin{aligned}
\| v_f - \eta_- \hat v \|_{H^k(M_T)}  
= \mathcal{O} \big( h^{\frac{K+1}{2} - \frac{1}{4} - k}\big),
\end{aligned}
\end{equation}
To derive this estimate we first note that
\begin{align*}
(\Box_g + Q) [v_f - \eta_- \hat v] = (\eta_+ - 1) (\Box_g + Q)(\eta_- \hat v). %
\end{align*}
Now by the energy estimates of the linear problem in Proposition \ref{prop_energy_est_lin} and Lemma \ref{lem_Formal_GB_approx_sol},
we get that
\begin{align*}
\|  v_f -\eta_- \hat v \|_{H^k(M_T)} 
\leq C 
\| (\square_g + Q)  \hat v \|_{H^{k-1}(M_T)} 
= \mathcal{O} \big( h^{\frac{K+1}{2} - \frac{1}{4} - k }\big),
\end{align*}
which proves \eqref{eq_v_f_rem}. Note that we can write
$$
v_f  = \hat v + r, \qquad \text{ with } \qquad r :=  v_f -  \eta_-\hat v,
$$
and that $v_f$ therefore behaves as a Gaussian beam of the form in
\eqref{eq_GB_eq} and \eqref{eq_GB_form}, for  times greater than $t_0 + \eps$.
We will also need solutions that act as time reversed variants of $v_f$. These are
constructed by switching the roles of $\eta_+$ and $\eta_-$ in the above construction,
so that in particular $v_f  = \hat v +  r$, with  $ r :=  v_f -  \eta_+\hat v$, and
$v_f$ be haves as a Gaussian beam for times less than $t_0 - \eps$.

\begin{remark}
Later we will need that the source terms $f$ and $\tilde f$ that are constructed as  above, 
are identical for the coefficients $q_1,q_2,F$ and $\tilde q_1,\tilde q_2, \tilde F$.
Note that this follows from the above construction, particularly from \eqref{eq_GB_with_source}, since the formal Gaussian beam 
$\hat v$ is the same on the cylinder $I'\times B_{\delta'}$ for both sets of coefficients.
This is because 
the coefficients $q_1,q_2,F$ and $\tilde q_1,\tilde q_2, \tilde F$ coincide on the entire set $U$
due to Lemma \ref{lem_coeff_determined_in_U}.
\end{remark}

\section{WKB interaction solutions} \label{sec_WKB_interaction}
In this section, we solve a source problem for the wave equation where the source is given by the product of two Gaussian beams. These solutions arise from the interaction between two colliding Gaussian beams, and we refer to them as WKB interaction solutions.

Specifically, we aim to find solutions to the problem
$$
(\square_g + Q ) w = q \app{v}{i}\app{v}{j},
$$
where $\app{v}{i}$ and $\app{v}{j}$ are as in \eqref{eq_GB_eq} and $q,Q \in C^\infty( N)$. Here,
$\app{v}{i}$ and $\app{v}{j}$ are Gaussian beams corresponding to distinct lightlike geodesics
$\gamma_i$ and $\gamma_j$, respectively. 
We denote by $p_0,\dots,p_N$ the intersection points in $D \setminus U$ of $\gamma_i$ and $\gamma_j$
i.e.
\begin{equation} \label{eq_geodesic_cond1}
\begin{aligned}
\gamma_i \cap \gamma_j \cap (D \setminus U) =
\{   p_0,p_1,\dots,p_N \}.
\end{aligned}
\end{equation}
The WKB interaction solution $w$ will be concentrated around these intersection points,
and have the form of a WKB ansatz in the neighbourhood of the point $p_0$.
The solution $w$ is split into
\begin{equation} \label{eq_wkb_interaction_form}
\begin{aligned}
\app{w}{ij} = \app{\hat w}{ij} + \app{w}{ij}_N + \app{r}{ij},
\end{aligned}
\end{equation}
here $\app{\hat w}{ij}$ will be a formal part supported in a neighbourhood of $p_0$ and $r$ is a remainder term small in $h$.
The second term $\app{w}{ij}_N$ is
an extra term that may or may not be present, and is due to the other 
possible intersection points $ \{ p_1,\dots, p_N \}$. 
\begin{remark}
Here we could also define $\app{w}{ij}_N$ as a formal expansion as we do for $\app{\hat w}{ij}$.
Note however that for the higher order WKB interaction solutions in subsection \ref{sec_high_WKB_interaction} this is not anymore possible.
\end{remark}
We begin by constructing the formal part $\app{\hat w}{ij}$. This is an approximate solution corresponding to a source $F$,
which takes the form of a product of two formal Gaussian beams near the point $p_0$. More specifically, we
consider a source term of the form:
\begin{equation} \label{eq_WKB_source}
\begin{aligned}
F =  q \chi_{p_0} \app{\hat v}{i} \app{\hat v}{j} 
=  \chi_{p_0} e^{i \Phi/h} \sum_{l=0}^{2K} q b_l h^l,
\end{aligned}
\end{equation}
where
\[
  \app{\hat v}{k}  =  e^{i \app{\Phi}{k} /h} \chi_k \sum_{l=0}^{K} \ap{k}_l h^l, \quad k = 1,2, 
\]
and $\chi_{p_0} = \chi_i \chi_j \chi_{B(p_0,r)}$, and where $\chi_i$ and $\chi_j$ are the cut-off functions
related to the amplitudes of the Gaussian beams $\app{v}{i}$ and $\app{v}{j}$, and $B(p_0,r)$ is a small ball around $p_0$.
The $\Phi$ in \eqref{eq_WKB_source} is given by $\Phi = \app{\Phi}{1} + \app{\Phi}{2}$, and  we further require that
\begin{equation} \label{eq_nabla_Phi_nonzero}
\begin{aligned}
\langle \nabla  \Phi , \nabla \Phi  \rangle_g (p) \neq 0 \quad \text{ for } \quad p \in N_0,
\end{aligned}
\end{equation}
where $N_0$ is a neighbourhood of the intersection point $p_0$, which we can set to be 
$$
N_0 := \supp(\chi_{p_0}).
$$
Note that this  differs from the requirement in definition \ref{def_FGB}.
The need for this will become clear in the proof of Lemma
\ref{lem_WKB_source}.

Since  $\Phi =  \app{\Phi}{i} + \app{\Phi}{j}$ and since the gradients $\nabla\app{\Phi}{i}(p_0)$  are lightlike,
we can rewrite condition \eqref{eq_nabla_Phi_nonzero} as
\begin{equation} \label{eq_Phi_nonzero}
\begin{aligned}
\langle \nabla  \Phi , \nabla \Phi  \rangle_g (p_0) = 
2 \langle  \nabla \app{\Phi}{i} ,  \nabla \app{\Phi}{j} \rangle_g (p_0)  \neq 0.
\end{aligned}
\end{equation}
This is something that always holds as Lemma \ref{lem_lightlike_independence} below shows. Note 
that by the continuity of $\nabla \Phi$,  
equation \eqref{eq_Phi_nonzero} holds for all $p \in N_0$.
\begin{lemma} \label{lem_lightlike_independence}
Let $p\in N$ and $\zeta_1, \zeta_2 \in T_pN$ be two linearly independent lightlike vectors. Then
$ \langle \zeta_1 , \zeta_2  \rangle_g \neq 0$.
\end{lemma}
\begin{proof} In normal coordinates the metric is the Minkowski metric at $p$. 
By scaling  $\zeta_j = (\zeta_j^0, \zeta'_j)$, $j=1,2$, we may assume that $\zeta_j  = (1, \zeta_j')$, where
$ \langle \zeta_j', \zeta'_j \rangle_{\R^n} = 1$. 
Assume the contrary that $\langle \zeta_1 , \zeta_2  \rangle_{g} = 0$.
Then $1 = \langle  \zeta_1' , \zeta'_2 \rangle_{\R^n}$ with $\zeta'_j \in \S^{n-1}$.
But this is possible only if $\zeta'_1 = \zeta'_2$. Thus $\zeta_1 = \zeta_2$ which is a contradiction.
\end{proof}

\noindent 
The following lemma shows that the solution $\app{\hat w}{ij}$ corresponding to a source term $F = q \chi_{p_0} \app{\hat v}{1} \app{\hat v}{2}$ 
admits an amplitude expansion similar to that of the formal Gaussian beams when restricted to the neighborhood $N_0$ of the point $p_0$.

\begin{lemma} \label{lem_WKB_source}
Assume that $Q \in C^\infty(M_T)$, and
let $F$ be of the form \eqref{eq_WKB_source}, so that  \eqref{eq_Phi_nonzero} holds 
in the neighbourhood $N_0$ of $p_0$. 
Then there exists a solution $\hat w \in C^\infty(M_T)$ 
to the equation
\begin{equation} \label{eq_wkb_source_prob}
(\Box_g + Q ) \hat w =  F + \hat r, \quad  \text{ in } \quad   M_T,
\end{equation}
where  $\hat w$ can be written in the form
\begin{equation} \label{eq_cj}
\begin{aligned}
\hat  w =   e^{i\Phi/h}  \sum_{j = 2}^{2K+2} c_j h^j, \quad \text{ in } \quad M_T,
\end{aligned}
\end{equation}
and where the $c_j \in C_0^\infty(N_0)$ are given by \eqref{eq_ck} and where the error term $\hat r \in C^\infty_0(N_0)$ 
satisfies the estimate
$$
\|  \hat r \|_{H^k( M_T )} \leq C  h^{2K-k+1}. 
$$
\end{lemma}

\begin{proof}
Our aim is to find a series  $\hat w = e^{ i\Phi/h} \sum_{j = 2}^{2K+2}  c_j h^j$ such that
\begin{equation} \label{eq_hat_w}
\begin{aligned}
(\square_g + Q)  \hat w  = e^{ i\Phi/h} \sum_{j=0}^{2K} q b_j h^j + \hat r. %
\end{aligned}
\end{equation}
We will proceed by finding the coefficients $c_j$ recursively by equating powers of $h$.
Firstly note that
\begin{small}
\begin{align*}
(\square_g +Q)  \hat w
&=
e^{i\Phi/h}  \sum_{j=2}^{2K+2} \big[
h^j  (\square_g+Q)  c_j  
+ i h^{j-1}(\square_g \Phi  c_j + 2 \langle \nabla \Phi , \nabla  c_j  \rangle_g) 
- h^{j-2}\langle  \nabla \Phi ,\nabla \Phi  \rangle_g  c_j \big].
\end{align*}
\end{small}
The series on the right has no zero or first order terms in $h$, so we set
\begin{equation} \label{eq_c0_c1}
\begin{aligned}
 c_0 =  c_1 = 0.
\end{aligned}
\end{equation}
Re indexing by $k=j-2$ and rearranging the terms in powers of $h$ gives
\begin{small}
\begin{align*}
(\square_g  + Q ) \hat w
&=
e^{i\Phi/h}  \sum_{k=0}^{2K} h^k \big[
(\square_g + Q)  c_k  
+ i\square_g \Phi   c_{k+1} + 2i \langle \nabla \Phi , \nabla   c_{k+1}  \rangle_g 
- \langle  \nabla \Phi ,\nabla \Phi  \rangle_g   c_{k+2} \big].
\end{align*}
\end{small}
Equating like powers on the right of this equation and \eqref{eq_hat_w}, gives firstly that 
\begin{equation} \label{eq_c2c3_def}
\begin{aligned}
  c_2 = \frac{ - q b_0 }{ \langle  \nabla \Phi ,\nabla \Phi
\rangle_g }, \qquad 
  c_3 = \frac{ q b_1 - i\square_g \Phi   c_{2} - 2i \langle
\nabla \Phi , \nabla   c_{2}  \rangle_g  }{ \langle  \nabla \Phi ,\nabla
\Phi  \rangle_g }.
\end{aligned}
\end{equation}
More generally we have with \eqref{eq_c0_c1}, that
\begin{equation} \label{eq_ck}
\begin{aligned} 
  c_j = \frac{ q b_{j-2} - (\square_g + Q)   c_{j-2} - i\square_g \Phi   c_{j-1} - 2i \langle \nabla \Phi , \nabla   c_{j-1}  \rangle_g  }
{ \langle  \nabla \Phi ,\nabla \Phi  \rangle_g }, \quad j=2,..,2K + 2.
\end{aligned}
\end{equation}
These equations define the coefficients $  c_j$ in terms of $ b_0,\dots, b_{2K}$ in
a recursive manner, and this determines  $\hat w$. Moreover we have that $  c_j \in C^\infty_0(N_0)$,
because $ b_j \in C^\infty_0(N_0)$.

Note that the above construction results also in an error $\hat r$, which consists
of the leftover terms of $(\Box_g + Q)\hat w$ 
that are not included in the definition of any of the $  c_2,..,  c_{2K+2}$.
These are 
\begin{equation} \label{eq_check_r}
\begin{aligned}
\hat r 
&:=
\big[ (\Box_g + Q) (  c_{2K+2} h^{2K+2} +  c_{2K+1}h^{2K+1} ) \\
&\quad - i\square_g \Phi   c_{2K+2} h^{2K+2} 
- 2i \langle \nabla \Phi , \nabla   c_{2K+2}  \rangle_g h^{2K+2} \big] e^{i\Phi/h}.
\end{aligned}
\end{equation}
For these we have the estimate 
$$
\| \hat r \|_{H^k( M_T)} \leq C  h^{2K-k+1}. 
$$
\end{proof}

\noindent
The previous Lemma shows that we can write $\app{\hat w}{ij}$ using an amplitude expansion modulo some error
in the neighbourhood $N_0$, when $\app{\hat w}{ij}$ is given by a source $F$ that is the product of
two formal Gaussian beams at $N_0$. 

We will now specify  how to obtain $\app{w}{ij}_N$. The $\app{w}{ij}_N$ part of $\app{w}{ij}$ is non zero if there are additional intersection points 
$ \{  p_1,\dots,p_N\}$. We let $\app{w}{ij}_N$ be the solution of
\begin{equation} \label{eq_wkb_w_N_prob}
\begin{aligned}
\begin{cases}
(\square_g + Q) \app{w}{ij}_N = F_N,  &\text{ in } \quad M_T, \\
\app{w}{ij}_N(0,x') = 0,\quad \p_t  \app{w}{ij}_N(0,x')=0, &\text{ on } \quad M.
\end{cases}
\end{aligned}
\end{equation}
where
$$
F_N =  q \chi_{N} \app{\hat v}{i} \app{\hat v}{j}, 
$$
and where $\chi_N$ is a cut-off function supported in the neighbourhood of the additional intersection
points $p_1,\dots,p_N$ and away from $p_0$. Note that \eqref{eq_wkb_w_N_prob} admits a solution by 
the linear theory of the wave equation on Lorenztian manifolds 
see~\cite[Theorem~13 and Corollary~17]{Bar15}.  

We will now add an error term to $\app{\hat w}{ij} + \app{w}{ij}_N$, to obtain the complete solution $\app{w}{ij}$,
with the source term $F$ given by the product of two Gaussian beams including their error terms. 

\begin{proposition} \label{prop_WKB_source_GBs}
Let  $Q \in C^\infty(M_T)$. Then there exists a  $w = \hat w + w_N + r \in C^\infty(M_T)$, where 
$\hat w$ is given by Lemma \ref{lem_WKB_source}, $w_N$ by \eqref{eq_wkb_w_N_prob}, and $r \in C^\infty(M_T)$, 
and $w$ solves the equation 
\begin{equation} \label{eq_wkb_source_prob_GBs}
\begin{aligned}
\begin{cases}
(\square_g + Q) w = q \app{v}{1} \app{v}{2},  &\text{ in } \quad M_T, \\
w(0,x') = 0,\quad \p_t  w(0,x')=0, &\text{ on } \quad M.
\end{cases}
\end{aligned}
\end{equation}
Furthermore we have the estimate
\begin{equation} \label{eq_w_rem}
\begin{aligned}
\| r \|_{H^k(M_T)} &= \mathcal O ( h^{\frac{K+1}{2} + \frac{3}{4} - k}), 
\end{aligned}
\end{equation}
\end{proposition}

\begin{proof}
The linear theory guarantees the existence of a unique solution $w$ to equation \eqref{eq_wkb_source_prob_GBs}, 
see~\cite[Theorem~13 and Corollary~17]{Bar15}. It remains to prove the estimate of the claim. 
If we set $r = w - \hat w - w_N$, then it follows from \eqref{eq_WKB_source} and Lemma \ref{lem_WKB_source}
that $r$ solves the equations
\begin{equation} \label{eq_wkb_r_prob}
\begin{aligned}
\begin{cases}
(\Box_g + Q) r = \ap{1} \app{r}{2} + \ap{2} \app{r}{1} + \app{r}{1} \app{r}{2} - \hat r,  &\text{ in }\quad M_T, \\ 
r(0,x') = 0,\quad \p_t  r(0,x')=0, &\text{ on } \quad M.
\end{cases}
\end{aligned}
\end{equation}
By the estimates of Proposition \ref{prop_energy_est_lin}, Lemma \ref{lem_GB_solves_wave_eq}, and Lemma \ref{lem_WKB_source},
we have that
\begin{align*}
\| r \|_{H^k(M_T)} 
\leq C \|  \ap{1} \app{r}{2} + \ap{2} \app{r}{1} + \app{r}{1} \app{r}{2} - \hat r\|_{H^{k-1}(M_T)} 
= \mathcal O ( h^{\frac{K+1}{2} + \frac{3}{4} - k} )
\end{align*}
\end{proof}

\begin{remark} We will apply Proposition \ref{prop_WKB_source_GBs} when $\app{v}{1}$ and $\app{v}{2}$
are given by some compactly supported source term as in \eqref{eq_GB_with_source}. The
source terms will however be supported on disjoint sets, so that Proposition
\ref{prop_WKB_source_GBs} is applicable.
\end{remark}

\subsection{A Higher order WKB interaction solution}\label{sec_high_WKB_interaction}

We shall also need another type interaction solution $\app{w}{ijk}$ that solves 
$$
(\square_g + Q ) \app{w}{ijk} = q \big( \app{v}{i}\app{w}{jk} + \app{v}{j}\app{w}{ik} + \app{v}{k} \app{w}{ij}\big),
$$
where the indices $i,j$ and $k$ are distinct, and
$\app{v}{i}$ are Gaussian beams as in \eqref{eq_GB_eq}, and $\app{w}{jk}$ are  WKB interaction solutions as in \eqref{eq_wkb_source_prob_GBs},
and $q,Q \in C^\infty(N)$. 
The solutions  $\app{v}{i}$ and $\app{w}{jk}$  correspond to the distinct lightlike geodesics $ \app{\gamma}{i}$, $ \app{\gamma}{j}$,  
and $ \app{\gamma}{k}$ respectively.
First we will construct a formal part $\app{\hat w}{ijk}$ which has a
source term $F$ of the form 
\begin{equation} \label{eq_WKB_source_w_v}
\begin{aligned}
F =  q \big( \app{\hat v}{i}\app{\hat w}{jk} + \app{\hat v}{j}\app{\hat w}{ik} + \app{\hat v}{k} \app{\hat w}{ij}\big)
=  e^{i \app{\Phi}{ijk}/h} \sum_{m=2}^{3K+2} q\beta_m h^m, 
\end{aligned}
\end{equation}
where 
\begin{equation*} 
\begin{aligned}
\app{\hat v}{i}  =  e^{i \app{\Phi}{i} /h} \chi_i \sum_{m=0}^{K} \ap{i}_m h^m, \quad 
\app{\hat w}{jk}  =  e^{i \app{\Phi}{jk} /h} \chi_{p_0} \sum_{m=2}^{2K+2} \ap{jk}_m h^m, 
\end{aligned}
\end{equation*}
where $i,j,k = 1,2,3$ are distinct, $\chi_{p_0}$ is as in \eqref{eq_WKB_source}
and where  $\app{\Phi}{ijk} = \app{\Phi}{i} + \app{\Phi}{jk}$, which can
be expanded as 
\begin{equation} \label{eq_Phi_ijk}
\begin{aligned}
\app{\Phi}{ijk}   %
= \app{\Phi}{i} + \app{\Phi}{j} + \app{\Phi}{k}.
\end{aligned}
\end{equation}
Note also that the lowest order term of the source has a $\beta_2$ term of the form
\begin{equation} \label{eq_beta_2}
\begin{aligned}
\beta_2 = \app{a}{i}_0 \app{a}{jk}_0 +  \app{a}{j}_0 \app{a}{ik}_0 +  \app{a}{k}_0 \app{a}{ij}_0, 
\end{aligned}
\end{equation}
where $ \app{a}{ij}$ is given by \eqref{eq_c2c3_def}.
As in \eqref{eq_Phi_nonzero}  we will need the following condition on the phase function 
\begin{equation} \label{eq_nabla_Phi_nonzero_2}
\begin{aligned}
\langle \nabla  \app{\Phi}{ijk}, \nabla \app{\Phi}{ijk} \rangle_g (p) 
\neq 0, \quad \text{ for } \quad p \in N_0,
\end{aligned}
\end{equation}
where $N_0$ is a neighbourhood of $p_0$.

The solution $\app{w}{ijk}$ is constructed following similar steps as  section \ref{sec_WKB_interaction}, i.e.
we build $\app{w}{ijk}$ out of three parts 
\begin{equation} \label{eq_wkb_2_interaction_form}
\begin{aligned}
\app{w}{ijk} = \app{\hat w}{ijk} + \app{w}{ijk}_N + \app{r}{ijk},
\end{aligned}
\end{equation}
here $\app{\hat w}{ijk}$ will be a formal part supported in a neighbourhood $N_0$ of $p_0$ 
and $r$ is a remainder term small in $h$, and $\app{w}{ijk}_N$ is a part due to other
possible intersections points $p_1,\dots,p_N$ in \eqref{eq_geodesic_cond1}, of the geodesics $ \app{\gamma}{i}$, $ \app{\gamma}{j}$,  
and $ \app{\gamma}{k}$ besides $p_0$.

Note that we do not have a counter part to Lemma \ref{lem_lightlike_independence}
in the case of $\app{w}{ijk}$. 
We need thus to be careful that $|\nabla  \app{\Phi}{ijk}|_g \neq 0$, so that \eqref{eq_nabla_Phi_nonzero_2} holds
when we construct the formal part $\app{\hat w}{ijk}$. The decomposition in \eqref{eq_wkb_2_interaction_form}
is made so that we need to worry about the condition \eqref{eq_nabla_Phi_nonzero_2} only in the neighbourhood $N_0$ near $p_0$,
since it is hard to guarantee that this conditions holds near the other intersections points.

The formal part $\app{\hat w}{ijk}$ is given by the following lemma.

\begin{lemma} \label{lem_WKB_source_w_v}
Assume that $Q \in C^\infty(N)$, and
let $F$ be of the form \eqref{eq_WKB_source_w_v}, and assume that  \eqref{eq_nabla_Phi_nonzero_2} holds 
in the neighbourhood $N_0$ of $p_0$. 
Then there exists a solution $\hat \omega \in C^\infty(M_T)$ 
to the equation
\begin{equation} \label{eq_wkb_source_prob_2}
(\Box_g + Q ) \hat \omega =  F + \hat r, \quad  \text{ in } \quad   M_T,
\end{equation}
Where  $\hat \omega$ can be written in the form
\begin{equation} \label{eq_thetaj}
\begin{aligned}
\hat  \omega = e^{i\app{\Phi}{ijk}/h}  \sum_{j = 4}^{3K+4} \theta_j h^j, \quad \text{ in } \quad M_T,
\end{aligned}
\end{equation}
and where the $\theta_j \in C_0^\infty(N_0)$ are given by \eqref{eq_thetak} and where the error term $\hat r \in C^\infty_0(N_0)$ 
satisfies the estimate
$$
\|  \hat r \|_{H^k( M_T )} \leq C  h^{3K-k+1}. 
$$
\end{lemma}

\begin{proof}
Our aim is to find a series  $\hat \omega = e^{i \Phi/h}  \sum_{j = 4}^{3K+4} \theta_j h^j$,
such that
\begin{equation} \label{eq_hat_w_2}
\begin{aligned}
(\square_g + Q)  \hat \omega   
=  e^{i \Phi /h} \sum_{m=2}^{3K+2} q \beta_m h^m 
+ \hat r. %
\end{aligned}
\end{equation}
We will proceed by finding the coefficients $ \theta_j$ recursively by equating powers of $h$.
Firstly note that
\begin{small}
\begin{align*}
(\square_g +Q)  \hat\omega 
&=
e^{i\Phi/h} \sum_{j = 4}^{3K+4}
\big[
h^j  (\square_g+Q)  \theta_j  
+ i h^{j-1}(\square_g \Phi  \theta_j + 2 \langle \nabla \Phi , \nabla  \theta_j  \rangle_g) 
- h^{j-2}\langle  \nabla \Phi ,\nabla \Phi  \rangle_g  \theta_j \big].
\end{align*}
\end{small}
Re indexing by $k=j-2$ and rearranging the terms in powers of $h$ gives
\begin{small}
\begin{align*}
(\square_g  + Q ) \hat\omega 
&=
e^{i\Phi/h}  \sum_{k=2}^{3K+2} h^k \big[
(\square_g + Q)\theta_k  
+ i\square_g \Phi \theta_{k+1} + 2i \langle \nabla \Phi , \nabla \theta_{k+1}  \rangle_g 
- \langle  \nabla \Phi ,\nabla \Phi  \rangle_g \gamma_{k+2} \big].
\end{align*}
\end{small}
and defining the missing 2nd and 3rd order coefficients as 
\begin{equation} \label{eq_theta2_theta3}
\begin{aligned}
 \theta_2 =  \theta_3 = 0.
\end{aligned}
\end{equation}
Equating like powers on the right of the previous equation and \eqref{eq_hat_w_2}, gives firstly that 
\begin{equation} \label{eq_beta4_beta5}
\begin{aligned}
\theta_4 = \frac{ -q\beta_2 }{ \langle  \nabla \Phi ,\nabla \Phi  \rangle_g }, \qquad 
\theta_5 = \frac{ q\beta_3 - i\square_g \Phi \theta_{4} - 2i \langle \nabla
\Phi , \nabla \theta_{4}  \rangle_g  }{ \langle  \nabla \Phi ,\nabla \Phi
\rangle_g }.
\end{aligned}
\end{equation}
More generally we have with \eqref{eq_theta2_theta3}, that
\begin{equation} \label{eq_thetak}
\begin{aligned} 
 \theta_j = q \theta_j =  \frac{ q\beta_{j-2} - (\square_g + Q) \theta_{j-2} 
- i\square_g \Phi \theta_{j-1} - 2i \langle \nabla \Phi , \nabla \theta_{j-1}  \rangle_g  }
{ \langle  \nabla \Phi ,\nabla \Phi  \rangle_g }, \quad j=4,..,3K + 2.
\end{aligned}
\end{equation}
These equations define the coefficients $\theta_j$ in terms of $\beta_0,\dots,\beta_{3K}$ in
a recursive manner, and this determines  $\hat w$. 
Moreover we have that $\theta_j \in C^\infty_0(N_0)$,
because $\beta_j \in C^\infty_0(N_0)$.

Note that the above construction results also in an error $\hat r$, which consists
of the leftover terms of $(\Box_g + Q)\hat \omega$ 
that are not included in the definition of any of the $\theta_4,..,\theta_{3K+2}$.
These are 
\begin{equation} \label{eq_check_r_2}
\begin{aligned}
\hat r 
&:=
\big[ (\Box_g + Q) (\theta_{3K+2} h^{3K+2} +\theta_{3K+1}h^{3K+1} ) \\
&\quad - i\square_g \Phi \theta_{3K+2} h^{3K+2} 
- 2i \langle \nabla \Phi , \nabla \theta_{3K+2}  \rangle_g h^{3K+2} \big] e^{i\Phi/h}.
\end{aligned}
\end{equation}
For these we have the estimate 
$$
\| \hat r \|_{H^k( M_T)} \leq C  h^{3K-k+1}. 
$$
\end{proof}

\noindent
The second part $\app{w}{ijk}_N$ in \eqref{eq_wkb_2_interaction_form} that is due to the  
additional intersection points $ \{  p_1,\dots,p_N\}$
is constructed as follows. We let $\app{w}{ij}_N$ be the solution of
\begin{equation} \label{eq_wkb_w_N_prob_2}
\begin{aligned}
\begin{cases}
(\square_g + Q) \app{w}{ijk}_N = F_N,  &\text{ in } \quad M_T, \\
\app{w}{ijk}_N(0,x') = 0,\quad \p_t  \app{w}{ijk}_N(0,x')=0, &\text{ on } \quad M.
\end{cases}
\end{aligned}
\end{equation}
and where
$$
F_N =  q \chi_{N} q ( \app{\hat v}{i}\app{\hat w}{jk} + \app{\hat v}{j}\app{\hat w}{ik} + \app{\hat v}{k} \app{\hat w}{ij})
$$
and where $\chi_N$ is a cut-off function supported in the neighbourhood of the additional intersection
points $p_1,\dots,p_N$
and away from $p_0$.
The existence and uniqueness of a solution \eqref{eq_wkb_w_N_prob_2} is guaranteed by 
the linear theory of the wave equation on Lorentzian manifolds 
see~\cite[Theorem~13 and Corollary~17]{Bar15}.  

Finally we specify the remainder term  $\app{r}{ijk}$ in \eqref{eq_wkb_2_interaction_form}.
This is given by the next lemma.

\begin{proposition} \label{prop_WKB_source_GBs_2}
Let  $Q \in C^\infty(M_T)$. Then there exists a  $\omega = \hat \omega + \omega_N + r \in C^\infty(M_T)$, where 
$\hat \omega$ is given by Lemma \ref{lem_WKB_source_w_v}, $\omega_N$ is given by \eqref{eq_wkb_w_N_prob_2},
and $r \in C^\infty(M_T)$, and $w$ solves the equation 
\begin{equation} \label{eq_wkb_source_prob_GBs_2}
\begin{aligned}
\begin{cases}
(\square_g + Q) \omega = q \app{v}{i} \app{w}{jk},  &\text{ in } \quad M_T, \\
\omega(0,x') = 0,\quad \p_t  \omega(0,x')=0, &\text{ on } \quad M.
\end{cases}
\end{aligned}
\end{equation}
Furthermore we have the estimate
\begin{equation} \label{eq_w_rem_2}
\begin{aligned}
\| r \|_{H^k(M_T)} &= \mathcal O ( h^{\frac{K+1}{2} + \frac{3}{4} - k}), 
\end{aligned}
\end{equation}
\end{proposition}

\begin{proof}
The linear theory guarantees the existence of a unique solution $\omega$ to equation \eqref{eq_wkb_source_prob_GBs_2}, 
see~\cite[Theorem~13 and Corollary~17]{Bar15}. It remains to prove the estimate of the claim. 
If we set $r = \omega - \hat \omega - \omega_{N}$, then it follows from \eqref{eq_WKB_source_w_v} and Lemma \ref{lem_WKB_source_w_v}
that $r$ solves the equations
\begin{equation} \label{eq_wkb_r_prob_2}
\begin{aligned}
\begin{cases}
(\Box_g + Q) r = q e^{i \app{\Phi}{ijk}/h} (\ap{1} \app{r}{23} + \ap{23} \app{r}{1} + \app{r}{1} \app{r}{23}) - \hat r,  &\text{ in }\quad M_T, \\ 
r(0,x') = 0,\quad \p_t  r(0,x')=0, &\text{ on } \quad M.
\end{cases}
\end{aligned}
\end{equation}
By the estimates of Proposition \ref{prop_energy_est_lin}, Lemma \ref{lem_GB_solves_wave_eq}, and Lemma \ref{lem_WKB_source},
we have that
\begin{align*}
\| r \|_{H^k(M_T)} 
\leq C \|  \ap{1} \app{r}{23} + \ap{23} \app{r}{1} + \app{r}{1} \app{r}{23} - \hat r\|_{H^{k-1}(M_T)} 
= \mathcal O ( h^{\frac{K+1}{2} + \frac{3}{4} - k})
\end{align*}
as claimed.
\end{proof}

\section{Higher order linearization}

\noindent
We will use the method of higher order linearization, 
introduced in \cite{LLLS21} to prove Theorem \ref{thm_thm1}.
For this we need an integral identity where the source-to-solution map is linearized using several parameters. 
Because of this we will consider in this section a number of linearized versions of the nonlinear 
equation \eqref{eq_wave_equation}.

We start by defining $u$ as a solution to
\begin{equation} \label{eq_u0_equation}
\begin{aligned}
\begin{cases}
\square_g u + q_1 u + q_2 u^2 =  F + f   ,& \text{ in } \quad M_T, \\
u(0,x') = 0,\quad \p_t  u(0,x')=0, &\text{ on } \quad M.
\end{cases}
\end{aligned}
\end{equation}
where the source that we control is of the form 
\begin{align}\label{eq_sources}
f &= \sum_{i= 1}^3 \epsilon_i \app{f}{i},   %
\end{align}
with $\supp(f) \Subset U$ and $\epsilon_1, \epsilon_2, \epsilon_3 > 0$.
First we set 
\begin{equation} \label{eq_u0}
\begin{aligned}
u_0 = u|_{\epsilon = 0},
\end{aligned}
\end{equation}
where  $\epsilon := (\epsilon_1, \epsilon_2, \epsilon_3)$.
The first linearizations of equation	\eqref{eq_wave_equation}  are given by 
$$
\app{v}{i} := \p_{\epsilon_i} u |_{\epsilon = 0}
$$
and solve the equations
\begin{equation} \label{eq_vi}
\begin{aligned}
\square_g \app{v}{i} + q_1 \app{v}{i} + 2 q_2 u_0  \app{v}{i} &=    \app{f}{i} , \quad \text{ in } \quad M_T. 
\end{aligned}
\end{equation}
The second linearizations are in turn given by 
$$
\app{w}{ij} := \p^2_{\epsilon_i \epsilon_j} u|_{\epsilon = 0}, \qquad i \neq j,
$$
that correspondingly satisfy the equations
\begin{equation} \label{eq_wij}
\begin{aligned}
\square_g \app{w}{ij} + (q_1 + 2 q_2 u_0) \app{w}{ij}  = - 2q_2 \app{v}{i} \app{v}{j} ,\quad  \text{in } \quad  M_T. 
\end{aligned}
\end{equation}
The third linearizations are given by 
$$
\app{w}{ijk} := \p^3_{\epsilon_i \epsilon_j \epsilon_k } u|_{\epsilon = 0},
$$
and these solve 
\begin{equation} \label{eq_w123}
\begin{aligned}
\square_g \app{w}{ijk} + (q_1 + 2 q_2 u_0)  \app{w}{ijk}  = 
-  2q_2 \big( \app{v}{i} \app{w}{jk} + \app{v}{j} \app{w}{ik} + \app{v}{k} \app{w}{ij} \big), \quad \text{ in } \quad M_T. 
\end{aligned}
\end{equation}
Next we derive an integral identity related to the third linearization.
We will assume that $\app{v}{0}$ is a wave satisfying a backwards wave equation, that is 
\begin{equation} \label{eq_vb}
\begin{aligned}
(\Box_g + q_1)\app{v}{0} + 2q_2u_0 \app{v}{0} = \app{f}{0}, \qquad \supp(\app{v}{0}) \Subset J^{-}(U),
\end{aligned}
\end{equation}
where $\supp(\app{f}{0}) \Subset U$. 
In the following lemma we will use the fact that the source-to-solution map $S$ is Fr\'echet differentiable.
This is verified during the proof of Proposition~\ref{prop: well posedness of nonlinear waves with sources}.

\begin{lemma} \label{lem_iid_3rd}
Let $S$ and $\tilde S$ be the source-to-solution maps corresponding to the coefficients $q_1,q_2,F$ and 
$\tilde q_1, \tilde q_2, \tilde F$ respectively. Furthermore, let $\app{v}{0}$ satisfy \eqref{eq_vb}.  We have the following 
\begin{small}
\begin{equation} \label{eq_iid_1}
\begin{aligned}
\big ( \p^3_{\eps_1 \eps_2 \eps_3} (S - \tilde S)|_{\eps=0} f \, , \app{f}{0}  \big )_{L^2(\Omega_T)} 
&= 
2 \big ( \tilde q_2 ( \app{\tilde v}{1} \app{\tilde w}{23} + \app{\tilde v}{2} \app{\tilde w}{13} + \app{\tilde v}{3} \app{\tilde w}{12}),	
\, \app{\tilde v}{0}   \big )_{L^2(\Omega_T)}  \\
&\quad -
2 \big ( q_2 ( \app{v}{1} \app{w}{23} + \app{v}{2} \app{w}{13} + \app{v}{3} \app{w}{12}), \, \app{v}{0}   \big )_{L^2(\Omega_T)},
\end{aligned}
\end{equation}
\end{small}
where $\app{v}{i}$ and $\app{\tilde v}{i}$ are given by \eqref{eq_vi}, 
and $\app{w}{ij}$ and $\app{\tilde w}{ij}$ are given by \eqref{eq_wij}.
\end{lemma}

\begin{proof}
Using \eqref{eq_w123}, \eqref{eq_vb}, and integration by parts, we have that
\begin{small}
\begin{align*}
\big ( \p^3_{123} S|_{\eps=0} f, \,  \app{f}{0}  \big )_{L^2} 
&= 
\big ( \app{w}{123} , \,  \app{f}{0}  \big )_{L^2(\Omega_T)}  \\
&= 
\big ( \app{w}{123} , \,  (\Box_g + q_1) \app{v}{0} + 2q_2 u_0 \app{v}{0} \big )_{L^2(\Omega_T)}  \\
&= 
-\big ((q_1 + 2 q_2 u_0)  \app{w}{123}  + 2q_2 ( \app{v}{1} \app{w}{23} + \app{v}{2} \app{w}{13} + \app{v}{3} \app{w}{12}), \, \app{v}{0}   \big )_{L^2(\Omega_T)}  \\
&\quad + \big ( \app{w}{123} , \,  q_1 \app{v}{0} + 2q_2 u_0 \app{v}{0} \big )_{L^2(\Omega_T)} 
- \big ( \app{w}{123} , \, \p_\nu \app{v}{0} \big )_{L^2(\p \Omega_T)} \\
&\quad + \big ( \p_\nu \app{w}{123} , \,  \app{v}{0} \big )_{L^2( \p \Omega_T)} \\
&=
-2 \big ( q_2 ( \app{v}{1} \app{w}{23} + \app{v}{2} \app{w}{13} + \app{v}{3} \app{w}{12}), \, \app{v}{0}   \big )_{L^2(\Omega_T)},
\end{align*}
\end{small}
where the boundary terms are zero, since $\app{v}{0}$ satisfies \eqref{eq_vb}, so that 
$$
\supp( \nabla_g \app{w}{123} \app{v}{0}) , \> \supp(\app{w}{123} \nabla_g \app{v}{0}) \subset J^{+}(U) \cap J^{-}(U) \Subset \Omega_T.
$$
Subtracting from the above integral identity for $q_1,q_2$ and $F$  the corresponding expression for $\tilde q_1, \tilde q_2$ and $\tilde F$,
yields the identity 
\begin{small}
\begin{align*}
\big ( \p^3_{123} (S - \tilde S)|_{\eps=0} f \, , \app{f}{0}  \big )_{L^2(\Omega_T)} 
&= 
2 \big ( \tilde q_2 ( \app{\tilde v}{1} \app{\tilde w}{23} + \app{\tilde v}{2} \app{\tilde w}{13} + \app{\tilde v}{3} \app{\tilde w}{12}), 
\, \app{\tilde v}{0}   \big )_{L^2(\Omega_T)} \\ 
&\quad -
2 \big (q_2 ( \app{v}{1} \app{w}{23} + \app{v}{2} \app{w}{13} + \app{v}{3} \app{w}{12}), \, \app{v}{0}   \big )_{L^2(\Omega_T)},
\end{align*}
\end{small}
proving the claim.
\end{proof}

\noindent
We shall furthermore need a fourth order linearization and a corresponding integral identity,
which will be used in section \ref{sec_q2_only}. The source is in  the fourth order case given by
\begin{align}\label{eq_sources_2}
f &= \sum_{i= 1}^4 \epsilon_i \app{f}{i},   %
\end{align}
with $\supp(f) \Subset U$ and $\epsilon_1, \epsilon_2, \epsilon_3, \epsilon_4 > 0$.
The fourth linearization is given by 
$$
\app{w}{1234} := \p^4_{\epsilon_1 \epsilon_2 \epsilon_3 \epsilon_4 } u|_{\epsilon = 0},
$$
where $u$ is given by \eqref{eq_u0_equation}. The first and second linearizations are defined as before,
with the source term given by \eqref{eq_sources_2},
and the third linearizations solves
\begin{equation} \label{eq_wijk}
\begin{aligned}
\square_g \app{w}{ijk} + (q_i + 2 q_2 u_0)  \app{w}{ijk}  = -  2q_2 ( \app{v}{i} \app{w}{jk} + \app{v}{j} \app{w}{ik} + \app{v}{k} \app{w}{ij}), \quad \text{ in } \quad M_T,  
\end{aligned}
\end{equation}
as in \eqref{eq_w123}. The fourth linearization $\app{w}{1234}$ is a  solution of
\begin{equation} \label{eq_w1234}
\begin{aligned}
\square_g \app{w}{1234} + (q_1 + 2 q_2 u_0)  \app{w}{1234}  = 
&-  2q_2 ( \app{w}{14} \app{w}{23} + \app{w}{24} \app{w}{13} + \app{w}{34} \app{w}{12}) \\
&-  2q_2 ( \app{v}{1} \app{w}{234} + \app{v}{2} \app{w}{134} + \app{v}{3} \app{w}{124} + \app{v}{4} \app{w}{123}), %
\end{aligned}
\end{equation}
in $M_T$.
We will again assume that $\app{v}{0}$ is solution moving backward in time that satisfies
\begin{equation} \label{eq_vb2}
\begin{aligned}
(\Box_g + q_1)\app{v}{0} + 2q_2u_0 \app{v}{0} = \app{f}{0}, \qquad \supp(\app{v}{0}) \Subset J^{-}(U),
\end{aligned}
\end{equation}
where $\supp(\app{f}{0}) \Subset U$. 

\begin{lemma} \label{lem_iid_4th}
Let $S$ and $\tilde S$ be the source-to-solution maps corresponding to the coefficients $q_1,q_2,F$ and 
$\tilde q_1, \tilde q_2, \tilde F$ respectively. Furthermore  let $\app{v}{0}$ satisfy \eqref{eq_vb}.  We have the following 
\begin{small}
\begin{equation} \label{eq_iid_2}
\begin{aligned}
\big ( \p^4_{\eps_1 \eps_2 \eps_3 \eps_4} (S - \tilde S)|_{\eps=0} f \, , \app{f}{0}  \big )_{L^2} 
&= 
\Big ( 2\tilde q_2 ( \app{\tilde w}{14} \app{\tilde w}{23} + \app{\tilde w}{24} \app{\tilde w}{13} + \app{\tilde w}{34} \app{\tilde w}{12}\\
&\quad+ \app{\tilde v}{1} \app{\tilde w}{234} + \app{\tilde v}{2} \app{\tilde w}{134} + \app{\tilde v}{3} \app{\tilde w}{124}+ \app{\tilde v}{4} \app{\tilde w}{123}) ,
\, \app{\tilde v}{0}   \Big )_{L^2(\Omega_T)}  \\
&\quad- 2 q_2 \Big ( \app{w}{14}  \app{w}{23} +  \app{w}{24}  \app{w}{13} +  \app{w}{34}  \app{w}{12}\\
&\quad+  \app{v}{1}  \app{w}{234} +  \app{v}{2}  \app{w}{134} +  \app{v}{3}  \app{w}{124} + \app{v}{4} \app{w}{123}) , \, \app{v}{0}   \Big )_{L^2(\Omega_T)}  
\end{aligned}
\end{equation}
\end{small}
where $\app{v}{i}$ and $\app{\tilde v}{i}$ are given by \eqref{eq_vi}, 
and $\app{w}{ij}$ and $\app{\tilde w}{ij}$ are given by \eqref{eq_wij},
and $\app{w}{ijk}$ and $\app{\tilde w}{ijk}$ are given by \eqref{eq_wijk}.
\end{lemma}

\begin{proof}
Using \eqref{eq_w1234}, \eqref{eq_vb2}, and integration by parts, we have that
\begin{align*}
\big ( \p^3_{1234} S|_{\eps=0} f, \,  \app{f}{0}  \big )_{L^2(U)} 
&= 
\big ( \app{w}{1234} , \,  \app{f}{0}  \big )_{L^2(\Omega_T)}  \\
&= 
\big ( \app{w}{1234} , \,  (\Box_g + q_1) \app{v}{0} + 2q_2 u_0 \app{v}{0} \big )_{L^2(\Omega_T)}  \\
&= 
-\big ((q_1 + 2 q_2 u_0)  \app{w}{1234}  , \, \app{v}{0}   \big )_{L^2(\Omega_T)}  \\
&\quad -\big ( 2q_2 ( \app{w}{14} \app{w}{23} + \app{w}{24} \app{w}{13} + \app{w}{34} \app{w}{12}), \, \app{v}{0}   \big )_{L^2(\Omega_T)}  \\
&\quad -\big ( 2q_2 ( \app{v}{1} \app{w}{234} + \app{v}{2} \app{w}{134} + \app{v}{3} \app{w}{124} + \app{v}{4} \app{w}{123}), \, \app{v}{0}   \big )_{L^2(\Omega_T)}  \\
&\quad + \big ( \app{w}{123} , \,  q_1 \app{v}{0} + 2q_2 u_0 \app{v}{0} \big )_{L^2(\Omega_T)} \\
&\quad - \big ( \app{w}{1234} , \, \p_\nu \app{v}{0} \big )_{L^2(\p \Omega_T)} \\
&\quad + \big ( \p_\nu \app{w}{1234} , \,  \app{v}{0} \big )_{L^2( \p \Omega_T)} \\
&=
-\big ( 2q_2 ( \app{w}{14} \app{w}{23} + \app{w}{24} \app{w}{13} + \app{w}{34} \app{w}{12}), \, \app{v}{0}   \big )_{L^2(\Omega_T)}  \\
&\quad -\big ( 2q_2 ( \app{v}{1} \app{w}{234} + \app{v}{2} \app{w}{134} + \app{v}{3} \app{w}{124} + \app{v}{4} \app{w}{123}), \, \app{v}{0}   \big )_{L^2(\Omega_T)}  
\end{align*}
where the boundary terms are zero, since $\app{v}{0}$ moves backward in time and thus satisfies \eqref{eq_vb}, so that 
$$
\supp( \nabla_g \app{w}{123} \app{v}{0}) , \> \supp(\app{w}{123} \nabla_g \app{v}{0}) \subset J^{+}(U) \cap J^{-}(U) \Subset \Omega_T.
$$
Subtracting from the above integral identity for $q_1,q_2$ and $F$  the corresponding expression for $\tilde q_1, \tilde q_2$ and $\tilde F$,
gives the identity 
\begin{small}
\begin{align*}
\big ( \p^4_{1234} (S - \tilde S)|_{\eps=0} f \, , \app{f}{0}  \big )_{L^2} 
&= 
\Big ( 2\tilde q_2 ( \app{\tilde w}{14} \app{\tilde w}{23} + \app{\tilde w}{24} \app{\tilde w}{13} + \app{\tilde w}{34} \app{\tilde w}{12}\\
&\quad\quad+ \app{\tilde v}{1} \app{\tilde w}{234} + \app{\tilde v}{2} \app{\tilde w}{134} + \app{\tilde v}{3} \app{\tilde w}{124}) + \app{\tilde v}{4} \app{\tilde w}{123}, \, 
\app{\tilde v}{0}   \Big )_{L^2(\Omega_T)}  \\
&-
\Big ( 2 q_2 (  \app{w}{14}  \app{w}{23} +  \app{w}{24}  \app{w}{13} +  \app{w}{34}  \app{w}{12}\\
&\quad\quad+  \app{v}{1}  \app{w}{234} +  \app{v}{2}  \app{w}{134} +  \app{v}{3}  \app{w}{124}) +  \app{v}{4}  \app{w}{123}, \, \app{v}{0}   \Big )_{L^2(\Omega_T)}  
\end{align*}
\end{small}
which proves the claim.
\end{proof}

\section{Proof of Theorem \ref{thm_thm1}} \label{sec_prf_thm1}

In the following subsections we will prove Theorem \ref{thm_thm1}.
Theorem \ref{thm_thm1} follows directly from Propositions \ref{prop_q32}, \ref{prop_q1} and
\ref{prop_F}.
We begin with some geometric preliminaries and the choosing of suitable geodesics in subsection \ref{sec_geom_setup}.
In subsection \ref{sec_q_2}  we illustrate the method of using 
Gaussian beams and  WKB interaction solutions, by proving that $q^2_2 = \tilde q_2^2$.
Then in subsection  \ref{sec_q_2} we refine this approach to show that $q_2 = \tilde q_2$.
Finally in subsection \ref{sec_q_1} we prove the claims concerning $q_1$ and $F$.

\subsection{Geometric setup} \label{sec_geom_setup} %

\noindent
In order to use the Gaussian beams and WKB interaction solutions discussed in the earlier sections,
we need to choose appropriate lightlike geodesics. These geodesics will be such that they intersect
at the point 
$$
p_0 \in D \setminus U,
$$
which will act as a common intersection point of all the geodesics we define below. See figure \ref{fig_geom_setup}.
It will  be important that the geodesics do not have certain types of other intersection points.
The absence of these intersection points is of especial importance  when we use the higher 
order WKB interaction solutions, that we defined  in section \ref{sec_high_WKB_interaction}.
The main  main task in this section is therefore to choose suitable geodesics 
with the desired properties.

\medskip
\noindent
Towards this end 
we define the time separation function $\tau : N \times N \to \R$,
which is for $q \in J^+(p)$ given by
\begin{small}
\begin{equation} \label{eq_tau}
\begin{aligned}
\tau(p,q) := \sup \{ L(\alpha) \,:\, \alpha:[0,1] \to N \text{ is a future pointing causal curve from $p$ to $q$} \},
\end{aligned}
\end{equation}
\end{small}
and for $p \not \in J^+(p) $, we set $\tau(p,q) = 0$.
See also Definition 14.15  in \cite{On83}. Here $L$ is the length of the segment $\alpha$, given by 
$$
L(\alpha) := \int_0^1 \sqrt{ -\langle \dot \alpha(s), \dot \alpha(s)  \rangle_g } ds,
$$
and $\alpha(0)= p$ and $\alpha(1)= q$.
For a point $p_0 \in D \setminus U$, we will choose points $q_0,q_1 \in U$,
and corresponding geodesics $\app{\gamma}{0}$ and $\app{\gamma}{1}$ in an optimal sense as follows.
By the definition of $D$, there exists a future pointing causal curve from
$p_0$ to a point $\tilde q_0 \in U$. Let $\tilde q_0 = (\tilde t, \tilde x)$.
The earliest observation time of a signal originating at the event $p_0$, and arriving at the 
spatial location $\tilde x$, is given by 
$$
t_0 = \inf \{ s \in [0,T] \,:\, \tau \big(p_0, (s,\tilde x)  \big) > 0 \},
$$
where $\tau$ is the time separation function given by  \eqref{eq_tau}.
We now let $q_0 := (t_0,\tilde x) \in U$. By definition we have that $\tau(p_0,q_0) = 0$.
If two distinct points have a time separation of zero, then they are connected by a lightlike geodesic, see
Theorem 10.46 in \cite{On83}. There is thus a future pointing lightlike geodesic $\app{\gamma}{0}$  from
$p_0$ to $q_0$

Next we pick $q_1$ and $\app{\gamma}{1}$ in an analogous manner.
By the definition of $D$, there exists a past pointing causal curve from
$p_0$ to a point $\tilde q_1 \in U$. Let $\tilde q_1 = (\tilde t, \tilde x)$.
Consider the earliest time $t_1$
at which a signal sent from  the spatial location $\tilde x$,
can be observed at the event $p_0$, this is given by 
$$
t_1 = \sup \{ s \in [0,T] \,:\, \tau((s,\tilde x), p_0) > 0  \},
$$
We now set $q_1 := (t_1, \tilde x) \in U$.
We pick the geodesic $\app{\gamma}{1}$, similarly as $\app{\gamma}{0}$, as the future pointing lightlike geodesic
connecting  $q_1$ to $p_0$.

The main properties of $\app{\gamma}{0}$, $\app{\gamma}{1}$, $q_0$ and $q_1$ are given by the following result, 
see Lemma 4 in \cite{FO20}. Note that this result is slightly more general then the above, since 
we in section \ref{sec_q_1} want to be able to make a small change to $p_0$ along the geodesic $\app{\gamma}{0}$,
without having to pick a new $\app{\gamma}{0}$. 
This changes $\app{\gamma}{1}$ to another nearby geodesics $\app{\hat \gamma}{1}$. In the sequel we will however denote
both geodesics by $\app{\gamma}{1}$.

\begin{figure}
\centering 
\includegraphics[width=5cm]{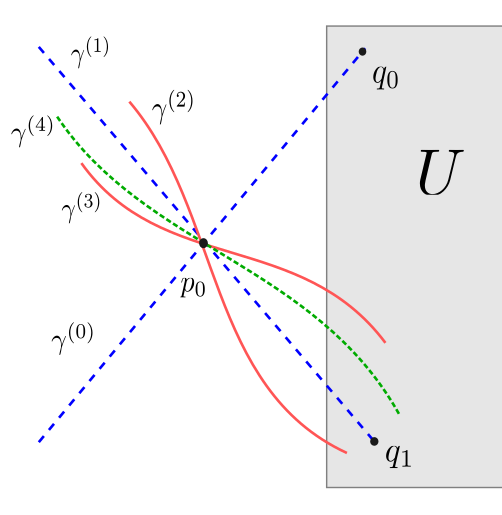}
\hspace{3.5cm}
\includegraphics[width=5cm]{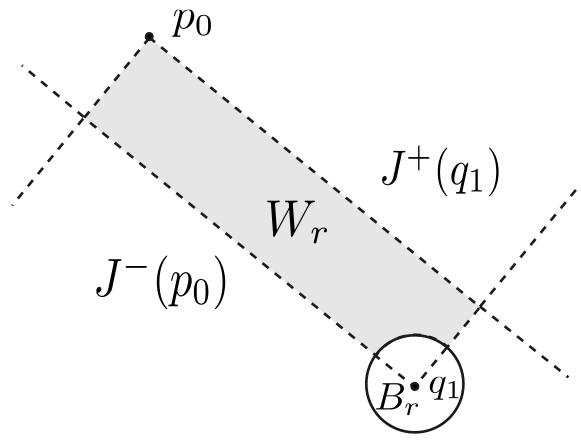}
\caption{On the left is an illustration of the choice of geodesics. The optimal geodesics $\app{\gamma}{0}$ and $ \app{\gamma}{1}$
are indicated by straight lines. On the right is an illustration of the set $W_r$.}\label{fig_geom_setup}
\end{figure}

\begin{lemma}\label{lem_variation_of_geodesics} 
Let $p_0 \in D \setminus U$. Then there exists a point $q_0  \in U$
and a lightlike geodesic $\app{\gamma}{0}$, with  $\app{\gamma}{0}(0) = q_0$. Let 
$p_0 = \app{\gamma}{0}(s_0)$, where $s_0 < 0$, then the following holds:
\begin{enumerate}[(i)]
\item  For all $\hat s_0 \in (s_0 -\eps , s_0+ \eps)$, where $\eps > 0$ is small, there exists 
a point $\hat q_{1} \in U$ and a future pointing lightlike geodesic $\app{\hat \gamma}{1}$, such that
$\app{\hat \gamma}{1}$ goes from $ \hat q_{1}$ to the point $\app{\gamma}{0}(\hat s_0)$.

\item  The point $ \app{\hat\gamma}{1}(\hat s_0)$ is the unique intersection point of $\app{\hat \gamma}{1}$ and $\app{\gamma}{0}$
in $M_T$.
\end{enumerate}
\end{lemma}

\noindent 
We will further need the additional geodesics $\app{\gamma}{2}$, $\app{\gamma}{3}$ and $\app{\gamma}{4}$,
which we will specify next.
To select $\app{\gamma}{2},\app{\gamma}{3}$ and $\app{\gamma}{4}$, we fix coordinates as follows.
Let $\app{\xi}{0},\app{\xi}{1} \in T_{p_0}N$ be the tangent vectors
$
\app{\xi}{0} = (\app{\gamma}{0})'|_{ p_0},
$ and $
\app{\xi}{1} = (\app{\gamma}{1})'|_{ p_0},
$
so that
\begin{equation} \label{eq_gamma_0_1}
\begin{aligned}
\app{\gamma}{0} = \gamma_{p_0,\app{\xi}{0}}
\quad \text{ and } \quad
\app{\gamma}{1} = \gamma_{p_0,\app{\xi}{1}}.
\end{aligned}
\end{equation}
We choose normal coordinates $(x^0,x^1,\dots,x^n)$ centered at the point $p_0$.
By rescaling the lightlike vectors if necessary, we may assume that the coordinates are such that
$$
\app{\xi}{0} =  (1,\dots), \qquad \app{\xi}{1} = (1,\dots).
$$
Notice that the spatial components $(\app{\xi}{0})'$ and
$(\app{\xi}{1})'$ of $(\app{\xi}{0})$ and
$(\app{\xi}{1})$,
lie in some $2$-dimensional plane. We can thus rotate the coordinate
system in the spatial coordinates, so that this plane corresponds to the $x^1 x^2$-plane, and so that 
the last coordinates $x^j$, $j \geq 3$ of $(\app{\xi}{1})'$ and $(\app{\xi}{0})'$ are zero. Further rotating in the
$x^1x^2$-plane, we can assume that $(\app{\xi}{0})' = (1,0,\dots,0)$,
since $\app{\xi}{0}$ is lightlike. Thus, we have 
$$
\app{\xi}{1} = (1,\dots,0), 
\qquad 
\app{\xi}{0} = (1,1,0,\dots,0). 
$$
Using the fact that $\app{\xi}{1}$ is lightlike, we have that
$$
0 = \langle  \app{\xi}{1}, \app{\xi}{1} \rangle_\eta = - 1 + s^2 + t^2, \quad \text{ where } \quad
\app{\xi}{1} = (1,s,t,0,\dots,0).
$$
By the above, we can express $t = \pm \sqrt{1-s^2}$, yielding coordinates in which
\begin{equation} \label{eq_coords_xi_0_1}
\begin{aligned}
\app{\xi}{0} = (1,s, \pm \sqrt{1-s^2},0,\dots,0), \qquad \app{\xi}{1} =  (1,1,0,\dots,0), \quad s \in [-1,1].
\end{aligned}
\end{equation}
We now choose two additional directions  $\app{\xi}{2}$ and $\app{\xi}{3}$, which will specify
the tangent vectors for $\app{\gamma}{2}$ and $\app{\gamma}{3}$.
These light like vectors  $\app{\xi}{2}$ and $\app{\xi}{3}$ are given by
\begin{equation} \label{eq_coords_xi_2_3}
\begin{aligned}
\app{\xi}{2} := (1,\sqrt{1-\sigma^2},\sigma,0,\dots,0), \qquad 
\app{\xi}{3} := (1,\sqrt{1-\sigma^2},- \sigma,0,\dots,0), \qquad \sigma \in (0,1).
\end{aligned}
\end{equation}
We define $\app{\gamma}{2}$ and $\app{\gamma}{3}$ as
\begin{equation} \label{eq_gamma_2_3}
\begin{aligned}
\app{\gamma}{2} := \gamma_{p_0,\app{\xi}{2}}
\quad \text{ and } \quad
\app{\gamma}{3} := \gamma_{p_0,\app{\xi}{3}}.
\end{aligned}
\end{equation}
We will also need a $\app{\gamma}{4}$ in section \ref{sec_q2_only}. This we define as 
\begin{equation} \label{eq_xi4}
\begin{aligned}
\app{\gamma}{4} := \gamma_{p_0,\app{\xi}{4}} \quad  \text{ where } \quad
\app{\xi}{4} :=  \big( \sqrt{ 1 + \sigma^2 } , 1 , \sigma ,0 ,\dots,0 \big).
\end{aligned}
\end{equation}
Note that the vectors $\app{\xi}{j}$, $j=2,3,4$ are  lightlike, so that $\app{\gamma}{j}$ are also lightlike. 
Notice also 
that $\app{\xi}{j}$  depends on the parameter $\sigma$, and that for small $\sigma>0$,
$$
\app{\xi}{j} \sim (1,1,0,\dots,0) = \app{\xi}{1}.
$$
The geodesics $\app{\gamma}{0},\dots,\app{\gamma}{4}$ have only one common intersection point in $M_T$,
which is $p_0$, since $\app{\gamma}{0}$ and $\app{\gamma}{1}$ have according to Lemma \ref{lem_variation_of_geodesics}
only one intersection point in the set $M_T$.
Pairs and triples of the geodesics $\app{\gamma}{0},\dots,\app{\gamma}{4}$ may however have several intersection points.
In order to use the higher order WKB interaction construction of Lemma \ref{lem_WKB_source_w_v}
we need to avoid these types of intersection points in  parts of $J^-(p_0)$. 
This is because we want to disregard the $\app{w}{ijk}_N$ and $\app{w}{ij}_N$ parts of the  WKB interaction solutions,
and this will be possible if $\supp(\app{w}{ijk}_N)$ and $\supp(\app{w}{ij}_N)$ do not intersects parts of $J^-(p_0)$.
The set where we want to avoid having additional intersection points is  
\begin{equation} \label{eq_D_P}
\begin{aligned}
W_r  := J^+(q_1) \cap J^-(p_0) \setminus J^-(B(q_1,r)),\qquad r>0.
\end{aligned}
\end{equation}
See figure \ref{fig_geom_setup}.
Next we apply a shortcut argument (see section 2 in \cite{KLU18}) to show that $\app{\gamma}{1}$ cannot intersect the other geodesics in $W_r$
except at $p_0$.
This type of argument is also used to prove Lemma \ref{lem_variation_of_geodesics}
in \cite{FO20}.

\begin{lemma} \label{lem_gamma_1_no_intersections}
The geodesic  $\app{\gamma}{1}$ intersects the geodesic $\app{\gamma}{j}$, $j=2,3,4$,
in the set $W_r$,  only at $p_0$.
\end{lemma}

\begin{proof}
Assume that the claim is false and that e.g. $\app{\gamma}{2}$ intersects $\app{\gamma}{1}$
at $p_1 \in W_r$, and $p_1 \neq p_0$.
We make the following construction. Consider the curve
\begin{align*}
\gamma = 
\begin{cases}
\app{\gamma}{1}, \quad \text{from point $q_1$ to $p_1$, } \\		
\app{\gamma}{2}, \quad \text{from point $p_1$ to $p_0$. }
\end{cases}
\end{align*}
Theorem 10.46 in \cite{On83} states that if $\mu$ is a causal curve from $p$ to $q$,
that is not a null pregeodesic, then there exists a timelike curve arbitrarily close to $\mu$ and in particular $\tau(p,q) > 0$.
The curve $\gamma$ is not a null pregeodesic, since we change the geodesic at point $p_1$.
It follows that there is a timelike curve $\alpha$ between $q_1$ and $p_0$, and that 
$\tau(q_1,p_0) > 0$. But this is a contradiction, since we chose $q_1$, so that 
$\tau(q_1,p_0) = 0$.
\end{proof}

\noindent
We will furthermore need to check that $\app{\gamma}{2}$, $\app{\gamma}{3}$ and $\app{\gamma}{4}$ do
not intersect in $W_r$.  The underlying reason why this holds for small $\sigma > 0$,
is that $\app{\gamma}{1}$ is in a sense optimal and that
$\app{\gamma}{j} \sim \app{\gamma}{1}$, $j=2,3,4$, when $\sigma$ becomes small.

For this we will employ similar arguments as in \cite{KLU18} and define the cut locus function 
$\rho : L^-N \to \R \cup \{ \infty\}$, where $L^-N$ is the bundle of past pointing
lightlike vectors, and
\begin{equation} \label{eq_cut_locus_func}
\begin{aligned}
\rho(p,\xi) := \sup \{ t \in [0,T') \,:\, \tau( \gamma_{p,\xi}(t), p ) = 0  \}, 
\end{aligned}
\end{equation}
where $T'$ is the maximal value for which the geodesic $\gamma_{p,\xi}$ is defined.
See also p. 794 in \cite{KLU18} and Definition 9.32 in \cite{BEE81}.
The  past null cut point $p(x,\xi)$ is given by 
\begin{equation} \label{eq_cut_point}
\begin{aligned}
p(x,\xi) = \gamma_{x,\xi}(t)|_{t=\rho(x,\xi)}. %
\end{aligned}
\end{equation}
Note that a cut point $p(x,\xi)$ is, on a globally hyperbolic manifold, the first conjugate point along $\gamma_{x,\xi}$
or the first point where there is another geodesic $\gamma_{x,\eta}$ to $p(x,\xi)$,
with $\eta \neq c \xi$, $c\in \R$, see \cite{BEE81} Theorem 9.15. 

Before proving that $\app{\gamma}{j}$, $j=2,3,4$ do not intersect in $W_r$, when $\sigma>0 $ is small,
we shall prove a slightly more general result.
We will show that %
there are no cut points $p(p_0, \app{\xi}{j})$, $j=2,3,4$ in $W_r$, when the $\sigma>0$ is small.
Due to the optimality of $\gamma^{(1)}$, geodesics originating from $p_0$ close to $\gamma^{(1)}$ cannot intersect each other between $p_0$ and $q_1$. To make the argument rigorous, we use the lower semicontinuity of the cut locus function and the fact that for $\sigma$  sufficiently small, the geodesics $\gamma^{(k)}$, $k=2,3,4$ lie close to $\gamma^{(1)}$. Thus any possible intersection points of the geodesics must occur in the past of $q_1$.

\begin{lemma} \label{lem_no_cut_points_in_W_r}
Let $r>0$. There are no cut points of the form $p \big(p_0, \app{\xi}{k} \big)$, with $k=2,3,4$, 
in $W_r$, when the parameter $\sigma>0$ is small.
\end{lemma}

\begin{proof}
We prove the claim for the geodesic $\app{\gamma}{2}$ in \eqref{eq_gamma_2_3}. The other cases are  identical.
Fix $r>0$ and assume, for contradiction, that the claim is false. Then there exists a sequence $\sigma_j \to 0$, as $j\to \infty$, and a corresponding sequence of cut points 
$$
\tilde p_j =  p\big(p_0,\app{\xi}{2}_j\big) \in W_r,
$$
where we write $\app{\xi}{2}_j = \app{\xi}{2}(\sigma_j)$ and $W_r$ was defined in \eqref{eq_D_P}. 
The cut locus function $\rho$ is lower semicontinuous on all globally hyperbolic manifolds,
see Propositions 9.33 in \cite{BEE81}. This means that  
$$
\liminf_{j} \rho \big(p_0, \app{\xi}{2}_j \big) \geq \rho \big(p_0, \app{\xi}{1} \big)
$$
since $\xi_j^{(2)}\to\xi^{(1)}$. 
It will be convenient to define $\tilde t$, $\tilde t_j$ and $t_0$ by
$$
\tilde t := \rho \big(p_0, \app{\xi}{1} \big), \qquad
\tilde t_j := \rho \big( p_0, \app{\xi}{2}_j \big), \qquad
\app{\gamma}{1} (t_0) = q_1.
$$
Note that we might have that $\tilde t> t_0$. 
Since  $\rho$  is lower semicontinuous at $(p_0,\app{\xi}{1}) \in L^-N$
and $\app{\xi}{2}_j \to \app{\xi}{1}$, there exists $\eps'_j \to 0$
such that
\begin{equation} \label{eq_t_ineq}
\begin{aligned}
\tilde t_j > \tilde t - \eps_j' 
\quad \Rightarrow \quad 
\tilde t_j \geq t_0 - \eps_j',
\end{aligned}
\end{equation}
where we used that $\tilde t \geq t_0$, which holds since $\app{\gamma}{1}$ has no cut points 
in between $p_0$ and $q_1$.

\medskip
\noindent
We will now show that $\tilde p_j \in J^-(B(q_1,r))$, when $j$ is large, which is a contradiction
by the definition of $W_r$. Using \eqref{eq_t_ineq}, we write $\tilde t_j$ as
$$
\tilde t_j = t_0 - \eps_j' + \beta_j,
$$
where $\beta_j \geq 0$. 
Let us first consider the point $\app{\gamma}{2}(\tilde t'_j)$, where 
$$
\tilde t'_j := t_0 - \eps_j',
$$
This point $\app{\gamma}{2}(\tilde t'_j)$ is not as far along the geodesic $\app{\gamma}{2}_j$ as $\tilde p_j = \app{\gamma}{2}(\tilde t_j)$.
By the continuity of the exponential map and, since $\exp_{p_0}(t \xi) = \gamma_{p_0,\xi}(t)$, we have that
$$
\app{\gamma}{2}_j(\tilde t'_j) = \exp_{p_0}( \tilde t'_j\xi^{(2)}_j) \in B(q_1, r)
$$
for all $j$ large enough, since $\tilde t_j'\to t_0$ and $\xi_j^{(2)}\to \xi^{(1)}$ as $j\to \infty$.  
Thus we see that $\app{\gamma}{2}_j(\tilde t'_j)  \notin W_r$.
Now let us go back to  $\tilde p_j = \app{\gamma}{2}(\tilde t_j)$.
We know that
$
\app{\gamma}{2}( \tilde t'_j ) \in B_r(q_1, r)
$
and since $\tilde t_j = \tilde t'_j  + \beta$, with $\beta > 0$, we see that $\tilde p_j$ is further to 
the past along the geodesic $\app{\gamma}{2}_j$, i.e.
$$
\tilde p_j = \app{\gamma}{2}_j( \tilde t'_j  + \beta) \in J^-(B_r(q_1, r)),
$$
and thus $\tilde p_j \notin W_r$, when $j$ is large.
It follows that $W_r$ contains no cut points $\tilde p_j$, when $j$ is large
or, in other words, when $\sigma_j$ is small.
\end{proof}

\noindent 
As a direct consequence of Lemma \ref{lem_no_cut_points_in_W_r} we have the following result.

\begin{lemma} \label{lem_gamma_2_3_4_no_intersections}
Let $r>0$. The geodesics  $\app{\gamma}{2}$, $\app{\gamma}{3}$ and $\app{\gamma}{4}$  have no common intersection points 
in $W_r$, apart from $p_0$, provided that  $\sigma > 0$ is small.
\end{lemma}

\begin{proof}
Assume the claim is false and that e.g. $\gamma^{(2)}$ and $\gamma^{(3)}$ have an intersection point in $W_r$ for all arbitrarily small $\sigma >0$. The other cases are identical. 
Then there exists a sequence $\sigma_j \to 0$, as $j \to \infty$,
and a sequence of intersection points of $p_j \in \app{\gamma}{2}_j \cap \app{\gamma}{3}_j$,
with $p_j \in W_r$.

Let $k=2$ or $k=3$.
The point $p_j$ is either the cut point $p(p_0, \app{\xi}{k}_j)$, or 
the cut point  $p(p_0,\xi^{(k)}_j) = \tilde p_j$, where $\tilde p_j$ lies 
in between $p_0$ and $p_j$ on the geodesic $\gamma^{(k)}_j$   (see e.g. Lemma 9.13 \cite{BEE81}). Thus   $p_j\in W_r$.
In both cases we have a cut point in the set $W_r$ for arbitrarily small $\sigma_j>0$ contradicting 
Lemma \ref{lem_no_cut_points_in_W_r}.
\end{proof}

\noindent
\begin{remark}\label{rem_support} Later we choose Gaussian beams $ \app{v}{j}$, $j = 2,3,4$, with
sources $ \app{f}{j}$, such that $\supp( \app{f}{j}) \subset B(q_j, r_j) \subset U$,
where the points $q_2,q_3,q_4$ are near the point $q_1$
and $r_j >0 $ is small. Notice that Lemmas \ref{lem_gamma_1_no_intersections} 
and \ref{lem_gamma_2_3_4_no_intersections} imply that we can choose small enough $\sigma> 0$
and $B(q_1,r)$, so that $q_2,q_3,q_4 \notin B(q_1,r)$ and the  geodesic segments
\begin{align*}
\app{\gamma}{j}_+ = \app{\gamma}{j} \cap J^+(q_j)
\end{align*}
have only the intersection point $p_0 \in \app{\gamma}{i}_+ \cap \app{\gamma}{j}_+$ in $W_r$, when $i,j=1,2,3,4$
are distinct.
\end{remark}

\noindent
Finally we will need  information on the possible future intersection points of 
the geodesics $\app{\gamma}{j}$ in $J^+(p_0)$. We want in particular that these stay away from the
geodesic $\gamma^{(0)}$. The following lemma ensures this:  It states that for a formal Gaussian beam on $\gamma^{(0)}$ with support in a sufficiently small tube around it, has a support that is disjoint from the geodesics $\gamma^{(j)}$ ($j=1,2,3,4$) outside a neighborhood of $p_0$.

\begin{lemma} \label{lem_gamma_1_future}
There exists $r>0$ such that the following holds. Let $N_r$ be a tubular neighbourhood of radius $r>0$ about the geodesic segment of $\app{\gamma}{0}$
starting at a point $p \notin B(p_0, \eps)$, $\eps>0$ and ending at $q_0$. Then 
$$
 \app{\gamma}{j} \cap N_r  = \emptyset, \qquad j=1,2,3,4,
$$
when $\sigma > 0$ is sufficiently small.
\end{lemma}

\begin{proof}
Denote by $\tilde \gamma$ the geodesic segment of $ \app{\gamma}{0}$ starting at the 
$p \notin B(p_0, \eps)$ and ending at $q_0$ in the claim.
We have that $\tilde \gamma$ and $ \app{\gamma}{1}$ do
not intersect by Lemma \ref{lem_variation_of_geodesics}.
Now  since the geodesics $ \app{\gamma}{j} \to \app{\gamma}{1}$, $j=2,3,4$, as $\sigma \to 0$,
we also have that  $\tilde \gamma$ and $ \app{\gamma}{j}$ do not intersect,
when $\sigma >0$ becomes small. There is hence a minimum positive distance (with respect to some auxiliary Riemannian distance) between the geodesics $ \app{\gamma}{j}$,
which are closed sets, and the
segment $\tilde\gamma$, which is compact. We can hence pick  a tubular neighbourhood $N_r$ about
$\tilde \gamma$ for which the claim holds.
\end{proof}

\subsection{Choosing special solutions } \label{sec_GBs_and_geodesics}

We now select specific  Gaussian beams  $\app{v}{0},\app{v}{1},\app{v}{2}$ and  $\app{v}{3}$, along with
the corresponding WKB interaction solutions $\app{w}{12}, \app{w}{13}$ and $\app{w}{23}$,
which will be used along with the integral identity of Lemma \ref{lem_iid_3rd} to prove $q^2_2 = \tilde q^2_2$, in
subsection \ref{sec_q_2} and then again in subsection \ref{sec_q_1} to prove the claims about $q_1$ and $F$. 
For this, we use  the following lemma (see also \cite[Lemma 1]{CLOP19}).

\begin{lemma} \label{lem_xis_lin_indep}
Let $\app{\xi}{0},\app{\xi}{1},\app{\xi}{2},\app{\xi}{3} \in T_{p_0}N$ 
be the lightlike vectors defined in section \ref{sec_geom_setup},
and which can in the normal coordinates chosen in section \ref{sec_geom_setup}, be written 
as
$$
\app{\xi}{1} =  (1,1,0,\dots,0),  \qquad
\app{\xi}{0} = \Big(1,s_0,\pm \sqrt{1-s_0^2},0,\dots,0 \Big), \quad s_0 \in [-1,1],
$$
and
$$
\app{\xi}{2} := (1,\sqrt{1-\sigma^2},\sigma,0,\dots,0), \qquad 
\app{\xi}{3} := (1,\sqrt{1-\sigma^2},- \sigma,0\dots,0), \qquad \sigma \in (0,1).
$$
Then the set $\{ \app{\xi}{0},\app{\xi}{1}, \app{\xi}{2}, \app{\xi}{3}\}$ is linearly dependent, so that 
\begin{equation} \label{eq_lin_dependent}
\begin{aligned}
 \sigma^2 \app{\xi}{0} = \kappa_1 \app{\xi}{1} + \kappa_2 \app{\xi}{2} + \kappa_3 \app{\xi}{3}, %
\end{aligned}
\end{equation}
where the $\kappa_i$ are such that
\begin{align*}
\kappa_1 \sim  -2(1-s_0) + \frac{\sigma^2}{ 2 }( 3 - s_0),  \quad %
\kappa_2 \sim  1-s_0 \pm \frac{\sigma}{2} \sqrt{ 1 - s_0^2}, \quad 
\kappa_3 \sim  1-s_0 \mp  \frac{\sigma}{2} \sqrt{ 1 - s_0^2},
\end{align*}
where the sign corresponds to sign in the definition of $\app{\xi}{0}$.
\end{lemma}

\begin{proof}
We can find the $(\kappa_1,\kappa_2,\kappa_3)$ from the equation \eqref{eq_lin_dependent}
by solving an explicit linear system by Gaussian elimination. 
The asymptotics for $\kappa_j$ are obtained from the exact expressions listed in 
Appendix~\ref{sec_explicit}.
\end{proof}

\medskip 
\noindent
We now pick the formal Gaussian beams $ \app{\hat v}{j}$, $j=0,1,2,3$, corresponding to the geodesics
$\app{\gamma}{j}$ that we defined in section \ref{sec_geom_setup}.
Let $Q := q_1 + 2 q_2$ as in \eqref{eq_vi}. Choosing  cut-off  functions as in \eqref{eq_cut_offs}
and bearing in mind remark \ref{rem_support}, we let
$\app{v}{j}$ be solutions of 
\begin{equation} \label{eq_specific_GB_source_prob}
\begin{aligned}
\begin{cases}
(\square_g + Q) \app{v}{j} = \app{f}{j}, \qquad  &\text{ in } \quad M_T \\
\app{v}{j}(0,x') = 0,\quad \p_t  \app{v}{j}(0,x')=0, \qquad &\text{ on } \quad M.
\end{cases}
\end{aligned}
\end{equation}
where $\app{f}{j} : = \eta_{-}(\square_g + Q)(\eta_+ \app{\hat v}{j})$, for $j=1,2,3$ 
and $\app{f}{0} : = \eta_{+}(\square_g + Q)(\eta_- \app{\hat v}{0})$. Note
that $\app{v}{0}$ can thus be thought of as a solution propagating back in time from a source 
at the point $q_0$ specified in Lemma \ref{lem_variation_of_geodesics}.
The solutions  $\app{v}{j}$ are of the form 
\begin{equation} \label{eq_GB_choices}
\begin{aligned}
\app{v}{j}(x) = \big [e^{i \app{\Phi}{j}/h} \ap{j} + \app{r}{j}\big] (x),\qquad x \in M_T \setminus U. 
\end{aligned}
\end{equation}
where the $\Phi^{(j)}$ are the phase functions determined by the $\app{\xi}{j}$ of Lemma \ref{lem_xis_lin_indep},
so that
$$
\nabla_g \app{\Phi}{j}( p_0) = \app{\xi}{j}.
$$
As a further step we will rescale the parameter $h$ in these Gaussian beams.
We replace $h$ with $ \frac{ h }{ |\kappa_j| }$,
where we choose $\kappa_j$, $ j=0,1,2,3$, as in Lemma \ref{lem_xis_lin_indep}, and set $\kappa_0 = \sigma_0^2$.
For $j=2,3$ we redefine $\app{v}{j}$ as
\begin{equation} \label{eq_GB_reparam_1}
\begin{aligned}
\app{v}{j} 
= 
e^{i \app{|\kappa_j| \Phi}{j}/h}\sum_{k=0}^K \ap{j}_k \Big( \frac{h} {|\kappa_j|}\Big)^k  + \app{r}{j}, 
\quad j=2,3.
\end{aligned}
\end{equation}
Note that $\kappa_2,\kappa_3 > 0$, for small $\sigma > 0$.
For $j=0,1$ we use the complex conjugates $ \app{\bar v}{j}$, which is motivated by the fact that $\kappa_0,\kappa_1 < 0$,
as $\sigma > 0$. And we redefine $ \app{v}{j}$, $j=0,1$ as  
\begin{equation} \label{eq_GB_reparam_2}
\begin{aligned}
\app{v}{j} 
= 
e^{-i \app{|\kappa_j| \overline \Phi}{j} / h }\sum_{k=0}^K \app{\bar a}{j}_k \Big( \frac{h} {|\kappa_j|}\Big)^k  + \app{\bar r}{j}, 
\quad j=0,1.
\end{aligned}
\end{equation}
Here we have dropped the complex conjugate from our notation, i.e. we let $\app{v}{j}$, in the case of $j=0,1$, denote
the $\kappa$ rescaled  solutions of \eqref{eq_GB_choices} that are complex conjugated.
Note that since $\app{v}{0}$ and $\app{v}{1}$ are complex conjugates of Gaussian beams they are also solutions to
the wave equation \eqref{eq_GB_eq}, and that the estimates of section \ref{sec_GBs} hold.

It will be convenient to  introduce the notation
\begin{equation} \label{eq_Psi}
\begin{aligned}
\Psi^{(j)} :=
\begin{cases}
-|\kappa_j| \overline \Phi^{(j)}, \quad &j = 0,1.  \\
\phantom{-}|\kappa_j| \Phi^{(j)}, \quad  &j = 2,3. 
\end{cases}
\end{aligned}
\end{equation}
The $\Psi^{(j)}$ depends also on the parameter $\sigma$ specified in Lemma \ref{lem_xis_lin_indep}
and note that since $\kappa_0,\kappa_1 < 0$ and $\kappa_2,\kappa_3 > 0$, for small $\sigma > 0$, we have 
that
\begin{equation} \label{eq_no_abs}
\begin{aligned}
\Psi^{(j)} = \kappa_j \Phi^{(j)}, \qquad \text{ when $\sigma >0$ is small}. 
\end{aligned}
\end{equation}
We will in the following assume that $\sigma > 0$ is small enough so that we can use this 
way of writing $\Psi^{(j)}$.
We also use in analogy with the notations $\app{\Phi}{ij}$ and $\Phi^\sharp_{ij}$,
the notations related to  $\app{\Psi}{j}$, i.e.
\begin{equation} \label{eq_Phi_notation}
\begin{aligned}
\app{\Psi}{ij} := \app{\Psi}{i} +\app{\Psi}{j}, \qquad \Psi^\sharp_{ij} := 
 \frac{ 1 }{ \langle  \nabla \app{\Psi}{ij} ,\nabla \app{\Psi}{ij}  \rangle_g}.
\end{aligned}
\end{equation}
Lemma \ref{lem_WKB_source} gives us the WKB interaction solutions $\app{w}{12},\app{w}{13}$, and $\app{w}{23}$.
These solve in accordance with \eqref{eq_wij}, the equation
\begin{equation} \label{eq_wij_2}
\begin{aligned}
(\square_g + Q) \app{w}{ij} = q_2\app{v}{i}\app{v}{j} 
\end{aligned}
\end{equation}
where $Q := q_1 + 2 q_2 u_0$.
The WKB interaction solution is well-defined when \eqref{eq_Phi_nonzero} holds.
By computing the Minkowski inner products we see that 
\begin{align*}
\big \langle\nabla \app{\Psi}{i}( p_0), \, \nabla \app{\Psi}{j}( p_0)  \big \rangle_g 
=  \kappa_i \kappa_j \langle\app{\xi}{i} , \,  \app{\xi}{j}  \rangle_\eta \neq 0,
\end{align*}
where $ij = 12,13,23$.
Later we will need the following lemma.

\begin{lemma} \label{lem_Phis_nonzero} 
Let $  p_0 \in D \setminus U$ be the intersection point of the geodesics $\app{\gamma}{i}$, $i=0,1,2,3$.
We have that
\begin{align*}
\big(\Psi_{12}^\sharp  +\Psi_{13}^\sharp  +\Psi_{23}^\sharp\big)(  p_0)  \,\sim\,  C\sigma^{-4},
\end{align*}
with $C \neq0$.
\end{lemma}

\begin{proof}
The definition of $\app{\Psi}{j}$ gives  that 
$$
\nabla \app{\Psi}{j} (  p_0) = \kappa_j \app{\xi}{j}. 
$$
Moreover since the $\nabla \app{\Psi}{j}( p_0)$ are lightlike we get from \eqref{eq_Phi_notation} that
$$
\Psi^\sharp_{jk}( p_0) 
= \frac{ 1 }{ 2 \langle  \nabla \app{\Psi}{j} ,\nabla \app{\Psi}{k}  \rangle_g}
= \frac{ 1 }{ 2 \kappa_j \kappa_k \langle  \app{\xi}{j} ,\app{\xi}{i}  \rangle_g}.
$$
By the definitions in Lemma \ref{lem_xis_lin_indep} we have that
\begin{align*}
\langle \app{\xi}{1} ,\app{\xi}{2}  \rangle_\eta =\langle \app{\xi}{1} ,\app{\xi}{3}  \rangle_\eta=  -1 + \sqrt{1-\sigma^2} \sim - \frac{\sigma^2}{2},
\qquad 
\langle \app{\xi}{2} ,\app{\xi}{3}  \rangle_\eta = - 2 \sigma^2. %
\end{align*}
Lemma \ref{lem_xis_lin_indep} asserts moreover that
$$
\kappa_1 \sim  -2(1-s_0) + \frac{\sigma^2}{ 2 }( 3 - s_0),  \quad %
\kappa_2 \sim  1-s_0 \pm \frac{\sigma}{2} \sqrt{ 1 - s_0^2}, \quad 
\kappa_3 \sim  1-s_0 \mp  \frac{\sigma}{2} \sqrt{ 1 - s_0^2}.
$$
Thus by the earlier definitions and the above asymptotics, we have 
at $  p_0$ that
\begin{small}
\begin{align*}
\sum_{ij=12,13,23}\Psi_{ij}^\sharp
= 
\frac{
\kappa_3
\langle \app{\xi}{1} ,\app{\xi}{3}  \rangle_\eta 
\langle \app{\xi}{2} ,\app{\xi}{3}  \rangle_\eta 
+
\kappa_2 
\langle \app{\xi}{1} ,\app{\xi}{2}  \rangle_\eta 
\langle \app{\xi}{2} ,\app{\xi}{3}  \rangle_\eta 
+
\kappa_1
\langle \app{\xi}{1} ,\app{\xi}{2}  \rangle_\eta 
\langle \app{\xi}{1} ,\app{\xi}{3}  \rangle_\eta 
}
{
2\kappa_1 \kappa_2 \kappa_3 
\langle \app{\xi}{1} ,\app{\xi}{3}  \rangle_\eta 
\langle \app{\xi}{2} ,\app{\xi}{3}  \rangle_\eta 
\langle \app{\xi}{1} ,\app{\xi}{2}  \rangle_\eta
}
\end{align*}
\end{small}
Note that for the nominator we have  by the above
\begin{small}
\begin{align*}
\kappa_3
\langle \app{\xi}{1} ,\app{\xi}{3}  \rangle_\eta 
\langle \app{\xi}{2} ,\app{\xi}{3}  \rangle_\eta 
+
\kappa_2 
\langle \app{\xi}{1} ,\app{\xi}{2}  \rangle_\eta 
\langle \app{\xi}{2} ,\app{\xi}{3}  \rangle_\eta 
+
\kappa_1
\langle \app{\xi}{1} ,\app{\xi}{2}  \rangle_\eta 
\langle \app{\xi}{1} ,\app{\xi}{3}  \rangle_\eta 
\sim
\frac{3b }{2}\sigma^2.
\end{align*}
\end{small}
And furthermore for denominator we have that
\begin{small}
\begin{align*}
2\kappa_1 \kappa_2 \kappa_3 
\langle \app{\xi}{1} ,\app{\xi}{3}  \rangle_\eta 
\langle \app{\xi}{2} ,\app{\xi}{3}  \rangle_\eta 
\langle \app{\xi}{1} ,\app{\xi}{2}  \rangle_\eta
\sim 8b^3\sigma^6  
\end{align*}
\end{small}
Putting these together
\begin{small}
\begin{align*}
\sum_{ij=12,13,23}\Psi_{ij}^\sharp
\sim C \sigma^{-4}
\end{align*}
\end{small}
which  proves of the claim.
\end{proof}

\subsection{Uniqueness of $q^2_2$} \label{sec_q_2}

\noindent
In this subsection, we show that $q^2_2 = \tilde q^2_2$ under the assumption that $S = \tilde S$. 
This step is not strictly necessary for the proof of Theorem \ref{thm_thm1}, but it
serves as an illustration of how to use the  WKB interaction solutions. We also 
introduce some notations and other results  used in the later sections and in particular in Section~\ref{sec_q_1}.

Throughout this section, quantities
marked with a tilde correspond to the coefficients $\tilde q_1$, $\tilde q_2$
and $\tilde F$,  while quantities without a tilde correspond to $q_1$, $q_2$
and $F$. 
To establish $q^2_2 = \tilde q^2_2$, we apply the integral identity from Lemma \ref{lem_iid_3rd}, selecting the functions
$\app{v}{j}$ and $\app{w}{ij}$ as the Gaussian beams and WKB interaction solutions
specified in the previous sections.  We then decompose the integral identity
into terms based on their order in the small parameter  $h$. 

Let us assume that $S =\tilde S$. 
It follows that 
$\p^3_{\eps_1 \eps_2 \eps_3} (S - \tilde S)|_{\eps=0}=0$, and consequently by Lemma 
\ref{lem_iid_3rd} that 
\begin{multline} \label{eq_iid_zero}
0= \int_{\Omega_T} 
\tilde q_2 \big( \app{\tilde v}{1} \app{\tilde w}{23} + \app{\tilde v}{2} \app{\tilde w}{13} + \app{\tilde v}{3} \app{\tilde w}{12}\big)   \app{\tilde v}{0} \\
- q_2\big( \app{v}{1} \app{w}{23} + \app{v}{2} \app{w}{13} + \app{v}{3} \app{w}{12} \big) \app{v}{0}     \,dV.
\end{multline}
Here $\app{v}{j},\app{w}{ij},\app{\tilde v}{j}$ and $\app{\tilde w}{ij}$ are as defined in \eqref{eq_GB_reparam_1} and \eqref{eq_wij_2}. 
Let us expand the second term on the right-hand side  
in powers of $h$. The first term is analyzed similarly. Using the formal
Gaussian beams and their error terms, we can write
$$
\app{v}{j} = \app{\hat v}{j} + \app{r}{j} \quad \text{ and }\quad \app{w}{ij} = \app{\hat w}{ij} + \app{w}{ij}_N + \app{r}{ij}. 
$$
Recall that
\[
 \app{\hat v}{j} = e^{i \app{\Psi}{j}/h} (a^{(j)}_0 + a^{(j)}_1 \tfrac{h}{|\kappa_j|} +\cdots) \quad \text{ and } \quad 
\app{\hat w}{jk}=e^{i \app{\Psi}{jk}/h} \big( c^{(jk)}_2 h^2
+ c^{(jk)}_3 h^3 + \cdots \big),
\]
where \eqref{eq_c2c3_def} gives
\[
c^{(jk)}_2= \frac{q_2 a_0^{(j)}a_0^{(k)}}{ \langle  \nabla \app{\Psi}{jk}  ,\nabla \app{\Psi}{jk}   \rangle_g }\ \text{ and } \ 
c^{(jk)}_3 = \frac{ q_2 b_1 - i\square_g \app{\Psi}{jk}   q_2 c_{2}^{(jk)} 
- 2i \langle \nabla \app{\Psi}{jk} , \nabla ( q_2 c_{2}^{(jk)})  \rangle_g  }{ \langle  \nabla
\app{\Psi}{jk}  ,\nabla \app{\Psi}{jk}   \rangle_g }
\]
and 
\[
\app{\Psi}{jk} = \app{\Psi}{j} + \app{\Psi}{k}, \qquad  b_1 =  \big[a_0^{(j)}a_1^{(k)}|\kappa_k|^{-1} + a_0^{(k)}a_1^{(j)} |\kappa_j|^{-1}\big].
\]
Let us first show that 
\begin{equation} \label{eq_I_0}
\begin{aligned}
h^{-2} &\int_{\Omega_T} q_2 \big(  \app{v}{1}  \app{w}{23} +  \app{v}{2}  \app{w}{13} +  \app{v}{3}  \app{w}{12} \big)  \app{v}{0} \,dx \\
&= 
h^{-2} \int_{\Omega_T} q_2 \big( \app{\hat v}{1}  \app{\hat w}{23} + \app{\hat v}{2} \app{\hat w}{13} + \app{\hat v}{3} \app{\hat w}{12} \big) 
\app{\hat v}{0} \,dx  + \mathcal O (h^2).
\end{aligned}
\end{equation}
In order to get \eqref{eq_I_0}  we have firstly  chosen the order $K$ of the Gaussian beams to be
sufficiently large. By the estimates provided in Lemma
\ref{lem_GB_solves_wave_eq}  and Proposition \ref{prop_WKB_source_GBs}, it
follows that
\[
\| \app{r}{j} \|_{L^2(\Omega_T)},\,\, \| \app{r}{ij} \|_{L^2(\Omega_T)} = \mathcal O (h^4).
\]
This implies that the remainder terms do not effect the zeroth and first order terms in \eqref{eq_I_0}
and are therefore included in the $\mathcal O (h^2)$ terms.

Secondly we have that the $\app{w}{ij}_N$ part of the solution does not affect the integrals in
\eqref{eq_iid_zero}  or \eqref{eq_I_0}. Let us check that this is the case.
The solution $\app{w}{ij}_N$ is such that 
$$
\supp{\app{w}{ij}_N} \Subset J^+(p_0),
$$
when $\sigma > 0$ is small enough. 
This is because the solution $\app{w}{ij}_N$ is due to sources located at the intersection points
$p_1,\dots,p_N$ of the geodesics $\app{\gamma}{j}$, $j=1,2,3,4$. 
These intersection points are properly contained in the causal future $J^+(p_0)$ or in the past causal set
$J^-(B(q_1,r))$. The intersection points are in fact such that 
$$
p_1,\dots,p_N \notin W_r = J^-(p_0) \cap J^+(q_1) \setminus B(q_1,r),
$$
for small $\sigma > 0$, this follows directly from  Lemmas \ref{lem_gamma_1_future} and \ref{lem_gamma_1_no_intersections}. 
The solution $\app{v}{0}$ can on the other hand be thought as propagating backwards in time from a source located at $q_0$,
so that $\supp(\app{v}{0}) \subset J^-_\delta(q_0)$, when the source is supported in a small enough neighbourhood of $q_0$.
It follows now from Lemma \ref{lem_gamma_1_future}, that
$$
\supp(\app{w}{ij}_N) \cap \supp(\app{v}{0}) = \emptyset,
$$
for a small enough $\sigma >0$. The product $\app{w}{ij}_N\app{v}{0}$ is thus zero in the integration region
in \eqref{eq_I_0} and \eqref{eq_iid_zero}.
The only terms that affect the zero and first order terms in \eqref{eq_I_0} are of the form $ \app{v}{j} \app{w}{jk}$,
and we thus see that \eqref{eq_I_0} holds.

Next, we give more explicit expressions for the zeroth and first-order terms in
\eqref{eq_I_0}. We write 
\begin{equation} \label{eq_I0_I1_R1}
\begin{aligned}
h^{-2} \int_{\Omega_T} q_2 \big( \app{\hat v}{1}  \app{\hat w}{23} + \app{\hat v}{2} \app{\hat w}{13} + \app{\hat v}{3} \app{\hat w}{12} \big) \hat v^{(0)}
= I_0 + I_1 + R_1,
\end{aligned}
\end{equation}
where $I_0$ is a zero order term and $I_1$ and $R_1$ are first order terms.
Using the amplitude expansions given in \eqref{eq_GB_amplitude}
and \eqref{eq_cj}, the zeroth-order term 
$I_0$ in $h$ can be written as
\begin{align}\label{eq_I_0_def}
I_0 &= \int_{\Omega_T} q^2_2 e^{ \frac{i}{h} (\app{\Psi}{0} +\app{\Psi}{1} + \app{\Psi}{2} + \app{\Psi}{3}) } (\Psi^\sharp_{12}+\Psi^\sharp_{13}+\Psi^\sharp_{23})
\app{\bar a}{1}_{0} \ap{2}_{0} \ap{3}_{0} \app{\bar a}{0}_{0} \,dx, 
\end{align}
where
\begin{equation} \label{eq_PsiSharp}
\begin{aligned}
 \Psi^\sharp_{jk} = \frac{ 1 }{ \langle  \nabla \app{\Psi}{jk} ,\nabla \app{\Psi}{jk}  \rangle_g },
\end{aligned}
\end{equation}
for $j,k =1,2,3$. 
Note that the factor $q_2^2$ in \eqref{eq_I_0_def} arises because one factor of
$q_2$ originates from the expansion of the solutions $\app{\hat w}{ij}$, see \eqref{eq_c2c3_def}. Additionally, observe  that
$\Psi^\sharp_{jk}$ depends only on the geometry of $N$. 

Next we will consider the first order terms. 
These are split into two parts $I_1$ and $R_1$ and originate from the terms $h^{-2}\app{\hat v}{i} \app{\hat w}{jk} \app{\hat v}{0}$. 
The first order contribution  of the $h^{-2}\app{\hat v}{i} \app{\hat w}{jk} \app{\hat v}{0}$ comes from the first and the second term on the right of
\begin{equation} \label{eq_v_w_parts}
\begin{aligned}
h^{-2} \app{\hat v}{i} \app{\hat w}{jk} \app{\hat v}{0} = e^{i\sum_{i,j,k}\app{\Psi}{l}/h}
\big[ \ap{i}_0 \app{c}{jk}_2  + [\ap{i}_0 \app{c}{jk}_3
+  |\kappa_i|^{-1}\ap{i}_1 \app{c}{jk}_2 ]  h + \mathcal O (h^2)\big]  \app{\hat v}{0} ,  
\end{aligned}
\end{equation}
where  $\app{c}{jk}$ denotes the coefficients of $\app{\hat w}{jk}$ given in  Lemma \ref{lem_WKB_source} in
equation \eqref{eq_c2c3_def}.
The integral $I_1$ below includes the terms coming from $\ap{i}_1 \app{c}{jk}_2$ and the part 
coming from $b_1$ in the product $\ap{i}_0 \app{c}{jk}_3$, multiplied by the zero order part of $\app{\hat v}{0}$.
The integral $I_1$ also contains $\ap{i}_0 \app{c}{jk}_2$ multiplied by the first order part of $\app{\hat v}{0}$.
We have more specifically that
\begin{small}
\begin{equation*} %
\begin{aligned}
I_1 &= \int_{\Omega_T} q^2_2 e^{ \frac{i}{h} (\app{\Psi}{0} +\app{\Psi}{1} + \app{\Psi}{2} + \app{\Psi}{3}) }  \\
\Big[ &\Psi^\sharp_{23}
\big(
   \tfrac{1}{|\kappa_3|}  \app{\bar a}{1}_{0} \ap{2}_{0} \ap{3}_{1} \app{\bar a}{0}_{0} 
+  \tfrac{1}{|\kappa_2|}  \app{\bar a}{1}_{0} \ap{2}_{1} \ap{3}_{0} \app{\bar a}{0}_{0} 
+  \tfrac{1}{|\kappa_0|}  \app{\bar a}{1}_{0} \ap{2}_{0} \ap{3}_{0} \app{\bar a}{0}_{1}
+  \tfrac{1}{|\kappa_1|}  \app{\bar a}{1}_{1} \ap{2}_{0} \ap{3}_{0} \app{\bar a}{0}_{0} \big) \\
+&\Psi^\sharp_{13} 
\big(
   \tfrac{1}{|\kappa_3|}  \app{\bar a}{1}_{0} \ap{2}_{0} \ap{3}_{1} \app{\bar a}{0}_{0} 
+  \tfrac{1}{|\kappa_1|}  \app{\bar a}{1}_{1} \ap{2}_{0} \ap{3}_{0} \app{\bar a}{0}_{0} 
+  \tfrac{1}{|\kappa_0|}  \app{\bar a}{1}_{0} \ap{2}_{0} \ap{3}_{0} \app{\bar a}{0}_{1} 
+  \tfrac{1}{|\kappa_2|}  \app{\bar a}{1}_{0} \ap{2}_{1} \ap{3}_{0} \app{\bar a}{0}_{0} \big)  \\
+ &\Psi^\sharp_{12} 
\big(
   \tfrac{1}{|\kappa_2|} \app{\bar a}{1}_{0} \ap{2}_{1} \ap{3}_{0} \app{\bar a}{0}_{0} 
+  \tfrac{1}{|\kappa_1|} \app{\bar a}{1}_{1} \ap{2}_{0} \ap{3}_{0} \app{\bar a}{0}_{0} 
+  \tfrac{1}{|\kappa_0|} \app{\bar a}{1}_{0} \ap{2}_{0} \ap{3}_{0} \app{\bar a}{0}_{1}
+  \tfrac{1}{|\kappa_3|} \app{\bar a}{1}_{0} \ap{2}_{0} \ap{3}_{1} \app{\bar a}{0}_{0}\big)\Big] 
\,dx. 
\end{aligned}
\end{equation*}
\end{small}
Rearranging we get that
\begin{equation} \label{eq_I1}
\begin{aligned}
I_1 = &\int_{\Omega_T} q^2_2 e^{ \frac{i}{h} (\app{\Psi}{0} +\app{\Psi}{1} + \app{\Psi}{2} + \app{\Psi}{3}) } 
(\Psi^\sharp_{23} +\Psi^\sharp_{13} +\Psi^\sharp_{12}) \\
&\big(
   \tfrac{1}{|\kappa_0|} \app{\bar a}{0}_{1} \app{\bar a}{1}_{0} \ap{2}_{0} \ap{3}_{0}
+  \tfrac{1}{|\kappa_1|} \app{\bar a}{0}_{0} \app{\bar a}{1}_{1} \ap{2}_{0} \ap{3}_{0}
+  \tfrac{1}{|\kappa_2|} \app{\bar a}{0}_{0} \app{\bar a}{1}_{0} \ap{2}_{1} \ap{3}_{0}
+  \tfrac{1}{|\kappa_3|} \app{\bar a}{0}_{0} \app{\bar a}{1}_{0} \ap{2}_{0} \ap{3}_{1}
\big)
\,dx 
\end{aligned}
\end{equation}
Next we specify $R_1$. Note that the $\ap{i}_0 \app{c}{jk}_3$ term  in \eqref{eq_v_w_parts} also contains a second part  
coming from the part of  $\app{c}{jk}_2$ which contains the terms other than $b_1$. 
These we put in the $R_1$ integral, given by
\begin{equation} \label{eq_R1}
\begin{aligned}
R_1 = - i\int_{\Omega_T} e^{ \frac{i}{h} \sum_{k=0}^3 \app{\Psi}{k}}
\sum_{\substack{ijk = 123, \\ 213,312}} \Psi^\sharp_{jk} q_2 \app{\check a}{i}_{0}  \app{\bar a}{0}_{0}   
\big(\square_g \app{\Psi}{jk} q_2 \app{ c}{jk}_2 + 2 \langle \nabla \app{\Psi}{jk} , \nabla (q_2 \app{ c}{jk}_2)  \rangle_g\big) \,dx . 
\end{aligned}
\end{equation}
Here $\app{\check a}{i}_0$  is defined as %
\begin{align*}
\app{\check a}{i}_0 := 
\begin{cases}
\app{\bar a}{i}_0, \quad \text{ for } i=0,1, \\
\app{a}{i}_0, \quad\text{ for } i=2,3.
\end{cases}
\end{align*}
Recall also that  $\ap{j}_0$, $\app{c}{jk}_2$ and $\app{\Psi}{j}$ are determined by the  geometry of the manifold.
It follows  that $R_1$ depends on the geometry and $q^2_2$. Thus once we know that $q^2_2 = \tilde q^2_2$, 
then, since in \eqref{eq_R1} we can write 
$$
q_2 \big \langle \nabla \app{\Psi}{jk} , \nabla (q_2 \app{ c}{jk}_2)  \big \rangle_g
=
\big \langle \nabla \app{\Psi}{jk} ,\, q_2^2 \nabla \app{ c}{jk}_2 +  \tfrac{1}{2} \nabla(q^2_2) \app{ c}{jk}_2   \big \rangle_g.
$$
It follows that $R_1 = \tilde R_1$, and the remainder terms $R_1$ are canceled out in \eqref{eq_iid_zero}.
This will be used later in section \ref{sec_q_1}.

\medskip
\noindent
We now turn to analyzing the terms $I_0$ and $\tilde I_0$, which will allow us to determine $q^2_2$.
A variant of the method of stationary phase will be used for this.
For this we state the following proposition which is a direct consequence of  Theorem 7.7.5 in \cite{Ho81}.

\begin{proposition}\label{prop_stationary_phase}
Let $K \subset \R^n$ be a compact set, $X$ an open neighborhood
of $K$, and let $u \in C^\infty_0(K)$ and $\psi\in C^\infty(X)$.  
If 
\begin{enumerate}
\item $\psi(x_0) = 0$ and  $\operatorname{Im} \psi\geq 0$ in $X$.
\item $\nabla \psi(x_0) = 0$ and $\nabla \psi\neq 0$ in $K\setminus \{x_0\}$.
\item $\det H\psi(x_0) \neq 0$, where $H\psi$ is the Hessian.
\end{enumerate}
Then
\begin{equation} \label{eq_stationary_phase}
\begin{aligned}
\Bigg |
\int_X e^{i \psi(x)/h }u(x) \,dx -    \frac{(2\pi/h)^{n/2} }{ \det [i H\psi(x_0)]^{1/2} } u(x_0)
\Bigg |
= \mathcal O(h).
\end{aligned}
\end{equation}
\end{proposition}

\noindent
We now use Proposition \ref{prop_stationary_phase} and the zero order part $I_0$ of \eqref{eq_I_0} to determine $q^2_2$,
as  the following result illustrates.

\begin{proposition} \label{prop_q2}
Assume that $S = \tilde S$. Then $q^2_2 = \tilde q^2_2$.
\end{proposition}

\begin{proof}
Since  $S = \tilde S$ implies that $\p^3_{\eps_1 \eps_2 \eps_3} (S - \tilde S)|_{\eps=0}=0$,
we see from \eqref{eq_iid_zero} and \eqref{eq_I_0_def} above that
\begin{align*}
0=I_0 - \tilde I_0 = &\quad \int_{\Omega_T} 
e^{ \frac{i}{h} (\app{\Psi}{0} +\app{\Psi}{1} + \app{\Psi}{2} + \app{\Psi}{3}) } q^2_2
(\Psi^\sharp_{12} + \Psi^\sharp_{13} + \Psi^\sharp_{23})
\app{\bar a}{1}_{0} \ap{2}_{0} \ap{3}_{0} \app{\bar a}{0}_{0} \,dx\\
&-
\int_{\Omega_T} 
e^{ \frac{i}{h} (\app{\tilde \Psi}{0} + \app{\tilde \Psi}{1} +  \app{\tilde \Psi}{2} +  \app{\tilde \Psi}{3}) } \tilde q^2_2
(\tilde\Psi^\sharp_{12} + \tilde\Psi^\sharp_{13} + \tilde\Psi^\sharp_{23})  
\app{\bar{\tilde a}}{1}_{0} \app{\tilde a}{2}_{0} \app{\tilde a}{3}_{0} \app{\bar{\tilde a}}{0}_{0} 
\,dx,
\end{align*}
Moreover we have that $\ap{j}_{0} = \app{\tilde a}{j}_{0}$ and $\Psi^\sharp_{ij} = \tilde \Psi^\sharp_{ij}$, as these 
are determined by the parameters $\app{\xi}{j}$ and $\kappa_j$ for the Gaussian beams, the  geometry of the manifold,
and the source terms specified in \eqref{eq_specific_GB_source_prob}, and since
source terms $\app{f}{j}$, which are identical for $\app{v}{j}$ and $\app{\tilde v}{j}$. The last part follows 
from the fact that the potentials are fully determined on the measurement set $U$.
It follows that
\begin{align*}
0=I_0 - \tilde I_0 = &\quad \int_{X} 
e^{ \frac{i}{h} (\app{\Psi}{0} +\app{\Psi}{1} + \app{\Psi}{2} + \app{\Psi}{3}) } (q^2_2- \tilde q^2_2)
(\Psi^\sharp_{12} + \Psi^\sharp_{13} + \Psi^\sharp_{23})
\app{\bar a}{1}_{0} \ap{2}_{0} \ap{3}_{0} \app{\bar a}{0}_{0} \,dx
\end{align*}
where $X$ is a neighborhood of the support of the product of the amplitude terms, which we choose as the set $K$
in Proposition \ref{prop_stationary_phase}. 
We will now apply Proposition \ref{prop_stationary_phase}. We need 
to thus check the conditions of Proposition \ref{prop_stationary_phase}. 
Let
$$
\psi:= \app{\Psi}{0} +\app{\Psi}{1} + \app{\Psi}{2} + \app{\Psi}{3} 
= \kappa_0 \overline{\app{\Phi}{0}} + \kappa_1\overline{\app{\Phi}{2}} +  \kappa_2  \app{\Phi}{2} +  \kappa_3\app{\Phi}{3}. 
$$
Firstly we have that $\psi( p_0) = \psi(x_0) = 0$, since the phase functions $ \app{\Phi}{j}$ vanish on the 
on the corresponding geodesics because of \eqref{eq_Phi}.  We also have that 
$$
\nabla \psi( x_0) =\kappa_0\app{\xi}{0} + \kappa_1\app{\xi}{1} +  \kappa_2 \app{\xi}{2} +  \kappa_3 \app{\xi}{3} = 0,
$$
by \eqref{eq_lin_dependent}. Furthermore for $x \in X$
\begin{align*}
\Im \psi =  \kappa_0 \Im \overline{\app{\Phi}{0}} + \kappa_1 \Im \overline{\app{\Phi}{1}} +  \kappa_2  \app{\Im\Phi}{2} + \kappa_3   \app{\Im\Phi}{3}  \geq 0, 
\end{align*}
which holds, because of \eqref{eq_Phi_cond}  and  $\kappa_0,\kappa_1 <0$, $\kappa_2,\kappa_3 > 0$.
For the Hessian we have similarly that 
$$
\Im H\psi(x_0) = (\kappa_0 \Im H \overline{\app{\Phi}{0}} + \kappa_1\Im  H \overline{\app{\Phi}{1}} 
+  \kappa_2 \Im H \app{\Phi}{2} + \kappa_3 \Im H\app{\Phi}{3})(x_0)  >  0,
$$
because of Lemma \ref{lem_Riccati} and  since $\kappa_0,\kappa_1 >0$, $\kappa_2,\kappa_3 < 0$. 
Finally note that $x_0$ is a local minima for $\Im \psi$, since $H \Im\psi(x_0) > 0$,
so that
\begin{align*}
\Im \nabla \psi
> 0,
\quad \text{ in } X \setminus \{  p_0 \},
\end{align*}
and hence $\nabla \psi \neq 0$ in $X \setminus \{  p_0 \}$.
We thus see that $\psi$ satisfies the conditions of Proposition \ref{prop_stationary_phase}.
Thus by  Proposition \ref{prop_stationary_phase} we  now have  that
$$
0 = h^{n/2}(I_0 - \tilde I_0) \to C \big[(\Psi^\sharp_{12} + \Psi^\sharp_{13} + \Psi^\sharp_{23})(q^2_2 - \tilde q^2_2)\big](x_0), 
$$
as $h \to 0$,
with some non-zero constant $C \in \C$. Lemma \ref{lem_Phis_nonzero} implies now that
$$
(\Psi^\sharp_{12} + \Psi^\sharp_{13} + \Psi^\sharp_{23})(x_0) \neq 0.
$$
It follows that  $q^2_2(x_0) = \tilde q^2_2(x_0)$.
\end{proof}

\subsection{Uniqueness of $q_2$} \label{sec_q2_only}

In this section we will show that $q_2 = \tilde q_2$. We will do this by using  integral identity 
\eqref{eq_iid_2}  related to the higher order linearization, which allows us to determine $q_2^3$.  
For this we need to modify our choices of  $\app{v}{i}$ and $\app{w}{ij}$, and choose the solutions $\app{w}{ijk}$ 
that appear in the integral identity. 

\medskip
\noindent
The solutions appearing in the integral equation \eqref{eq_iid_2} are built up using five distinct 
Gaussian beams $\app{v}{i}$, $i=0,1,2,3,4$. We thus need to specify an additional Gaussian beam
$\app{v}{4}$ traveling along the geodesic $\app{\gamma}{4}$ defined in section \ref{sec_geom_setup}.
The solutions $\app{v}{i}$ and $\app{w}{ij}$ are built up as in  subsection \ref{sec_GBs_and_geodesics}. Here we however 
need  to specify an additional Gaussian beam $\app{v}{4}$.
We also make use of the higher order WKB interaction solutions of section \ref{sec_high_WKB_interaction} and
need to define them. Furthermore we need to redefine $\kappa_j$ parameters.

Recall that in the normal coordinates used section \ref{sec_geom_setup} and in Lemma \ref{lem_xis_lin_indep}, 
the vector $\app{\xi}{4}$  could be written as 
\begin{equation} \label{eq_xi4}
\begin{aligned}
\app{\xi}{4} :=  \big( \sqrt{ 1 + \sigma^2 } , 1 , \sigma ,0, \dots ,0   \big),
\end{aligned}
\end{equation}
Note that this choice is such  that $\app{\xi}{4}$ is  lightlike. Recall also 
that $\app{\xi}{4}$  depends on the parameter $\sigma$, and that for small $\sigma>0$,
$$
\app{\xi}{4} \sim (1,1,0,\dots,0) = \app{\xi}{1},
$$
as is the case for the geodesics $\app{\xi}{j}$, $j= 2,3$.
Next we will formulate an analogue of Lemma \ref{lem_xis_lin_indep}, where we add the vector 
$\app{\xi}{4}$.

\begin{lemma} \label{lem_xis_lin_indep_2}
Let $\app{\xi}{j}$, $j=0,1,2,3$ be the lightlike vectors specified in Lemma \ref{lem_xis_lin_indep} and 
let $\app{\xi}{4}$ be given by \eqref{eq_xi4}.
We can choose $\tilde \kappa_i \neq 0$, $i=0,..,4$ such that
\begin{equation} \label{eq_lin_dependent_2}
\begin{aligned}
\tilde\kappa_0 \app{\xi}{0} +
\tilde\kappa_1 \app{\xi}{1} + \tilde\kappa_2 \app{\xi}{2} + 
\tilde\kappa_3 \app{\xi}{3} + \tilde\kappa_4 \app{\xi}{4} = 0.
\end{aligned}
\end{equation}
where the $\tilde \kappa_i$ are such that
$\tilde \kappa_0 = -\sigma^2$,  $ \tilde \kappa_4 = \sigma^2$ and
\begin{align*}
\tilde\kappa_1 \sim  -2(1-s_0) + \frac{\sigma^2}{ 2 }( 3 - s_0),  \quad %
\tilde\kappa_2 \sim  1-s_0 \mp \frac{\sigma}{2} \sqrt{ 1 - s_0^2}, \quad 
\tilde\kappa_3 \sim  1-s_0 \pm  \frac{\sigma}{2} \sqrt{ 1 - s_0^2},
\end{align*}
where the sign corresponds to sign in the definition of $\app{\xi}{0}$.
\end{lemma}

\begin{proof}
This can be proven by Gaussian similarly as Lemma \ref{lem_xis_lin_indep} using Gaussian elimination.
This gives explicit  and somewhat lengthy expressions for the coefficients $\tilde \kappa_i$.
These are listed in the appendix, see subsection \ref{sec_explicit}.
\end{proof}

\medskip 
\noindent
We now pick the Gaussian beams $ \app{v}{j}$, $j=0,1,2,3$, as in \eqref{eq_specific_GB_source_prob},
and a new solutions $\app{v}{4}$ that also solves \eqref{eq_specific_GB_source_prob}.
More specifically we define $\app{v}{j}$ for $j=2,3,4$ as
\begin{equation} \label{eq_GB_reparam_1_2}
\begin{aligned}
\app{v}{j} 
= 
e^{i \app{|\tilde\kappa_j| \Phi}{j}/h}\sum_{k=0}^K \ap{j}_k \Big( \frac{h} {|\tilde\kappa_j|}\Big)^k  + \app{r}{j}, 
\quad j=2,3,4.
\end{aligned}
\end{equation}
For $j=0,1$ we use the complex conjugates $ \app{\bar v}{j}$, since $\tilde\kappa_0,\tilde\kappa_1 < 0$.
We drop the complex conjugate from our notation for clarity.
That is  we set   
\begin{equation} \label{eq_GB_reparam_2_2}
\begin{aligned}
\app{v}{j} 
= 
e^{-i \app{|\tilde\kappa_j| \overline \Phi}{j} / h }\sum_{k=0}^K \app{\bar a}{j}_k \Big( \frac{h} {|\tilde\kappa_j|}\Big)^k  + \app{\bar r}{j}, 
\quad j=0,1.
\end{aligned}
\end{equation}
Where $\Phi$ are constructed, so that 
$$
\nabla_g \app{\Phi}{j}(  p_0) = \app{\xi}{j}.
$$
We will use the notation
\begin{equation} \label{eq_Upsilon_2}
\begin{aligned}
\Upsilon^{(j)} :=
\begin{cases}
-|\tilde\kappa_j| \overline \Phi^{(j)}, \quad &j = 0,1.  \\
\phantom{-}|\tilde\kappa_j| \Phi^{(j)}, \quad  &j = 2,3,4. 
\end{cases}
\end{aligned}
\end{equation}
We will again drop the absolute value from the notation, since
\begin{align*}
\Upsilon^{(j)} = \tilde\kappa_j \Phi^{(j)}, \qquad \text{ for small $\sigma > 0$}.
\end{align*}
We define $\Upsilon^{(ij)}$ and $\Upsilon_{ij}^\sharp$ are defined correspondingly as
$\Psi^{(ij)}$ and $\Psi_{ij}^\sharp$, and contain implicitly factors of $ \tilde \kappa_j$.
Furthermore we define  
\begin{equation} \label{eq_Upsilon_sharp_ijk}
\begin{aligned}
\Upsilon^\sharp_{ijk} := \frac{1}{| \nabla \app{\Upsilon}{i} + \nabla \app{\Upsilon}{j} + \nabla \app{\Upsilon}{k} |^2_g},
\end{aligned}
\end{equation}
which appears in our formulas because of the construction in Lemma \eqref{lem_WKB_source_w_v}.

\medskip
\noindent
We furthermore define the WKB interaction solutions $\app{w}{ij}$ and $\app{w}{ijk}$ in accordance
with \eqref{eq_wij} and \eqref{eq_wij_2}, as solutions to
\begin{equation} \label{eq_wij_3}
\begin{aligned}
(\square_g + Q) \app{w}{ij} = q_2\app{v}{i}\app{v}{j}, \qquad (\square_g + Q) \app{w}{ijk} = q_2 \app{v}{i} \app{w}{jk}, 
\end{aligned}
\end{equation}
where $Q := q_1 + 2q_2 u_0$. 
The higher order WKB interaction solution $\app{w}{ijk}$ is well-defined when \eqref{eq_nabla_Phi_nonzero_2} holds. 
The following two lemmas shows that these conditions hold near the point $p_0$.

\begin{lemma} \label{lem_xis_asymp}
We have the following asymptotics 
\begin{equation} \label{eq_xi_asymp}
\begin{aligned}
&\big \langle \app{\xi}{0} , \app{\xi}{1}  \big \rangle_\eta, 
 \big \langle \app{\xi}{0} , \app{\xi}{2}  \big \rangle_\eta, 
 \big \langle \app{\xi}{0} , \app{\xi}{3}  \big \rangle_\eta, 
 \big \langle \app{\xi}{0} , \app{\xi}{4}  \big \rangle_\eta \sim s_0-1, 
\end{aligned}
\end{equation}
and furthermore 
\begin{equation} \label{eq_xi_asymp}
\begin{aligned}
&\big \langle \app{\xi}{1} , \app{\xi}{2}  \big \rangle_\eta \sim -\tfrac{1}{2}\sigma^{2}, \quad 
 \big \langle \app{\xi}{1} , \app{\xi}{3}  \big \rangle_\eta \sim -\tfrac{1}{2}\sigma^{2}, \quad
 \big \langle \app{\xi}{1} , \app{\xi}{4}  \big \rangle_\eta \sim -\tfrac{1}{2}\sigma^{2}, \\ 
&\big \langle \app{\xi}{2} , \app{\xi}{3}  \big \rangle_\eta \sim -2\sigma^{2}, \quad 
 \big \langle \app{\xi}{2} , \app{\xi}{4}  \big \rangle_\eta \sim -\tfrac{1}{8}\sigma^{6}, \quad 
 \big \langle \app{\xi}{3} , \app{\xi}{4}  \big \rangle_\eta \sim -2\sigma^{2}.
\end{aligned}
\end{equation}
as $\sigma \to 0$. %
\end{lemma}
\begin{proof}
These are straight forward to compute from the definitions in Lemma \ref{lem_xis_lin_indep} and \eqref{eq_xi4}
by using Taylor expansions.
\end{proof}

\noindent
Next we will derive some further results that are needed to investigate the asymptotics of certain expressions
involving  the $\app{\xi}{i}$ and $\tilde \kappa_i$, in terms of the parameter $\sigma$.

\begin{lemma} \label{lem_Upsilon_asymptotics_1}
We have the following asymptotics
\begin{equation} \label{eq_Upsilon_asymp_1}
\begin{aligned}
\Upsilon^\sharp_{24}  &\sim  C \sigma^{-8}, \\
\Upsilon^\sharp_{ij}  &\sim  C \sigma^{-2}, \quad ij = 12,13,23, \\
\Upsilon^\sharp_{ij}  &\sim  C \sigma^{-4}, \quad ij = 14,34, 
\end{aligned}
\end{equation}
where the constant $C \neq 0$.
We also have that 
\begin{equation} \label{eq_Upsilon_asymp_2}
\begin{aligned}
\Upsilon^\sharp_{234},
 \Upsilon^\sharp_{134},
\Upsilon^\sharp_{124} \,\sim\, C \sigma^{-2}, \qquad 
\Upsilon^\sharp_{123} \,\sim\, C \sigma^{-4}, \qquad 
\end{aligned}
\end{equation}
where the constant $C \neq 0$.
\end{lemma}

\begin{proof}
We can use Lemmas \ref{lem_xis_lin_indep_2} and \ref{lem_xis_asymp}.
This gives that
\begin{align*}
(\Upsilon^\sharp_{24})^{-1} = 
\big \langle\nabla \app{\Upsilon}{2}(  p_0), \, \nabla \app{\Upsilon}{4}( p_0)  \big \rangle_g 
=  \tilde \kappa_1 \tilde \kappa_4 \langle\app{\xi}{2} , \,  \app{\xi}{4}  \rangle_\eta 
\sim (C + \sigma^2 )\sigma^2 C \sigma^6  
\,\sim\, C\sigma^8, 
\end{align*}
For some constant $C \neq 0$. 
One sees similarly that  $(\Upsilon^\sharp_{ij})^{-1} \sim C \sigma^2$, where $C \neq 0$,
$ij=12,13,23$, and 
that  $(\Upsilon^\sharp_{ij})^{-1} \sim C \sigma^4$, where $C \neq 0$,
$ij = 24,34$. Proving \eqref{eq_Upsilon_asymp_1}.
The other part of the claim follows from the above by observing that from \eqref{eq_Upsilon_sharp_ijk}
we get e.g. that 
$-\tilde\kappa_0 \app{\xi}{0} - \tilde\kappa_1 \app{\xi}{1} = \tilde\kappa_2 \app{\xi}{2} +
\tilde\kappa_3 \app{\xi}{3} + \tilde\kappa_4 \app{\xi}{4}$, so that
\begin{align*}
(\Upsilon^\sharp_{234} )^{-1}
=  -\langle  \tilde \kappa_0 \app{\xi}{0} , \,  \tilde \kappa_1 \app{\xi}{1}  \rangle_\eta
\,\sim\, C \sigma^{2},
\end{align*}
where $C \neq 0$. The other cases are similar.
\end{proof}

\noindent
Finally we will need the following lemma.

\begin{lemma}\label{lem_Upsilon_asymptotics} We have that 
\begin{small}
\begin{align*}
\mathbf \Upsilon^\sharp 
&:=
\big[\Upsilon^\sharp_{234}(\Upsilon^\sharp_{23} + \Upsilon^\sharp_{24} + \Upsilon^\sharp_{34})
+\Upsilon^\sharp_{134}(\Upsilon^\sharp_{13} + \Upsilon^\sharp_{14} + \Upsilon^\sharp_{34}) \\
&\quad +\Upsilon^\sharp_{124}(\Upsilon^\sharp_{12} + \Upsilon^\sharp_{14} + \Upsilon^\sharp_{24})
+\Upsilon^\sharp_{123}(\Upsilon^\sharp_{12} + \Upsilon^\sharp_{13} + \Upsilon^\sharp_{23}) 
 +\Upsilon^\sharp_{14}\Upsilon^\sharp_{23}
+\Upsilon^\sharp_{24}\Upsilon^\sharp_{13}
+\Upsilon^\sharp_{34}\Upsilon^\sharp_{12}\big] 
(p_0) \\ 
&\sim\,\, C \sigma^{-10}, 
\end{align*}
\end{small}
as $\sigma \to 0$, and  where $C \neq 0$.
\end{lemma}
 
\begin{proof}
Note that the term $\Upsilon^\sharp_{24}$ has the most extreme asymptotics
according to Lemma \ref{lem_Upsilon_asymptotics_1}. There are  three terms involving this factor.
Firstly we have
\begin{align*}
\Upsilon^\sharp_{234}\Upsilon^\sharp_{24}  \sim 
\Upsilon^\sharp_{01}\Upsilon^\sharp_{24}  \sim  C \sigma^{-8}, 
\end{align*}
where we used that
$\tilde\kappa_0 \app{\xi}{0} + \tilde\kappa_1 \app{\xi}{1} + \tilde\kappa_2 \app{\xi}{2} +
\tilde\kappa_3 \app{\xi}{3} + \tilde\kappa_4 \app{\xi}{4} = 0$. Similarly 
we have for the second term involving $\Upsilon^\sharp_{24}$, that
\begin{align*}
\Upsilon^\sharp_{124}\Upsilon^\sharp_{24}  \sim
\Upsilon^\sharp_{03}\Upsilon^\sharp_{24}  \sim C \sigma^{-8}. 
\end{align*}
For the final term with $\Upsilon^\sharp_{24}$, we get that
\begin{align*}
\Upsilon^\sharp_{13}\Upsilon^\sharp_{24}  \sim  C \sigma^{-10}, 
\end{align*}
The claim follows from these.
\end{proof}

\noindent
We will now show that $q_2 = \tilde q_2$. We will do this by using the integral identity 
\eqref{eq_iid_2}  for the fourth order linearization. From this we will obtain that $q_2^3=\tilde q_2^3$,
and since $q_2$ and $\tilde q_2$ are real we get that $q_2 = \tilde q_2$.

\medskip
\noindent
The proof that $q_2^3=\tilde q_2^3$, is similar to that of $q_2^2 = \tilde q_2^2$
in subsection \ref{sec_q_2}. Here we need however to use the 
higher order WKB interaction solutions $\app{w}{ijk}$ specified in \eqref{eq_wij_3}, and the  $\app{v}{i}$
and $\app{w}{jk}$ specified in \eqref{eq_GB_reparam_1_2}, \eqref{eq_GB_reparam_2_2} and  \eqref{eq_wij_3} earlier in this section.

\medskip
\noindent
Given that $S =\tilde S$ we have that
$\p^4_{\eps_1 \eps_2 \eps_3 \eps_4} (S - \tilde S)|_{\eps=0}=0$, and Lemma \ref{lem_iid_4th},
gives us that
\begin{equation} \label{eq_J_minus_J}
\begin{aligned}
J -\tilde J 
:= \int_{\Omega_T} &2\tilde q_2 \Big( \app{\tilde w}{14} \app{\tilde w}{23} + \app{\tilde w}{24} \app{\tilde w}{13} + \app{\tilde w}{34} \app{\tilde w}{12} \\
&\quad+ \app{\tilde v}{1} \app{\tilde w}{234} + \app{\tilde v}{2} \app{\tilde w}{134}
       + \app{\tilde v}{3} \app{\tilde w}{124}+ \app{\tilde v}{4} \app{\tilde w}{123}\Big) \app{\tilde v}{0}\\
-&2 q_2 \Big ( \app{w}{14}  \app{w}{23} +  \app{w}{24}  \app{w}{13} +  \app{w}{34}  \app{w}{12} \\
&\quad+ \app{v}{1}  \app{w}{234} +  \app{v}{2}  \app{w}{134} +  \app{v}{3}  \app{w}{124} + \app{v}{4} \app{w}{123}\Big) \app{v}{0} \, dV =0. 
\end{aligned}
\end{equation}
To determine $q^3_2$ we need to pick the lowest order term  in $h$ in this expressions, which will be of  order $4$.
Let us show that the lowest order parts  are of the form
\begin{equation} \label{eq_J0}
\begin{aligned}
J_0 &:= \int_{\Omega_T} q^3_2 e^{ \frac{i}{h} (\app{\Upsilon}{0} +\app{\Upsilon}{1} + \app{\Upsilon}{2} + \app{\Upsilon}{3}+ \app{\Upsilon}{4}) } 
\mathbf \Upsilon^\sharp
\app{\bar a}{1}_{0} \ap{2}_{0} \ap{3}_{0} \ap{4}_{0} \app{\bar a}{0}_{0} \,dx.
\end{aligned}
\end{equation}
The reasoning is analogous to that in section \ref{sec_q_2}. We can expand the solutions as
$$
\app{v}{i} = \app{\hat v}{i} + \app{r}{i}, 
\qquad \app{w}{ij} = \app{\hat w}{ij} + \app{w}{ij}_N + \app{r}{ij},
\qquad \app{w}{ijk} = \app{\hat w}{ijk} + \app{w}{ijk}_N + \app{r}{ijk}. 
$$
Firstly we can show that
\begin{equation} \label{eq_term1}
\begin{aligned}
\int_{\Omega_T} 2q_2 \app{v}{i} \app{w}{jkl} \app{v}{0}\, dV 
=
\int_{\Omega_T} 2q_2 \app{\hat v}{i} \app{\hat w}{jkl}  \app{\hat v}{0}\, dV  + \mathcal O (h^5).
\end{aligned}
\end{equation}
To derive this we  use on one hand the  estimates provided in Lemma
\ref{lem_GB_solves_wave_eq}  and Propositions \ref{prop_WKB_source_GBs} and \ref{prop_WKB_source_GBs_2}, from which
it follows that
$$
\| \app{r}{i} \|_{L^2(\Omega_T)},
\quad \| \app{r}{ij} \|_{L^2(\Omega_T)},
\quad \| \app{r}{ijk} \|_{L^2(\Omega_T)} = \mathcal O (h^5).
$$
This shows that the lowest order part in \eqref{eq_term1} does not contain any terms with a factor $\app{r}{i}, \app{r}{ij}$ or $\app{r}{ijk}$. 
On the other hand we know likewise as in section \ref{sec_q_2} that 
$$
\supp(\app{w}{ij}_N) \cap \supp(\app{v}{0}) = \emptyset,
$$
for a small enough $\sigma >0$. 
And hence the integral in \eqref{eq_term1} does not contain any terms with a factor $\app{w}{ij}_N$. 
We have furthermore  that 
$$
\supp(\app{w}{ijk}_N) \cap \supp(\app{v}{0}) = \emptyset,
$$
for a small enough $\sigma >0$. This reason for this essentially the same.
The solution $\app{w}{ijk}_N$ is such that 
$$
\supp{\app{w}{ijk}_N} \Subset J^+(p_0),
$$
when $\sigma > 0$ is small enough,  since the solution $\app{w}{ijk}_N$ is due to sources located 
at the intersection points $p_1,\dots,p_N$ of the geodesics $\app{\gamma}{j}$, $j=1,2,3,4$. 
But the intersection points are such that 
$$
p_1,\dots,p_N \notin J_\delta^-(p_0) \cap J^+(q_1) \setminus B(q_1,r),
$$
which again follows directly from 
Lemmas \ref{lem_gamma_1_future}, \ref{lem_gamma_1_no_intersections} and \ref{lem_gamma_2_3_4_no_intersections}.
The solution $\app{v}{0}$ can on the other hand be thought as propagating backwards in time from a source located at $q_0$,
so that $\supp(\app{v}{0}) \subset J^-_\delta(q_0)$, when the source is supported in a small enough neighbourhood of $q_0$.
It follows now that
$$
\supp(\app{w}{ijk}_N) \cap \supp(\app{v}{0}) = \emptyset,
$$
for a small enough $\sigma >0$. 
We thus see that \eqref{eq_term1} holds.
A similar analysis shows that
\begin{equation} \label{eq_term2}
\begin{aligned}
\int_{\Omega_T} 2q_2 \app{w}{ij} \app{w}{kl} \app{v}{0}\, dV 
=
\int_{\Omega_T} 2q_2 \app{\hat w}{ij}   \app{\hat w}{kl} \app{\hat v}{0}\, dV  + \mathcal O (h^5).
\end{aligned}
\end{equation}
By \eqref{eq_c2c3_def}, \eqref{eq_theta2_theta3} and \eqref{eq_beta_2},
we see that the form of the lowest order terms are 
\begin{align*}
\app{\hat v}{j} &= e^{\tfrac{i}{h} \app{\Upsilon}{j}} \app{a}{j}_0  + \mathcal O(h), \\
\app{\hat w}{jk} &= e^{ \tfrac{i}{h} \app{\Upsilon}{jk}} q_2\Upsilon^\sharp_{jk} \app{a}{j}_0 \app{a}{k}_0 h^2  + \mathcal O(h^3), \\
\app{\hat w}{jkl} &= e^{ \tfrac{i}{h} \app{\Upsilon}{jkl}} 
q^2_2 \Upsilon^\sharp_{jkl} 
[ \Upsilon^\sharp_{jl} + \Upsilon^\sharp_{jk} + \Upsilon^\sharp_{kl}] \app{a}{j}_0 \app{a}{k}_0 \app{a}{l}_0 
 h^4  + \mathcal O(h^5).
\end{align*}
We now see from \eqref{eq_term1} and \eqref{eq_term2} that \eqref{eq_J0} holds.

\medskip
\noindent
We are now ready to prove that $q_2^3 = \tilde q^3_2$. The proof
is similar to that of Propositions \ref{prop_q2} in section \ref{sec_q_2}.

\begin{proposition} \label{prop_q32}
Assume that $S = \tilde S$. Then $q^3_2 = \tilde q^3_2$.
\end{proposition}

\begin{proof}
Since  $S = \tilde S$ implies that $\p^3_{\eps_1 \eps_2 \eps_3 \eps_4} (S - \tilde S)|_{\eps=0}=0$,
we have that \eqref{eq_J_minus_J} holds. Looking at the lowest order terms in \eqref{eq_J_minus_J}
gives us, using \eqref{eq_J0}, that 
\begin{align*}
0=J_0 - \tilde J_0 = 
\int_{\Omega_T} (q^3_2 - \tilde q_2^3 )  e^{ \frac{i}{h} (\app{\Upsilon}{0} +\app{\Upsilon}{1} + \app{\Upsilon}{2} + \app{\Upsilon}{3}+ \app{\Upsilon}{4}) } 
\mathbf \Upsilon^\sharp
\app{\bar a}{1}_{0} \ap{2}_{0} \ap{3}_{0} \ap{4}_{0} \app{\bar a}{0}_{0} \,dx.
\end{align*}
We can now apply Proposition \ref{prop_stationary_phase} as in the proof of Proposition
\ref{prop_q2}. 
Denoting the phase function by 
$$
\psi:= \app{\Upsilon}{0} +\app{\Upsilon}{1} + \app{\Upsilon}{2} + \app{\Upsilon}{3}+ \app{\Upsilon}{4}
= \tilde \kappa_0\overline{\app{\Phi}{0}} +\tilde  \kappa_1 \overline{\app{\Phi}{1}} + \tilde  \kappa_2 \app{\Phi}{2} 
+ \tilde  \kappa_3 \app{\Phi}{3} + \tilde  \kappa_4 \app{\Phi}{4}. 
$$
We need to check the conditions of the stationary phase Lemma \ref{prop_stationary_phase}. 
The phase functions $\app{\Phi}{j}$ vanish on the 
corresponding geodesics because of \eqref{eq_Phi}.  Moreover
$$
\nabla \psi( p_0) = \tilde \kappa_0\app{\xi}{0} +\tilde  \kappa_1\app{\xi}{1} 
+ \tilde  \kappa_2 \app{\xi}{2} +\tilde  \kappa_3 \app{\xi}{3} +\tilde  \kappa_4 \app{\xi}{4} = 0,
$$
because of \eqref{eq_lin_dependent_2}.  Thus $ \nabla \psi( p_0) = 0$.
For $x \in X$, we need there to hold
\begin{align*}
\Im \psi= \Im ( \tilde \kappa_0 \overline{ \app{\Phi}{0}} + \tilde \kappa_1 \overline{\app{\Phi}{1}} +  \tilde \kappa_2\app{\Phi}{2} 
+  \tilde \kappa_3\app{\Phi}{3} +  \tilde \kappa_4\app{\Phi}{4})  \geq 0.
\end{align*}
This is true, because of equation \eqref{eq_Phi_cond} and the choices
$\tilde \kappa_0, \tilde\kappa_1 <0$, $\tilde \kappa_2,\tilde \kappa_3,\tilde \kappa_4 > 0$.
For the Hessian we have by linearity,  that 
$$
\Im H\psi(p_0) = \Im ( \tilde\kappa_0 H \overline{\app{\Phi}{0}} + \tilde\kappa_1 H\overline{ \app{\Phi}{1}} 
+  \tilde\kappa_2 H \app{\Phi}{2} +  \tilde\kappa_3 H\app{\Phi}{3}+  \tilde\kappa_4 H\app{\Phi}{4})(p_0)  > 0,
$$
which holds because of Lemma \ref{lem_Riccati}, and since $\tilde \kappa_0,\tilde \kappa_1 <0$, 
$\tilde \kappa_2,\tilde \kappa_3,\tilde \kappa_4 > 0$. 
Finally note that $p_0$ is a local minima $\Im \psi$, and since $H \Im\psi(p_0) > 0$,
so that
\begin{align*}
\Im \nabla \psi= \Im ( \tilde \kappa_0 \nabla \overline{\app{\Phi}{0}} + \tilde \kappa_1 \nabla \overline{\app{\Phi}{1}}
+ \tilde \kappa_2 \nabla \app{\Phi}{2} +\tilde  \kappa_3 \nabla \app{\Phi}{3}  + \tilde  \kappa_4 \nabla \app{\Phi}{4})   
\neq 0,
\quad \text{ in } X \setminus \{  p_0 \},
\end{align*}
and hence $\nabla \psi \neq 0$ in $X \setminus \{  p_0 \}$.
We thus see that $\psi$ satisfies the conditions of Proposition \ref{prop_stationary_phase}.
Using Proposition \ref{prop_stationary_phase} we  now have  that
$$
0 = h^{n/2}(J_0 - \tilde J_0) \to C \mathbf{\Upsilon}^\sharp (q^3_2 - \tilde q^3_2)(p_0), 
$$
as $h \to 0$,
with some non-zero constant $C \in \C$. Lemma \ref{lem_Upsilon_asymptotics} implies
$$
\mathbf{\Upsilon}^\sharp(p_0) \neq 0
$$
and it follows that  $q^3_2(p_0) = \tilde q^3_2(p_0)$.
\end{proof}

\subsection{Obtaining information about $q_1$ and $F$} \label{sec_q_1}

Here we prove the that $q_1$ and $\tilde q_1$ are related as Theorem \ref{thm_thm1} claims.
In this section we revert back to the solutions $\app{v}{i}$ and $\app{w}{ij}$ that are  constructed subsection
\ref{sec_GBs_and_geodesics}. The solutions $\app{v}{i}$ are thus given by \eqref{eq_GB_reparam_2} and
$\app{w}{ij}$ by \eqref{eq_wij_3}.

First we will show that $\ap{0}_1 (p_0) = \app{\tilde a}{0}_1(p_0)$, where these are the first order parts of the amplitude
expansion \eqref{eq_GB_amplitude} of the solutions $ \app{v}{0}$, and $ \app{\tilde v}{0}$. We can then use 
\eqref{eq_a_integrals} to obtain  information about $q_1$. At the end of the section
we show how to obtain information about $F$ and $\tilde F$.

\medskip
\noindent
In order to show that $\ap{0}_1 (p_0) = \app{\tilde a}{0}_1(p_0)$ we use the first order terms in $h$
of the integral identity \eqref{eq_I_0}, i.e. $I_1$ and $R_1$. From 
\eqref{eq_iid_zero}, \eqref{eq_I_0}  and \eqref{eq_I0_I1_R1}, we obtain the identity
$$
I_1  - \tilde I_1 = \tilde R_1 - R_1.
$$
From \eqref{eq_R1} it follows that the geometry of $N$ and $q_2$ determine $R_1$. Since we know that $q_2 = \tilde q_2$ by 
Proposition \ref{prop_q32}, we have that $R_1 - \tilde R_1 = 0$, and hence
\begin{equation} \label{eq_I1_tildeI1}
\begin{aligned}
I_1  - \tilde I_1 = 0.
\end{aligned}
\end{equation}

\begin{lemma} \label{lem_a1}
Let $ p_0$ be the intersection of $ \app{\gamma}{0}$ and $ \app{\gamma}{1}$ as in subsection \ref{sec_geom_setup}.
Then  $S = \tilde S$ implies that  $\ap{0}_1 ( p_0) = \app{\tilde a}{0}_1( p_0)$. %
\end{lemma}

\begin{proof}
Let $\psi$ be the function
$$
\psi:= \app{\Psi}{0} +\app{\Psi}{1} + \app{\Psi}{2} + \app{\Psi}{3}, %
$$
as in the proof of Proposition \ref{prop_q2}. Expanding \eqref{eq_I1_tildeI1} using \eqref{eq_I1} gives
\begin{small}
\begin{align*}
0 =
\int q^2_2 e^{ i \psi }{h}
(\Psi^\sharp_{23} +\Psi^\sharp_{13} +\Psi^\sharp_{12}) 
\big(
& \tfrac{1}{|\kappa_0|} (\app{\bar a}{0}_1 - \app{ \bar{\tilde a} }{0}_1) \app{\bar a}{1}_0 \ap{2}_0 \ap{3}_0
+ \tfrac{1}{|\kappa_1|} \app{\bar a}{0}_0 (\app{\bar a}{1}_1 -\app{ \bar{\tilde a}}{1}_1) \ap{2}_0 \ap{3}_0 \\
+&\tfrac{1}{|\kappa_2|} \app{\bar a}{0}_0  \app{\bar a}{1}_0 (\ap{2}_1 - \app{\tilde a}{2}_1) \ap{3}_0
+ \tfrac{1}{|\kappa_3|} \app{\bar a}{0}_0  \app{\bar a}{1}_0  \ap{2}_0 (\ap{3}_1 - \app{\tilde a}{3}_1)
\big)
\,dx. 
\end{align*}
\end{small}
Notice that the amplitude term $\app{a}{j}_1$ involves a power of $ \frac{ h }{ |\kappa_j| }$
because of the rescaling, so that we get additional factors of $|\kappa_j|$ in the above expression.
We evaluate the above integral using Proposition \ref{prop_stationary_phase}. The
phase function $\psi$ is the same as in the proof of Proposition \ref{prop_q2},
and satisfies thus the requirements of the stationary phase result of Proposition \ref{prop_stationary_phase}.
We can thus apply the method of stationary phase to the above integral, 
as was done in Proposition \ref{prop_q2}, by taking the limit $h \to 0$. 
This gives that
\begin{align*}
\big[ 
c_1 (\app{\bar a}{1}_1 - \app{\bar{\tilde a}}{1}_1) 
+
c_2 
(\ap{2}_1 - \app{\tilde a}{2}_1)
+
c_3 (\ap{3}_1 - \app{\tilde a}{3}_1) 
+ 
c_0(\app{\bar a}{0}_1 - \app{\bar{\tilde a}}{0}_1) 
\big]
(p_0) = 0,
\end{align*}
where the $c_j$ are 
\begin{align*}
c_1 &= \tfrac{1}{|\kappa_1|} \big[(\Psi^\sharp_{12} + \Psi^\sharp_{23} + \Psi^\sharp_{13}) \app{\bar a}{0}_0 \ap{2}_0 \ap{3}_0\big](p_0),  \\
c_2 &= \tfrac{1}{|\kappa_2|} \big[(\Psi^\sharp_{12} + \Psi^\sharp_{23} + \Psi^\sharp_{13}) \app{\bar a}{0}_0 \app{\bar a}{1}_0 \ap{3}_0\big](p_0),   \\
c_3 &= \tfrac{1}{|\kappa_3|} \big[(\Psi^\sharp_{12} + \Psi^\sharp_{23} + \Psi^\sharp_{13}) \app{\bar a}{0}_0 \app{\bar a}{1}_0 \ap{2}_0\big](p_0),   \\
c_0 &= \tfrac{1}{|\kappa_0|} \big[(\Psi^\sharp_{12} + \Psi^\sharp_{23} + \Psi^\sharp_{13}) \app{\bar a}{1}_0 \ap{2}_0 \ap{3}_0\big](p_0).
\end{align*}
Next we work out the asymptotics of the coefficients $c_j$,  when $\sigma \to 0$.
Note that the $\ap{j}_0$ are independent of $\sigma$. Using Lemmas \ref{lem_xis_lin_indep} and 
\ref{lem_Phis_nonzero}, we have that
\begin{equation} \label{eq_c0}
\begin{aligned}
c_0 \sim C' \tfrac{1}{|\kappa_0|} ( \Psi^\sharp_{12} + \Psi^\sharp_{23} + \Psi^\sharp_{13})  \sim C \sigma^{-6}, %
\end{aligned}
\end{equation}
where $C,C' \neq 0$. Furthermore
\begin{equation} \label{eq_c1}
\begin{aligned}
c_1 \sim  C'\tfrac{1}{|\kappa_1|} ( \Psi^\sharp_{13} + \Psi^\sharp_{23} + \Psi^\sharp_{12}) \sim C \sigma^{-4},
\end{aligned}
\end{equation}
where $C,C' \neq 0$. 
Similarly 
\begin{equation} \label{eq_c2c3}
\begin{aligned}
c_2 \sim C \sigma^{-4}, \qquad  
c_3 \sim C' \sigma^{-4}, \qquad  
\end{aligned}
\end{equation}
where $C,C' \neq 0$. 
It follows from \eqref{eq_c0}, \eqref{eq_c1} and \eqref{eq_c2c3}
that we can make $c_0$ arbitrarily big in comparison to  $c_1, c_2$ and $c_3$, by choosing
a small $\sigma > 0$. We can therefore deduce that 
$$
\ap{0}_1( p_0)  =  \app{\tilde a}{0}_1 ( p_0),
$$
which is what we wanted to show.
\end{proof}

\noindent
We will now utilize the fact that  $\ap{0}_1  =  \app{\tilde a}{0}_1$ to obtain information
about the relation between $q_1$ and $\tilde q_1$.

\begin{proposition} \label{prop_q1}
Assume that $S = \tilde S$. Then 
$
q_1 = \tilde q_1 + 2 \tilde q_2 \varphi, %
$
where $\varphi := \tilde u_0 - u_0 \in C^\infty(N)$, and $u_0$ and $\tilde u_0$ are given by \eqref{eq_u0}. 
\end{proposition}

\begin{proof}
Let us abbreviate $ \app{a}{0}_1$ with $a_1$ for clarity. 
Recall firstly that we can vary the intersection point $p_0 = (\hat s_0,0)$ on the geodesic
$\app{\gamma}{0}$, along an interval $ (p_0 -\eps , p_0 + \eps)$, by Lemma \ref{lem_variation_of_geodesics}.
Furthermore  we have by the equations in \eqref{eq_a_integrals} that 
\begin{small}
\begin{equation*} 
\begin{aligned}
a_{1,0}(\hat s_0 ,0) &= a^\sharp_{1,0} ( \hat s_0 ,0) + a^\flat_{1,0}(\hat s_0,0) 
= \frac{ 1 }{ 2 (\det Y(\hat s_0))^{1/2} } \Big[ \int_{s_0}^{\hat s_0} \Box_g a_0 (t,0) (\det Y(t))^{1/2} + Q(t,0) \,dt \Big],
\end{aligned}
\end{equation*}
\end{small}
in the Fermi coordinates $(s,z')$, and where $Q = q_1 + 2q_2u_0$ in accordance with \eqref{eq_vi}. 
Then $a^\flat_{1,0}(\hat s_0,0)$ consists of the integral of $Q$, see \eqref{eq_a_integrals}.
To obtain the  value of $Q$ at $\hat s_0$, we differentiate, which gives
\begin{align*}
Q( \hat s_0 ,0) =  \p_{\hat s_0} \big (( a_{1,0}(\hat s_0 ,0) -a^\sharp_{1,0}(\hat s_0 ,0) ) 2 (\det Y(\hat s_0))^{1/2}\big).  
\end{align*}
Here $a^\sharp_{1,0}$ does not depend on the potentials and is determined by the geometry,
and thus $a^\sharp_{1,0} (\hat s_0,0) = \tilde a^\sharp_{1,0} ( \hat s_0,0)$ . 
We also know by Lemma \ref{lem_a1} that $a_{1,0} ( \hat s_0,0)= \tilde a_{1,0}( \hat s_0,0)$,
since $a_1(p_0) = \tilde a_1(p_0)$. 
We can therefore conclude that $Q(\hat s_0,0) = \tilde Q(\hat s_0,0) $ or that 
$$
(q_1 + 2q_2 u_0)( p_0) = (\tilde q_1 + 2 \tilde q_2 \tilde u_0)( p_0). 
$$
By varying the point $ p_0$ and setting $\varphi := \tilde u_0 - u_0$, we see that the claim holds.
\end{proof}

\noindent
The final step in proving Theorem \ref{thm_thm1} is to prove the part concerning 
the source term $F$. We do this in the following proposition.

\begin{proposition}\label{prop_F}
Assume that $S = \tilde S$. Then we have that $ F = \tilde F - \Box_g \varphi - \tilde q_1 \varphi -  \tilde q_2\varphi^2$,
where $\varphi$ is as in Proposition \ref{prop_q1}.
\end{proposition}

\begin{proof}
From Proposition \ref{prop_q32} we know that
$$
q_2 = \tilde q_2. 
$$
And from Proposition \ref{prop_q1} we also know that
\begin{equation} \label{eq_q1q2u}
\begin{aligned}
 q_1 + 2 q_2 u_0 =  \tilde q_1  + 2 q_2 \tilde u_0 
\quad \Rightarrow \quad
 \tilde q_1 - q_1 = -2q_2 \varphi,
\end{aligned}
\end{equation}
where $\varphi:= \tilde u_0 - u_0$. 
To prove the formula in the claim note that by \eqref{eq_q1q2u}, we have that 
\begin{align*}
\tilde F &= (\Box_g + \tilde q_1) (u_0 +\varphi) + \tilde q_2 (u_0 +\varphi)^2 \\
&=
\Box_g (u_0 +\varphi) + \tilde q_1 (u_0 +\varphi) +  q_2 (u_0^2 +\varphi^2 + 2\varphi u_0) \\
&=
\Box_g (u_0 +\varphi) + \tilde q_1 (u_0 +\varphi) +  q_2 u_0^2 + q_2\varphi^2 + (q_1 - \tilde q_1) u_0 \\
&= 
F + \Box_g \varphi+ \tilde q_1 \varphi+  q_2\varphi^2,
\end{align*}
which shows that the claim holds. 
\end{proof}

\section*{Acknowledgements}
M. Lassas was supported by a AdG project 101097198 of the European Research Council, Centre of Excellence of Research Council of Finland, and the
FAME flagship of the Research Council of Finland (grant 359186).
T.~L. was partly supported by the Academy of Finland (Centre of Excellence in Inverse Modelling and Imaging and FAME Flagship, grant numbers 312121 and 359208).
V.~P. and T.~T. were supported by the Research Council of Finland (Flagship of Advanced Mathematics for Sensing, Imaging and Modelling grant 359186) and by the Emil Aaltonen Foundation.

\appendix
\section{Explicit formulas}\label{sec_explicit}

In this section we  we provide some further details on the computations involved in proving Lemma
\ref{lem_xis_lin_indep} and Lemma \ref{lem_xis_lin_indep_2}.

\medskip
\noindent
We can obtain explicit formulas for the $\kappa_j$ in Lemma \ref{lem_xis_lin_indep},
by performing Gaussian eliminations.  For Lemma \ref{lem_xis_lin_indep} this yields the values 
\begin{small}
\begin{align*}
\kappa_1 = \frac{\sigma^2 (\sqrt{1-\sigma^2}- s_0 )}{\sqrt{1-\sigma^2}-1}, \,\,
\kappa_2 = \frac{ \sigma^2 (s_0 - 1)}{2\sqrt{1-\sigma^2}-1} \pm \frac{\sigma \sqrt{1 - s_0^2}}{2},\,\,
\kappa_3 = \frac{\sigma^2 (s_0-1)}{2\sqrt{1-\sigma^2}-1} \mp \frac{\sigma\sqrt{1-s_0^2}}{2}, 
\end{align*}
\end{small}\\
where the choice of sign corresponds to the choice of sign in the definition of $\app{\xi}{0}$.
To compute the asymptotics of $\kappa_1$ in Lemma \ref{lem_xis_lin_indep}, we write
\begin{align*}
\kappa_1 = - (\sqrt{1-\sigma^2}- s_0 )(\sqrt{1-\sigma^2}+1)
&\sim -(1- \tfrac{\sigma^2}{2}-s_0)(1- \tfrac{\sigma^2}{2}+1) \\
&\sim -2(1-s_0) + \tfrac{\sigma^2}{2} (3- s_0),
\end{align*}
where we used the Taylor expansion  $\sqrt{1- \sigma^2} = 1 - \tfrac{\sigma^2}{2} + O(\sigma^4)$ for small
$\sigma$.
The asymptotics for $\kappa_2$ and $\kappa_3$  can easily be obtained using similar computations.

\medskip
\noindent
We can similarly  obtain explicit formulas for the $\tilde \kappa_j$ in Lemma \ref{lem_xis_lin_indep_2},
by performing Gaussian eliminations. 
For Lemma \ref{lem_xis_lin_indep_2} this yields the value
\begin{align*}
\tilde \kappa_1 &=  \frac{\sigma^2 \left(-\sqrt{1-\sigma^4}+\sqrt{1-\sigma^2} - s_0 + 1\right)}{\sqrt{1-\sigma^2}-1}.
\end{align*}
For $\tilde \kappa_2$ we get 
\begin{align*}
\tilde \kappa_2 &= 
\frac{\sigma^2 \left(-\sqrt{1-\sigma^2}+\sqrt{\sigma^2+1} +  s_0 -1 \right)}{2\sqrt{1-\sigma^2}-2} \pm \frac{\sigma}{2} \sqrt{1-s_0^2},
\end{align*}
and for $\tilde \kappa_3$ we get 
\begin{align*}
\tilde \kappa_3 &= 
\frac{\sigma^2 \left(\sqrt{1-\sigma^2}+\sqrt{\sigma^2+1}+ s_0-3\right)}{2\sqrt{1-\sigma^2}-2} \mp \frac{\sigma}{2} \sqrt{1-s_0^2},
\end{align*}
and where the choice of sign corresponds to the choice of sign in the definition of $\app{\xi}{0}$.
One can check that these choices together with the definitions of the $\app{\xi}{i}$ solve equation  
\eqref{eq_lin_dependent_2}. It is moreover
a straight forward to evaluate the asymptotic formulas in Lemmas \ref{lem_xis_lin_indep_2}
and \ref{lem_xis_lin_indep}, by expanding these suitable using Taylor series in a neighbourhood
of $\sigma = 0$.

\section{Well-posedness for the forward problem} \label{sec_wellposed}

Given a time function $t$ for $N$, let $\Sigma_{t_0}$ denote the Cauchy surface $\{x\in N : t(x)=t_0\}.$ 
In this section all functions are assumed to be spatially compact, meaning that given a function $f:N\to \C$ and $t_0\in \R$, the set $\supp(f\big|_{\Sigma_{t_0}})\subset \Sigma_{t_0}$ is relatively compact.

We define
\[
E^s :=E^s(N) \vcentcolon =
\bigcap_{k=0}^s C^k(t(N), H^{s-k}(\Sigma_\bullet)).
\]
Let us also denote 
\[
N_{T}=t^{-1}([0,T]).
\]
If $U\subset N$ is  open, we denote
\[
E^s(U)=\{u|_U : \ u \in E^s\}
\]
and
\[
E_c^s(U)=\{u\in E^s(U) : \supp(u) \text{ compact}\}.
\]

Suppose $F\in E^s$, $\varphi_0\in H^{s+1}(\Sigma_0)$, and $\varphi_1\in H^{s}(\Sigma_0)$, and assume that there exists a unique solution $u_0\in E^{s+1}$ to the equation
\begin{equation}\label{eq: nonlinear eq for u_0}
\begin{cases}
\square_g u + q_1u + qu_2^p = F, & x\in N_T,\\
u(0)=\varphi_0,\ \p_t u(0) = \varphi_1
\end{cases}
\end{equation}
in $\big(J^+(\supp(\varphi_0))\cup J^+(\supp(\varphi_1)) \cup J^+(\supp(F+f))\big)\cap N_T$.

\begin{proposition}[{\cite[Corollary 12]{Bar15}}]\label{prop_energy_est_lin}
Let $[0,T]\subset t(N)$ and let $K\subset N$ be compact. Let $k\in\N$ and $A\in C^\infty(N)$. Then there exists a constant $C>0$ such that
\begin{equation}\label{eq: solution estimate}
\Vert u \Vert_{H^k(N_T)} 
\leq 
C \left(
\Vert u\big|_{t=0} \Vert_{H^k(\Sigma_0)} + 
\Vert \partial_t u\big|_{t=0} \Vert_{H^{k-1}(\Sigma_0)} + 
\Vert (\square_g + A)u \Vert_{H^{k-1}(N_T)}
\right)
\end{equation}
holds for all $u\in C^\infty(N)$ with $\supp(u)\subset J^+(K)$.
\end{proposition}

\begin{proposition}\label{prop: well posedness of nonlinear waves with sources}
Let $s>n/2$ and $p\geq 1$ be an integer. Suppose there exists a unique solution $u_0\in E^{s+1}$ to \eqref{eq: nonlinear eq for u_0}, and let $f\in E^s$. Then there is $\eps>0$ such that when $\Vert f \Vert_{E^s}< \eps$, there is a $C>0$ such that there exists a unique solution $u\in E^{s+1}$ to the equation
\begin{equation}\label{eq: well-posedness for nonlinear eq with sources}
\begin{cases}
\square_g u + q_1u + q_2u^p = F+f, & x\in N_T,\\
u(0)=\varphi_0,\ \p_t u(0) = \varphi_1
\end{cases}
\end{equation}
with $\supp(u)\subset \big(J^+(\supp(\varphi_0))\cup J^+(\supp(\varphi_1)) \cup J^+(\supp(F+f))\big)\cap N_T$. Moreover, the solution $u$ satisfies
\[
\Vert u-u_0\Vert_{E^{s+1}} \leq C \Vert f\Vert_{E^s}.
\]
\end{proposition}
\begin{proof}
We consider $p\geq 2$, because if $p=1$ then equation \eqref{eq: well-posedness for nonlinear eq with sources} is linear and the claim reduces to the results of~\cite[Theorem~13 and Corollary~17]{Bar15}, and holds with any $\eps>0$. Without explicitly mentioning it, we will often use the fact that $E^s$ forms an algebra, when $s>n/2$: if $v,w\in E^s$ then the pointwise product $vw\in E^s$. For this fact we refer to~\cite[Appendix III, Definition 3.4 (2) and Definition 3.5]{CB08} and to~\cite[Sec.~2, p.1077]{lassas2025stability}, where a brief proof is given.
Let us define the class $X^{s+1}\subset E^{s+1}$ as
\[
X^{s+1} \vcentcolon = \{ u \in E^{s+1} : \square_g u \in E^s\}
\]
equipped with the norm
\[
\Vert u \Vert_{X^{s+1}} \vcentcolon = \Vert u \Vert_{E^{s+1}} + \Vert \square_g u \Vert_{E^s}.
\]
We verify that $X^{s+1}$ is a Banach space and to that end, let $(u_k)_{k=1}^\infty\subset X^{s+1}$ be a Cauchy sequence.
Since $E^s$ is a Banach space, the sequences $(u_k)_{k=1}^\infty$ and $(\square_g u_k)_{k=1}^\infty$ are Cauchy in $E^{s+1}$ and $E^s$, respectively, and have limits $u_k \to u\in E^{s+1}$ and $\square_g u_k \to G\in E^s$. 
Then
\begin{align*}
\Vert \square_g u - G \Vert_{E^{s-1}}
&=
\Vert \square_g u - \square_g u_k + \square_g u_k - G\Vert_{E^{s-1}}\\
&\leq
\Vert \square_g u - \square_g u_k\Vert_{E^{s-1}} +\Vert \square_g u_k - G\Vert_{E^{s-1}}
\\
&\leq
C_0'( \Vert u-u_k\Vert_{E^{s+1}} + \Vert \square_g u_k - G\Vert_{E^s}) \to 0
\end{align*}
as $k\to \infty$. Therefore, $\square_g u = G $ a.e. implying $\square_g u = G \in E^s$, whence $X^{s+1}$ is a Banach space.

Next, let%
\[
A_1 = H^s(N),\quad
A_2 = X^{s+1},\quad
A_3 = H^{s+1}(\Sigma_0) \times H^s(\Sigma_0) \times E^{s}.
\]
and define $\Phi : A_1\times A_2 \to A_3$ by
\[
\Phi(f,u) = (u(0)-\varphi_0, \p_t u(0)-\varphi_1, \square_g u + q_1u +  q_2u^p - F-f)
\]
for a fixed $F\in E^s(N_T)$. The map is well-defined, since if $u\in X^{s+1}$, then $u(0)\in H^{s+1}(\Sigma_0)$, $\p_t u(0)\in H^s(\Sigma_0)$, and in the last term $\square_g u\in E^s$.
We next show that $\Phi$ is a $C^1$-map in the Fr\'echet sense. The other terms are linear, hence smooth, except the part $qu^p$.  For this term, we need to only show that $u\mapsto q_2u^p$ is $C^1$-map $X^{s+1}\to E^s$ in the Fr\'echet sense with the differential $pq_2u^{p-1}h$, $h\in E^{s+1}$, when $\Vert h\Vert_{E^{s+1}}\leq 1$. Now
\begin{align*}
\Vert q_2(u+h)^p - q_2u^p - pq_2u^{p-1}h \Vert_{E^{s}}
&\leq
C_0\Vert q_2 \Vert_{C^s(N)}
\left\Vert
\sum_{k=0}^p \binom{p}{k} u^{p-k}h^k - u^p - pu^{p-1}h 
\right\Vert_{E^{s}}\\
&=
C_0\Vert q_2 \Vert_{C^s(N)}
\left\Vert
\sum_{k=2}^p \binom{p}{k} u^{p-k}h^k
\right\Vert_{E^{s}}\\
&\leq
\Vert q_2 \Vert_{C^s(N)} O(\Vert h\Vert_{E^{s}}^2).
\end{align*}
This shows the differentiability. Continuity of the differential is proven similarly.
Therefore, $\Phi$ is $C^1$-map.
Now, it is clear that $\Phi(0,u_0)=(0,0,0)$. The first linearization of $\Phi$ in the second variable at $(0,u_0)$ is the bounded linear operator
\begin{align*}
&D_u \Phi\big|_{(0,u_0)} : X^{s+1}\to A_3\\
&D_u \Phi\big|_{(0,u_0)}(v) = (v(0),\p_t v(0),\square_gv + q_1v + pq_2u_0^{p-1} v ).
\end{align*}
We will next show that $D_u \Phi\big|_{(0,u_0)}$ is a homeomorphism. For this, consider the linear equation
\[
\begin{cases}
\square_gv + q_1v + pq_2u_0^{p-1} v = G & x\in N_T\\
v(0)=\phi_0,\quad \p_t v(0) = \phi_1.
\end{cases}
\]
The above linear Cauchy problem has a unique solution $v\in E^{s+1}$ depending
continuously on $(\phi_0,\phi_1,G)\in A_3$, see~\cite[Theorem~13 and Corollary~17]{Bar15}. Here $\square_g v = G-q_1v-pq_2u_0^{p-1}v \in E^s$, so that
$v\in X^{s+1}$. Thus, $D_u\Phi\big|_{(0,u_0)}$ is invertible. Moreover,
\begin{align*}
\Vert (D_u\Phi\big|_{(0,u_0)})^{-1}(\phi_0,\phi_1,G)\Vert_{X^{s+1}}
&= \Vert v \Vert_{E^{s+1}} + \Vert \square_g v \Vert_{E^s}
\\
& \leq
C \Big( \Vert G \Vert_{E^s} 
+ \Vert \phi_0 \Vert_{H^{s+1}(\Sigma_0)}
+ \Vert \phi_1 \Vert_{H^{s}(\Sigma_0)} 
\Big)\\
&\qquad
+ \Vert G-q_1v-pq_2u_0^{p-1}v \Vert_{E^s}
\\
&\leq C \Big( \Vert G \Vert_{E^s} 
+ \Vert \phi_0 \Vert_{H^{s+1}(\Sigma_0)}
+ \Vert \phi_1 \Vert_{H^{s}(\Sigma_0)} \Big)
\\
&\qquad
+ \Vert G \Vert_{E^s} + C' \Vert v \Vert_{E^{s+1}}
\\
& \leq
C'' \Big( \Vert G \Vert_{E^s} 
+ \Vert \phi_0 \Vert_{H^{s+1}(\Sigma_0)}
+ \Vert \phi_1 \Vert_{H^{s}(\Sigma_0)} \Big),
\end{align*}
where we liberally used the linear energy estimate for the norm of $v\in E^{s+1}$. Therefore, $D_u\Phi\big|_{(0,u_0)}: X^{s+1}\to A_3$ is a homeomorphism.

Now we are ready to apply the implicit function theorem on Banach spaces \cite[Theorem 10.6 and Remark 10.5]{RR06}. There exists $\eps>0$ and an open ball $B_\eps\vcentcolon=\{v\in E^s : \Vert v \Vert_{E^s}<\eps\}\subset E^s$, and a $C^1$-map $S : B_\eps \to A_2$, such that
\[
\Phi(f,S(f))=(0,0,0).
\]
Here $S(f)$ satisfies the nonlinear wave equation \eqref{eq: well-posedness for nonlinear eq with sources}.
Finally, since $S(0)=u_0$, there holds
\begin{align*}
\Vert S (f) - u_0\Vert_{E^{s+1}}
& \leq
\Vert S (f) - u_0 \Vert_{X^{s+1}} \\
&= 
\Vert S (f) - S (0)\Vert_{X^{s+1}} 
\\
&\leq C_0\Vert f \Vert_{E^s},
\end{align*}
where in the last step we used the fact that $S$ is  Lipschitz.

To demonstrate the uniqueness of the solution in $E^{s+1}$, suppose that there were two solutions $u_1,u_2\in E^{s+1}$ to \eqref{eq: well-posedness for nonlinear eq with sources}. Now the linear equation
\begin{equation}\label{eq: uniqueness for nonlinears}
\begin{cases}
\square_g v  +  \Big(q_1 +  q_2 P(u_1,u_2)\Big) v = 0,& \text{in } N_T\\
v(0)=\p_t v(0)=0,
\end{cases}
\end{equation}
where $P(a,b) = \sum_{i=0}^{p-1} a^ib^{p-1-i}$ and potential $q_1 + q_2 P(u_1,u_2)\in C(N_T)$ (by the embedding $E^s\subset C(N_T)$, when $s>n/2$, see~\cite[Appendix III, Definition 3.5]{CB08}), has a unique solution $v\in E^1$. Certainly the solution $v=0$. But also the function $v\vcentcolon = u_1-u_2$ satisfies \eqref{eq: uniqueness for nonlinears}. Therefore, $u_1=u_2$ in $E^1$, implying that $u_1=u_2$ in $E^{s+1}$, and we are done.
\end{proof}

\section{Gauge symmetry}\label{sec:gauge_symmetry_appendix}

Here we  show that the source-to-solution map $S$ has the
gauge invariance indicated by equation \eqref{gauge_intro}.
Let $q_1,q_2\in C^\infty(N)$, $F\in E^s(N)$, and suppose $\varphi \in C^2(N)$.  
If $u$ solves 
\begin{equation}\label{eq:nonlinear_original}
\square_g u + q_1 u + q_2 u^2 = F+f,
\end{equation}
then $\tilde u := u+\varphi$ satisfies
\begin{equation}\label{eq:nonlinear_tilde}
\square_g\tilde u + \tilde q_1\tilde u + \tilde q_2\tilde u^2 = \tilde F + f,
\end{equation}
provided that $\tilde q_1,\tilde q_2,\tilde F$ obey
\begin{align}\label{eq:nonlinear_gauge}
\begin{cases}
F = \tilde F - (\square_g\varphi + \tilde q_1\varphi + \tilde q_2\varphi^2),\\
q_1 = \tilde q_1 + 2\tilde q_2\varphi,\\
q_2 = \tilde q_2.
\end{cases}
\end{align}
If moreover $\varphi|_U=0$ and $\varphi(0,\cdot)=\partial_t\varphi(0,\cdot)=0$, then by uniqueness of solutions to the nonlinear problem \eqref{eq:nonlinear_tilde}, we have $\tilde u=u+\varphi$ as the (unique) solution from the same Cauchy data of $u$. In particular,
\[
u|_U=\tilde u|_U 
\quad\text{and}\quad 
F|_U=\tilde F|_U.
\]
Therefore the coefficients related by \eqref{eq:nonlinear_gauge} generate the same source-to-solution map, i.e.\ $S=\tilde S$.

\bigskip

\noindent{\footnotesize E-mail addresses:\\
Matti Lassas: {matti.lassas@helsinki.fi}\\
Tony Liimatainen: {tony.liimatainen@helsinki.fi}\\
Valter Pohjola: {valter.pohjola@gmail.com}\\
Teemu Tyni: {teemu.tyni@oulu.fi}
}

\bibliography{References} 
        
\bibliographystyle{alpha}

\end{document}